\documentclass[a4paper, 11pt]{article}
\usepackage{latexsym,amsfonts, amsmath, amssymb, amsthm}
\usepackage[utf8]{inputenc}
\usepackage[final=true,protrusion=true,expansion=true]{microtype}
\usepackage[noadjust]{cite}
\usepackage{enumitem}
\usepackage{color}
\usepackage{mathtools}

\usepackage{booktabs}
\usepackage{multirow}

\usepackage{subcaption}
\usepackage[labelfont=bf]{caption}
\usepackage[affil-it]{authblk}

\usepackage{geometry}
\usepackage{hyperref}
\hypersetup{
	final,
    colorlinks=true,
    linkcolor=blue,
    filecolor=magenta,
    urlcolor=cyan,
}

\theoremstyle{plain}
\newtheorem{theorem}{Theorem}[section]
\newtheorem{proposition}[theorem]{Proposition}
\newtheorem{lemma}[theorem]{Lemma}

\theoremstyle{plain}
\newtheorem{definition}[theorem]{Definition}

\theoremstyle{definition}
\newtheorem{remark}[theorem]{Remark}
\newtheorem{example}[theorem]{Example}

\providecommand*{\Var}[1]{\operatorname{Var}\left({#1}\right)}   % distance
\providecommand{\Supp}{\operatorname{supp}}                            % support
\providecommand{\supp}{\Supp}
\providecommand{\argmin}{\operatorname*{\arg\min}}  % argument yielding inf
\providecommand{\Id}{\Op{Id}}                     % Identity operator

\providecommand{\CC}{{\cal C}}

\providecommand{\CE}{{\cal E}}
\providecommand{\CF}{{\cal F}}

\providecommand{\CM}{{\cal M}}
\providecommand{\CN}{{\cal N}}
\providecommand{\CO}{{\cal O}}
\providecommand{\CP}{{\cal P}}

\providecommand{\CT}{{\cal T}}

\providecommand{\CV}{{\cal V}}

\providecommand{\bbE}{\mathbb{E}}

\providecommand{\bbN}{\mathbb{N}}

\providecommand{\bbP}{\mathbb{P}}

\providecommand{\bbR}{\mathbb{R}}
\providecommand{\bbS}{\mathbb{S}}

\providecommand*{\abs}[1]{\left|{#1}\right|} % Double bar norm
\providecommand*{\N}[1]{\left\|{#1}\right\|} % Double bar norm
\providecommand*{\Nnormal}[1]{\|{#1}\|} % Double bar norm
\providecommand*{\Nbig}[1]{\big\|{#1}\big\|} % Double bar norm
\providecommand*{\textN}[1]{\|{#1}\|} % Double bar norm

\newcommand*{\SN}[1]{\left|{#1}\right|}      % Single bar norm
\providecommand*{\abs}[1]{\left|{#1}\right|} % absolute value, same as \SN
\newcommand*{\Op}[1]{\mathsf{#1}} % Operators
\usepackage{bbm}
\usepackage{arydshln}
\usepackage{subcaption}
\usepackage{mdframed}
\usepackage{xcolor}
\usepackage{ifdraft}
\usepackage{scalerel}
\usepackage{mathrsfs}  

\newif\ifrevised
\newcommand{\revised}[1]{%
	\ifrevised
		\color{purple} #1 \color{black} %DO NOT simply set color to black, use \revisedTwofalse
	\else
		#1%
	\fi}
\revisedfalse
\newif\ifrevisedTwo
\newcommand{\revisedTwo}[1]{%
	\ifrevisedTwo
		\color{purple} #1 \color{black} %DO NOT simply set color to black, use \revisedTwofalse
	\else
		#1%
	\fi}
\revisedTwofalse
\newif\ifrevisedThree

\revisedThreefalse
\newif\ifMFAappendix

\MFAappendixtrue
\newcommand{\overbar}[1]{\makebox[0pt]{$\phantom{#1}\mkern 1.5mu\overline{\mkern-1.5mu\phantom{#1}\mkern-1.5mu}\mkern 1.5mu$}#1}
\renewcommand{\underbar}[1]{\makebox[0pt]{$\phantom{#1}\mkern 1.5mu\underline{\mkern-1.5mu\phantom{#1}\mkern-1.5mu}\mkern 1.5mu$}#1}

\newcommand{\overbarscript}[1]{\mkern 1.5mu\overline{\mkern-1.5mu#1\mkern-1.5mu}\mkern 0mu}
\newcommand{\underbarscript}[1]{\mkern 1.5mu\underline{\mkern-1.5mu#1\mkern-1.5mu}\mkern 1.5mu}

\newcommand{\divergence}{\textrm{div}}

\newcommand{\globmin}{{v^*}}
\newcommand{\minobj}{\underbar \CE}

\newcommand{\cutoff}[1]{H\!\left({#1}\right)}
\newcommand{\cutoffnoarg}{H}

\newcommand{\omegaa}[0]{\omega_{\alpha}}

\newcommand{\conspoint}[1]{v_{\alpha}({#1})}
\newcommand{\conspointnoarg}{v_{\alpha}}
\newcommand{\textconspoint}[1]{v_{\alpha}({#1})}

\newcommand{\empmeasure}[1]{\widehat\rho_{#1}^N}

\newcommand{\monopmeasure}[1]{\overbar\rho_{#1}^N}

\newcommand{\indivmeasure}[0]{\varrho} %\widetilde\rho already used

\DeclareMathOperator*{\Law}{Law}

\makeatother

\title{\usefont{OT1}{bch}{b}{n}
	\huge Consensus-based optimization methods\\converge globally \\
}

\date{}

\author[1,2,3]{Massimo Fornasier\thanks{Email: \texttt{massimo.fornasier@ma.tum.de}}}
\affil[1]{Technical University of Munich, School of Computation, Information and Technology, Department of Mathematics, Munich, Germany}
\affil[2]{Munich Center for Machine Learning, Munich, Germany}
\affil[3]{Munich Data Science Institute, Munich, Germany}
\author[4,5]{Timo Klock\thanks{Email: \texttt{tmklock@googlemail.com}}}
\affil[4]{Simula Research Laboratory, Department of Numerical Analysis and Scientific Computing, Oslo, Norway}
\affil[5]{University of San Diego, California, Department of Mathematics, San Diego, USA}
\author[1,2]{Konstantin Riedl\thanks{Email: \texttt{konstantin.riedl@ma.tum.de}}}

\begin{document}
\maketitle
\begin{abstract}
\noindent
In this paper, we study consensus-based optimization (CBO), which is a multi-agent metaheuristic derivative-free optimization method that can globally minimize nonconvex nonsmooth functions and is amenable to theoretical analysis.
Based on an experimentally supported intuition that, on average, CBO  performs a gradient descent of the squared Euclidean distance to the global minimizer, we devise a novel technique for proving the convergence to the global minimizer in mean-field law for a rich class of objective functions.
The result unveils internal mechanisms of CBO that are responsible for the success of the method. In particular, we prove that CBO performs a convexification of \revisedTwo{a large} class of optimization problems as the number of optimizing agents goes to infinity.
Furthermore, we improve prior analyses by requiring \revisedTwo{mild} assumptions about the initialization of the method and by covering objectives that are merely locally Lipschitz continuous.
As a core component of this analysis, we establish a quantitative nonasymptotic Laplace principle, which may be of independent interest.
From the result of CBO convergence in mean-field law, it becomes apparent that the hardness of any global optimization problem is necessarily encoded in the rate of the mean-field approximation, for which we provide a novel probabilistic quantitative estimate.
The combination of these results allows to obtain probabilistic global convergence guarantees of the numerical CBO method. %with provable polynomial complexity.
\end{abstract}



% \tableofcontents
\section{Introduction} \label{sec:introduction}
A long-standing problem in applied mathematics is the global minimization of a potentially nonconvex nonsmooth cost function $\CE: \bbR^d\rightarrow \bbR$ and the search for an associated globally minimizing argument $\globmin$.
\revised{Throughout,} we assume the unique existence of \revised{the} minimizer $\globmin$ and \revised{denote its associated minimal value by}
\begin{align*}
	\minobj := \CE(\globmin) = \inf_{v\in \bbR^d}\CE(v).
\end{align*}
\revised{The objective~$\CE$ is supposed to be locally Lipschitz continuous and to satisfy a tractability condition of the form~$\N{v-\globmin}_2 \leq \left(\CE(v)-\minobj\right)^\nu\!/\eta$ in a neighborhood of $\globmin$, see Assumption~\ref{asm:icp} for the details.}
While computing $\minobj$ or $\globmin$ are in general NP-hard problems \revised{under such conditions,} several instances arising in real-world scenarios can, at least empirically, be solved within reasonable accuracy and moderate computational time.
\revised{In the present work we are concerned with the class of derivative-free optimization algorithms, i.e., methods that are based exclusively on the evaluation of the objective function~$\CE$.}
%A popular class of methods, which 
\revised{Amongst them and achieving} the state of the art on challenging problems such as the Traveling Salesman Problem, are so-called metaheuristics \cite{blum2003metaheuristics,aarts1989simulated,back1997handbook,reeves2010genetic,kennedy1995particle}.
Metaheuristics orchestrate an interaction between local improvement procedures and global strategies, and combine deterministic and random decisions, to create a process capable of escaping from local optima and performing a robust search of the solution space.
Examples include
Random Search~\cite{rastrigin1963convergence},
%Simplex Heuristics~\cite{nelder1965simplex},
Evolutionary Programming~\cite{fogel2006evolutionary},
the Metropolis-Hastings algorithm~\cite{hastings1970monte},
Genetic Algorithms~\cite{holland1992adaptation},
Particle Swarm Optimization \cite{kennedy1995particle},
%Ant Colony Optimization~\cite{dorigo2005ant},
and Simulated Annealing~\cite{aarts1989simulated}.
Despite their tremendous empirical success and widespread use in practice, many metaheuristics, due to their complexity, lack a proper mathematical foundation that could prove robust convergence to global minimizers under suitable assumptions.
Nevertheless, for some of them, such as Random Search or Simulated Annealing, there exist probabilistic guarantees for global convergence, see, e.g., \cite{solis1981minimization,holley1989asymptotics}.
While transferring some of the ideas of~\cite{solis1981minimization} to Particle Swarm Optimization allows to establish guaranteed convergence to global minima, the proof argument uses a computational time coinciding with the time necessary to examine every location in the search space~\cite{van2007analysis}.

Recently, the authors of \cite{pinnau2017consensus,carrillo2018analytical} have introduced consensus-based optimization (CBO) methods, which follow the guiding principles of metaheuristic algorithms, but are of much simpler nature and more amenable to theoretical analysis.
Inspired by consensus dynamics and opinion formation, CBO methods use a finite number of agents $V^1,\ldots,V^N$, which are formally stochastic processes, to explore the domain and to form a global consensus about the location of the minimizer $\globmin$ as time passes.
The dynamics of the agents $V^1,\ldots,V^N$ are governed by two competing terms.
A drift term drags \revised{each} agent towards an instantaneous consensus point, \revised{denoted by $\conspointnoarg$,} which is computed as a weighted average of \revised{all} agents' positions and serves as a momentaneous proxy for the global minimizer $\globmin$.
\revised{This term may be deactivated individually for an agent if its position improves upon the consensus point through modulating the drift by a function $H$ approximating the Heaviside function.}
%The resulting drift induces a contracting effect on the convex hull of the agents $V^1,\ldots,V^N$ and thus aids with consensus formation.
The second term is stochastic and randomly diffuses agents according to a scaled Brownian motion in~$\bbR^d$, featuring the exploration of the energy landscape of the cost~$\CE$.
%but counteracting the contracting effect of the drift term.
Ideally, as result of the described drift-diffusion mechanism,  the agents eventually achieve a near optimal global consensus, in the sense that the associated empirical measure
$\empmeasure{t} := \frac{1}{N} \sum_{i=1}^{N} \!\delta_{V_t^i}$
%\vspace{-0.2cm}
% \begin{align} \label{eq:empirical_measure}
% 	\empmeasure{t} := \frac{1}{N} \sum_{i=1}^{N} \delta_{V_t^i}
% \end{align}
converges to a Dirac delta $\delta_{\widetilde v}$ at some $\widetilde v \in \bbR^d$ close to $\globmin$.

Let us now provide a formal description of the method.
Given a time horizon $T > 0$ and a time discretization $t_0 = 0 < \Delta t < \!\cdots< K \Delta t = T$ of $[0,T]$, we denote the location of agent $i$ at time $k\Delta t$ by $V_{k\Delta t}^i$, $k=0,\ldots,K$. For user-specified parameters
 $\alpha,\lambda,\sigma > 0$, the time-discrete evolution of
the $i$-th agent is defined by the update rule
\begin{alignat}{2} \label{eq:dyn_micro_discrete}
	\phantom{V_{(k+1)\Delta t}^i - V_{k\Delta t}^i} &\begin{aligned}[c]
		\mathllap{V_{(k+1)\Delta t}^i - V_{k\Delta t}^i} = &- \Delta t\lambda\left( V_{k\Delta t}^i - \conspoint{\empmeasure{k\Delta t}}\right)\cutoff{\CE( V_{k\Delta t}^i)-\CE\!\left(\conspoint{\empmeasure{k\Delta t}}\right)}\\
		&+ \sigma \N{ V_{k\Delta t}^i-\conspoint{\empmeasure{k\Delta t}}}_2  B_{k\Delta t}^i,\\
	\end{aligned} \\
	\mathllap{V_0^i} &\sim \rho_0 \quad \text{for all } i =1,\ldots,N, \label{eq:dyn_micro_discrete_2}
\end{alignat}
where $(( B_{k\Delta t}^i)_{k=0,\ldots,K-1})_{i=1,\ldots,N}$ are independent, identically distributed Gaussian random vectors in $\bbR^d$ with zero mean and covariance matrix $\Delta t \Id_d$.
The system is complemented with independent initial data~$(V_0^i)_{i=1,\ldots,N}$, distributed according to a common initial law~$\rho_0$.
Equation~\eqref{eq:dyn_micro_discrete} originates from a simple Euler-Maruyama time discretization%\footnote{We note that this is just one possible time discretization scheme, see \cite{higham2001algorithmic,platen1999introduction} for alternatives.}
~\cite{higham2001algorithmic,platen1999introduction} of the system of stochastic differential equations~(SDEs)
%, expressed in It\^o's form as
\begin{align} \label{eq:dyn_micro}
dV_t^i &= -\lambda\left(V_t^i - \conspoint{\empmeasure{t}}\right)\cutoff{\CE(V_t^i)-\CE\!\left(\conspoint{\empmeasure{t}}\right)}dt
+ \sigma \N{V_t^i-\conspoint{\empmeasure{t}}}_2 dB_t^i,\\
%\nonumber
V_0^i &\sim \rho_0\quad \text{for all } i =1,\ldots,N,
\end{align}
where $((B_t^i)_{t\geq 0})_{i = 1,\ldots,N}$ are now independent standard Brownian motions in $\bbR^d$.
As mentioned in the informal description above, the updates in the evolutions~\eqref{eq:dyn_micro_discrete} and \eqref{eq:dyn_micro} consist of two terms, respectively.
The first term  is the drift towards the momentaneous consensus $\textconspoint{\empmeasure{t}}$, which is defined by
\begin{align} \label{eq:momentaneous_consensus}
	\conspoint{\empmeasure{t}} := \int v \frac{\omegaa(v)}{\N{\omegaa}_{L_1(\empmeasure{t})}}\,d\empmeasure{t}(v),
	\quad \text{with}\quad
	\omegaa(v) := \exp(-\alpha \CE(v)).
\end{align}
Definition~\eqref{eq:momentaneous_consensus} is motivated by the well-known Laplace principle \cite{dembo2009large,miller2006applied,pinnau2017consensus}, which states that,
for any absolutely continuous probability distribution $\indivmeasure$ on $\bbR^d$, we have % supported on $\bbR^d$ sounds like $\supp(\indivmeasure)=\bbR^d$
\begin{align}
\label{eq:laplace_principle}
    \lim\limits_{\alpha\rightarrow \infty}\left(-\frac{1}{\alpha}\log\left(\int\omegaa(v)\,d\indivmeasure(v)\right)\right) = \inf\limits_{v \in \supp(\indivmeasure)}\CE(v).
\end{align}
Alternatively, we can also interpret \eqref{eq:momentaneous_consensus} as an approximation of $\argmin_{i=1,\ldots,N}\CE(V_t^i)$, which improves as $\alpha\rightarrow \infty$, provided the minimizer uniquely exists.
The univariate function $\cutoffnoarg : \bbR\rightarrow [0,1]$ appearing in the first term of \eqref{eq:dyn_micro_discrete} and \eqref{eq:dyn_micro} can be used to deactivate the drift term for agents $V_t^i$, whose objective \revised{is better than the one of} the momentaneous consensus, i.e., for which $\CE(V_t^i) < \CE(\textconspoint{\empmeasure{t}})$, by setting $\cutoff{x} \approx \mathbbm{1}_{x \geq 0}$.
The most frequently studied choice however is $\cutoffnoarg \equiv 1$.
The second term in \eqref{eq:dyn_micro_discrete} and \eqref{eq:dyn_micro} encodes the diffusion or exploration mechanism of the algorithm.
Intuitively, scaling by $\textN{V_t^i-\textconspoint{\empmeasure{t}}}_2$ encourages agents far from the consensus point to explore larger regions, whereas agents close to the consensus point try to enhance their position only locally.
Furthermore, the scaling is essential to eventually deactivate the Brownian motion and to achieve consensus among the individual agents.

%CBO algorithms implement the update rule~\eqref{eq:dyn_micro_discrete} (or possibly variations thereof, depending on the problem) and compute $\textconspoint{\empmeasure{T}}$ as a guess for the global minimizer.
CBO methods have been considered and analyzed in several recent papers \cite{kim2020stochastic,totzeck2020consensus,carrillo2018analytical,carrillo2019consensus,chen2020consensus,fornasier2020consensus_hypersurface_wellposedness,fornasier2020consensus_sphere_convergence,fornasier2021anisotropic,fornasier2021convergence,borghi2021constrained,carrillo2021consensus,kalise2022consensus}, even for optimization problems in high-dimensional and non-Euclidean settings, and using more sophisticated rules for the parameter choices $\alpha$ and $\sigma$ inspired by Simulated Annealing~\cite{carrillo2019consensus, fornasier2020consensus_sphere_convergence}.
Moreover, several variants of the dynamics have been proposed, such as  ones integrating memory mechanisms~\cite{totzeck2020consensus,riedl2022leveraging} or others using jump-diffusion processes~\cite{kalise2022consensus}.
\revised{To make the method feasible and competitive for large-scale applications, in particular, for problems arising in machine learning, random mini-batch sampling techniques have been employed when evaluating the objective function or computing the consensus point.
This significantly reduces the computational and communication complexity of CBO methods~\cite{carrillo2019consensus,fornasier2021convergence}
and further enables the parallelization of the algorithm by evolving disjoint subsets of particles independently for some time with separate consensus points, before aligning the dynamics through a global communication step.
However, despite bearing interesting questions concerning the trade-off between parallel efficiency and performance when it comes to the relevance of communication between the individual agents, this is a so far largely unexplored area for CBO.}
\revised{As an example for the applicability of CBO to such high-dimensional problems, we refer to \cite{carrillo2019consensus,fornasier2021convergence,riedl2022leveraging} where the method is used for training a shallow and a convolutional neural network for image classification of the MNIST database of handwritten digits~\cite{MNIST}, to the recent paper~\cite{trillos2023FedCBO} where CBO is used in the setting of clustered federated learning, to \cite{riedl2022leveraging} where a compressed sensing problem is solved, or to the line of works~\cite{fornasier2020consensus_sphere_convergence,fornasier2020consensus_hypersurface_wellposedness,fornasier2021anisotropic} where \eqref{eq:dyn_micro_discrete} and \eqref{eq:dyn_micro} are adapted to the sphere $\bbS^{d-1}$ achieving near state-of-the-art performance on a phase retrieval, a robust subspace detection problem and when  robustly computing eigenfaces.}
\revised{Recently, also general constrained optimization problems have been tackled by CBO through the use of penalization techniques, which allow to cast the constrained problem into an unconstrained optimization task~\cite{borghi2021constrained,carrillo2021consensus}.}

As initially mentioned, CBO methods are motivated by the urge to develop a class of metaheuristic algorithms with provable guarantees, while preserving their capabilities of escaping local minima through global optimization mechanisms.
The main theoretical interest focuses on understanding when consensus formation of \mbox{$\empmeasure{t}\rightarrow \delta_{\widetilde v}$} occurs, and on quantitatively bounding the associated errors $\CE(\widetilde v) - \minobj$ and $\Nnormal{\widetilde v - \globmin}_2$.
A theoretical analysis of the dynamics can either be done on the microscopic systems~\eqref{eq:dyn_micro_discrete} or \eqref{eq:dyn_micro}, as for instance in \cite{ha2020convergenceHD,ha2021convergence}, or, as in~\cite{pinnau2017consensus,carrillo2018analytical}, by analyzing the macroscopic behavior of the agent density through a mean-field limit associated with the particle-based dynamics~\eqref{eq:dyn_micro}, given, for initial data~$\overbar V_0 \sim \rho_0$, by
\begin{align} \label{eq:dyn_macro}
	d\overbar V_t
	= -\lambda\left(\overbar V_t - \conspoint{\rho_t}\right)\cutoff{\CE(\overbar  V_t)-\CE(\conspoint{\rho_t})}dt
	+ \sigma \N{\overbar  V_t-\conspoint{\rho_t}}_2 dB_t,
\end{align}
where $\rho_t = \Law(\overbar V_t)$.
\revised{The weak convergence of the microscopic system~\eqref{eq:dyn_micro} to the mean-field limit~\eqref{eq:dyn_macro}, or, more precisely, of the empirical measure~$\empmeasure{t}$ to $\rho_t$ as $N\rightarrow\infty$, has been shown recently in \cite{huang2021MFLCBO}, see also Remark~\ref{rem:convergence_details} for additional details.
This legitimates to analyze \eqref{eq:dyn_macro} in lieu of \eqref{eq:dyn_micro}.}
The measure $\rho\in\CC([0,T],\CP(\bbR^d))$ with $\rho_t = \rho(t)= \Law(\overbar V_t)$ satisfies the \revised{nonlinear nonlocal} Fokker-Planck equation
\begin{equation} \label{eq:fokker_planck}
	\partial_t\rho_t
	= \lambda\divergence \big(\!\left(v - \conspoint{\rho_t}\right)\cutoff{\CE(v)-\CE(\conspoint{\rho_t})}\rho_t\big)
	+ \frac{\sigma^2}{2}\Delta\big(\!\N{v-\conspoint{\rho_t}}_2^2\rho_t\big)
\end{equation}
%in a weak sense (see Definition~\ref{def:fokker_planck_weak_sense} and Remark~\ref{rem:test_functions_redefine}).
in a weak sense (see Definition~\ref{def:fokker_planck_weak_sense}).
Leveraging this partial differential equation (PDE), the authors of~\cite{pinnau2017consensus,carrillo2018analytical} analyze the large time behavior of the particle density \mbox{$t \mapsto \rho_t$} instead of the microscopic systems~\eqref{eq:dyn_micro_discrete} and \eqref{eq:dyn_micro}.
Studying the mean-field limit~\eqref{eq:dyn_macro} or \eqref{eq:fokker_planck} allows for agile deterministic calculus tools and typically leads to stronger theoretical results, which characterize the average agent behavior through the evolution of~$\rho$.
This analysis perspective is justified by the mean-field approximation, which quantifies the convergence of the microscopic system~\eqref{eq:dyn_micro} to the mean-field limit~\eqref{eq:dyn_macro} as the number of agents grows.
We discuss results about the mean-field approximation in Remark~\ref{rem:convergence_details} and make it rigorous in Proposition~\ref{prop:MFL}.
Hence, in view of its validity and as already done in the preceding works~\cite{pinnau2017consensus,carrillo2018analytical}, in the first part of the paper we concentrate on \revised{establishing} convergence in mean-field law \revised{for \eqref{eq:dyn_micro}, as defined in Definition~\ref{def:mean_field_law_convergence} below.}
\revised{That is, we analyze the mean-field dynamics~\eqref{eq:dyn_macro} and~\eqref{eq:fokker_planck} in place of the interacting particle system~\eqref{eq:dyn_micro}.}
Afterwards, \revised{by combining the mean-field approximation with convergence in mean-field law,} we close the paper with a global convergence result for the numerical method~\eqref{eq:dyn_micro_discrete}.

\begin{definition}[Convergence in mean-field law] \label{def:mean_field_law_convergence}
	Let $F,G: \CP(\bbR^d) \otimes \bbR^d \rightarrow \bbR^d$ be two functions and consider for $i=1,\ldots,N$ the SDEs expressed in It\^o's form as
	\begin{align*}
        \textstyle
		dV_t^i
		= F\!\left(\empmeasure{t}, V_t^i\right)dt + G\!\left(\empmeasure{t}, V_t^i\right)dB_t^i,\quad &\text{where }
        \textstyle
        \empmeasure{t}
		= \frac{1}{N}\sum_{i=1}^{N} \delta_{V_t^i},\ \text{and } V_0^i \sim \rho_0.\\
		\intertext{We say that this SDE system converges in mean-field law to $\widetilde v \in \bbR^d$ if all solutions of}
		d\overbar V_t
		= F\!\left(\rho_t, \overbar V_t\right)dt + G\!\left(\rho_t, \overbar V_t\right)dB_t,\quad &\text{where } \rho_t
		= \Law(\overbar V_t),\ \text{and } \overbar V_0 \sim \rho_0,
	\end{align*}
	satisfy $\lim_{t\rightarrow \infty} W_p\left(\rho_t,\delta_{\widetilde v}\right) = 0$ for some \mbox{Wasserstein-$p$} distance~$W_p$, $p\geq1$.
\end{definition}

\revised{Colloquially speaking, an interacting multi-particle system is said to converge {\it in mean-field law}, if the associated mean-field dynamics converges.}

%\begin{remark} \label{rem:convergence_details}
%Recently, the weak convergence of a variant of the microscopic system \eqref{eq:dyn_micro}, which is supported on a compact hypersurface $\Gamma$, to the corresponding mean-field limit has been established \cite[Theorem 3.1]{fornasier2020consensus_hypersurface_wellposedness}.
%The convergence rate reads $C_{\alpha}/N$ with
%\begin{align} \label{eq:Calpha_constant}
%	C_{\alpha}\approx \exp\left(\alpha \left(\sup_{v\in\Gamma} \CE(v) - \inf_{v \in \Gamma} \CE(v)\right)\right).
%\end{align}
%If $\CE:\bbR^d\rightarrow \bbR$ were bounded from above and below, the same technique could be used to establish the mean-field limit
%for the dynamics~\eqref{eq:dyn_micro}  in $\bbR^d$.
%However, the technique is no longer applicable under conditions such as coercivity, which can also be used to ensure the well-posedness of \eqref{eq:dyn_micro} and \eqref{eq:dyn_macro}, see~\cite{carrillo2018analytical} or Theorems~\ref{thm:well-posedness_FP} and \ref{thm:well-posedness_FP_cutoff} below.
%Since we believe that a finer analysis allows for establishing the $C_{\alpha}N^{-1}$ rate for the convergence of the microscopic system \eqref{eq:dyn_micro} to the mean-field limit
%also in the unbounded case, we focus in this work on the macroscopic dynamics~\eqref{eq:dyn_macro} and leave the mean-field approximation of the microscopic
%system for future work.
%\end{remark}

\begin{remark}[Mean-field approximation]
\label{rem:convergence_details}
	The definition of convergence in mean-field law \revised{as given in Definition~\ref{def:mean_field_law_convergence}} is justified as follows:
	As the number of agents~$N$ in the interacting particle system~\eqref{eq:dyn_micro} tends to infinity, one expects that, for any particle~$V^i$, the individual influence of any other particle disperses.
	This results in an averaged influence of the \revised{ensemble} rather than an interacting nature of the system, and allows to describe the dynamics in the large-particle limit by the law $\rho$ of the mono-particle process~\eqref{eq:dyn_macro}.
	This phenomenon is known as the mean-field approximation. %originates from physics~\cite{vlasov1968vibrational}
	More formally, as \mbox{$N\rightarrow\infty$,} we expect the empirical measure~$\empmeasure{t}$ to converge in law to $\rho_t$ for almost every $t$, see~\cite[Definition~1]{jabin2017mean}.
	The classical way to establish such mean-field approximation is to prove, \revised{by means of the coupling method,} propagation of chaos~\cite{mckean1967propagation,sznitman1991propagation}, as implied for instance by
	\begin{align} \label{eq:mean-field_approximation}
        \max_{i=1,\dots,N}\sup_{t\in[0,T]}\bbE\Nbig{V^i_t-\overbar V^i_t}_2^2\leq CN^{-1},
	\end{align}
	where $\overbar V^i$ denote $N$ i.i.d.\@~copies of the mean-field dynamics~\eqref{eq:dyn_macro}, \revised{which are coupled to the processes $V^i$ by choosing the same initial conditions as well as Brownian motion paths, see, e.g., the recent review~\cite{chaintron2021propagation,chaintron2022propagation}.}
	Despite being of fundamental numerical interest \revised{(since when combined with the convergence in mean-field law it allows to establish convergence of the interacting particle system itself),} a quantitative result about the mean-field approximation of CBO as in~\eqref{eq:mean-field_approximation} has been left as a difficult and open problem in \cite[Remark~3.3]{carrillo2018analytical} \revisedTwo{due to a lack of global Lipschitz continuity of the drift and diffusion terms, which impedes the application of McKean's theorem~\cite[Theorem~3.1]{chaintron2021propagation}.}

	However, the present work as well as \revised{three} recent works, which we outline in what follows, are shedding light on this issue. By employing a compactness argument in the path space, the authors of~\cite{huang2021MFLCBO} show that the empirical random particle measure~$\widehat\rho^N$ associated with the dynamics~\eqref{eq:dyn_micro} converges in distribution to the deterministic particle distribution~$\rho\in\CC([0,T],\CP(\bbR^d))$ satisfying~\eqref{eq:fokker_planck}.
	In particular, their result is valid for unbounded functions $\CE$ considered also in our work.
	While this does not allow for obtaining a quantitative convergence rate with respect to the number of particles~$N$ as in~\eqref{eq:mean-field_approximation}, it closes the mean-field limit gap qualitatively.
	A desired quantitative result has been established recently in~\cite[Theorem~3.1 and Remark~3.1]{fornasier2020consensus_hypersurface_wellposedness} for a variant of the microscopic system \eqref{eq:dyn_micro} supported on a compact hypersurface~$\Gamma$.
	In~\cite{fornasier2020consensus_hypersurface_wellposedness} the weak convergence of the variant of~\eqref{eq:dyn_micro} to the corresponding mean-field limit is established in the sense that for all $\phi\in\CC^1_{b}(\bbR^d)$ it holds
	\begin{align*}
		\sup_{t\in[0,T]} \bbE\left[\left|\langle\widehat\rho^N_t,\phi\rangle-\langle\rho_t,\phi\rangle\right|^2\right] \leq \frac{C}{N}\N{\phi}^2_{\CC^1(\bbR^d)} \rightarrow0 \quad \text{as } N\rightarrow\infty.
	\end{align*}
	The obtained convergence rate reads $CN^{-1}$ with $C$ depending in particular on
	\begin{align*}
		%C_{\alpha} := \exp\left(\alpha \left(\sup_{v\in\Gamma} \CE(v) - \inf_{v \in \Gamma} \CE(v)\right)\right).
		\textstyle
		C_{\alpha} := \exp\big(\alpha \big(\sup_{v\in\Gamma} \CE(v) - \inf_{v \in \Gamma} \CE(v)\big)\big).
	\end{align*}
	Their proof is based on the \revised{aforementioned coupling method and, by exploiting the inherent compactness of the dynamics due to its confinement to $\Gamma$, allows} to derive a bound of the form~\eqref{eq:mean-field_approximation}.
	Leveraging the techniques from~\cite{fornasier2020consensus_hypersurface_wellposedness} and the boundedness of moments established in \cite[Lemma~3.4]{carrillo2018analytical}, we provide in Proposition~\ref{prop:MFL} below a result of the form~\eqref{eq:mean-field_approximation} on the plane~$\bbR^d$ which holds with high probability.
	\revisedTwo{A more refined analysis conducted recently by the authors of \cite{gerber2023propagation},
	which adapts Sznitman's classical argument for the proof of McKean's theorem with the intention of allowing for coefficients which are not globally Lipschitz,
	even yields a non-probabilistic mean-field approximation of the form~\eqref{eq:mean-field_approximation} in the pathwise sense, requiring in comparison merely a higher moment bound~$\rho_0 \in \CP_{6}(\bbR^d)$ of the initial measure, see \cite[Theorem~2.6]{gerber2023propagation}.}

	Such quantitative mean-field approximation results substantiate the focus of the first part of this work on the analysis of the macroscopic mean-field dynamics~\eqref{eq:dyn_macro} and~\eqref{eq:fokker_planck}.
	Nevertheless, as a consequence thereof, we return to the analysis of the numerical scheme \eqref{eq:dyn_micro_discrete} and its global convergence in Section \ref{subsec:convergence_probability}.
\end{remark}

\paragraph{Contributions}
%In view of the versatility, simplicity, and efficiency of CBO methods, a theoretical understanding of the finite particle-based system~\eqref{eq:dyn_micro} and the mean-field limit~\eqref{eq:dyn_macro} is  of great interest.
In this work we unveil the surprising phenomenon that, in the mean-field limit, for a \revisedTwo{rich} class of objectives $\CE$, the individual agents of the CBO dynamics follow the gradient flow associated with the function $v\mapsto \Nnormal{v-\globmin}_2^2$, on average over all realizations of Brownian motion paths, see Figure \ref{figure:intuition}.
\begin{figure}
	\centering
%	\subcaptionbox{\footnotesize The Rastrigin function~$\CE$ and an exemplary initialization for one run of the experiment \label{fig:objective2d}}{\includegraphics[trim=33 21 23 22,clip,width=0.44\textwidth]{img/CBOIntuitionObjectiveFunction_RastriginN3200.pdf}}
	\subcaptionbox{\footnotesize The Rastrigin function~$\CE$ and an exemplary initialization for one run of the experiment \label{fig:objective2d}}{\includegraphics[trim=13 21 3 22,clip,width=0.46\textwidth]{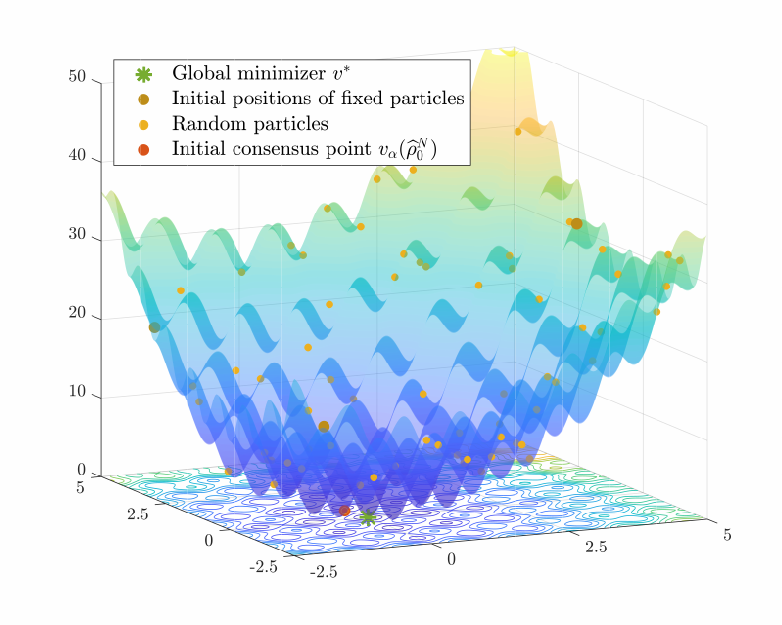}
	}
%	\hspace{2em}
\hspace{1.5em}
%	\subcaptionbox{\footnotesize Individual agents follow, on average, the gradient flow of the map $v\mapsto\N{v-\globmin}_2^2$. \label{fig:intuition_ce}}{\includegraphics[trim=48 209 51 200,clip,width=0.42\textwidth]{img/CBOIntuition_RastriginN32003sigma10div100_redres.pdf}}
	\subcaptionbox{\footnotesize Individual agents follow, on average, the gradient flow of the map $v\mapsto\N{v-\globmin}_2^2$. \label{fig:intuition_ce}}{\includegraphics[trim=28 209 31 200,clip,width=0.43\textwidth]{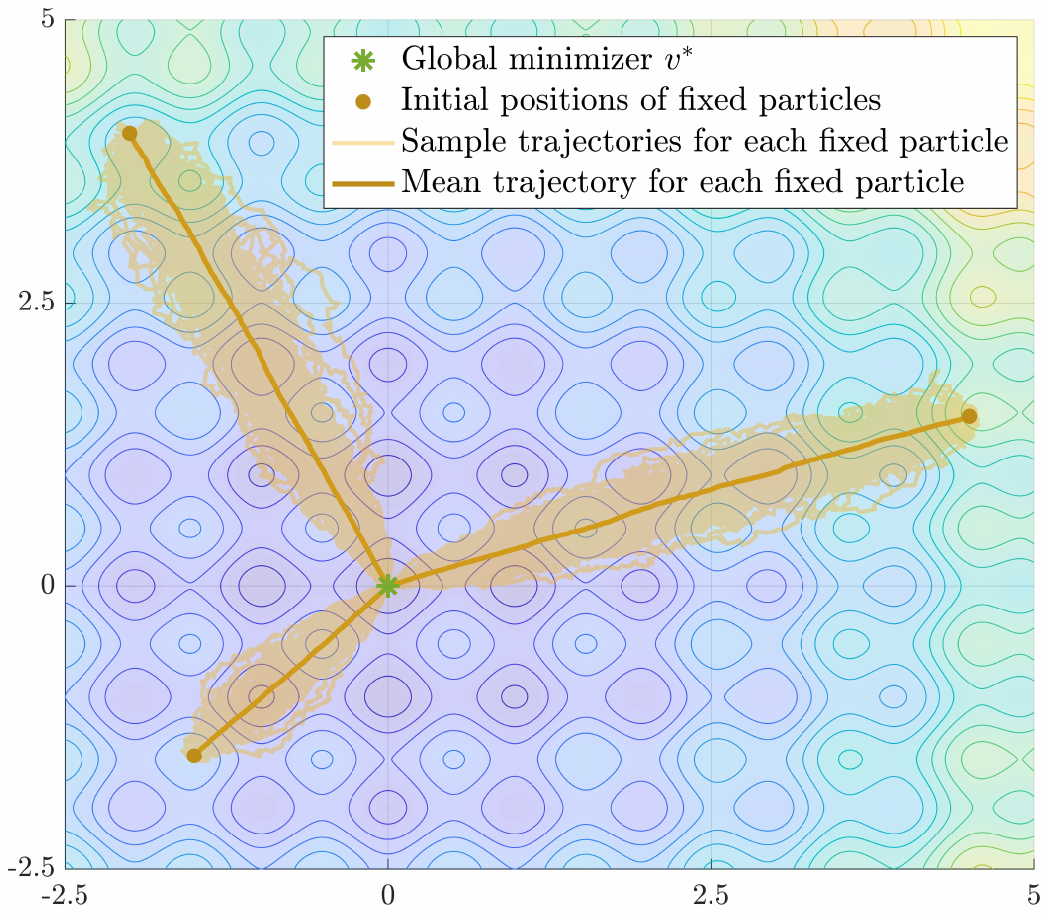}
	}
	\caption{An illustration of the internal mechanisms of CBO.
	We perform $100$ runs of the CBO algorithm \eqref{eq:dyn_micro_discrete}--\eqref{eq:dyn_micro_discrete_2}, with parameters $\Delta t=0.01$, $\alpha = 10^{15}$, $\lambda = 1$ and $\sigma = 0.1$, and $N=32000$ agents \revised{initialized according to} $\rho_0 = \CN\big((8,8), 20\big)$.
	In addition, we add three individual agents with starting locations $(-2,4)$, $(-1.5,-1.5)$ and $(4.5,1.5)$ to the set of agents in each run as shown in \textbf{(a)}, and depict each of their $100$ trajectories \revised{as well as} their mean trajectory in yellow color in \textbf{(b)}.
    \revised{With the (mean) trajectories being rather straight lines, we}
	observe that \revised{the individual agents} take a straight path from their initial positions to the global minimizer~$v^*$ and, in particular, disregard the local landscape of the objective function~$\CE$.
	The trajectories \revised{of the individual agents become more concentrated} as the \revised{overall} number of agents~$N$ grows.}
	\label{figure:intuition}
	\vspace{-0.5cm}
\end{figure}
Interestingly, this gradient flow is independent of the underlying energy landscape of $\CE$.
%, which is due to the fact that the consensus point $\conspoint{\rho_t}$ depends mostly on the behavior of $\CE$ in a small neighborhood around $\globmin$ under a suitable local coercivity condition of $\CE$ and in the regime of large parameter $\alpha \gg 1$.
In other words, CBO performs a canonical convexification of \revisedTwo{a large} class of optimization problems as the number of optimizing agents $N$ goes to infinity.
Based on these observations and justified by the mean-field approximation, first of all we develop a novel proof \revisedTwo{framework} for showing the convergence of the CBO dynamics in mean-field law to the global minimizer $\globmin$ for a rich class of objectives.
While previous analyses in \cite{ha2020convergenceHD,ha2021convergence,carrillo2018analytical} required \revisedTwo{restrictive} concentration conditions about the initial measure $\rho_0$ and $\CC^2$ regularity of the objective, we derive results that are valid under \revisedTwo{mild} assumptions about $\rho_0$ and local Lipschitz continuity of $\CE$.
We explain the key differences of this work with respect to previous work in detail in Section~\ref{sec:blueprints} and further \revisedTwo{showcase} the benefits of the proposed analysis by a numerical example.
These findings reveal that the hardness of any global optimization problem is necessarily encoded in the rate of the mean-field approximation as $N\rightarrow\infty$.
Secondly, in consideration of its central significance with regards to the computational complexity of the numerical scheme~\eqref{eq:dyn_micro_discrete} we establish a novel probabilistic quantitative result about the convergence of the interacting particle system~\eqref{eq:dyn_micro} to the corresponding mean-field limit~\eqref{eq:fokker_planck}, which is a result of independent interest.
By combining these two results, the convergence in mean-field law on the one hand, and the quantitative mean-field approximation on the other, we provide the first, and so far unique, holistic convergence proof of CBO on the plane, \revised{enabling to quantify the optimization capability of the numerical CBO algorithm~\eqref{eq:dyn_micro_discrete} in terms of the used parameters.}
The utilized proof technique may be used as a \revisedTwo{blueprint} for proving global convergence for other recent adaptations of the CBO dynamics, see, e.g., \cite{carrillo2019consensus,fornasier2020consensus_sphere_convergence,fornasier2021anisotropic,fornasier2021convergence,borghi2021constrained,kalise2022consensus}, as well as other metaheuristics such as the renowned Particle Swarm Optimization, \revised{which is related to CBO through a zero-inertia limit, see, e.g., \cite{grassi2020particle,cipriani2021zero,qiu2022PSOconvergence}.}
While the present paper has foundational and theoretical nature and aims at completely clarifying the convergence of the numerical scheme \eqref{eq:dyn_micro_discrete} with a detailed analysis, we do not include extensive numerical experiments.
\revised{For numerical evidence that CBO does solve difficult optimizations also in high dimensions without necessarily incurring in the curse of dimensionality, the reader may want to consult previous work such as \cite{carrillo2019consensus,chen2020consensus,fornasier2020consensus_sphere_convergence,fornasier2021anisotropic,fornasier2021convergence,riedl2022leveraging,trillos2023FedCBO}.}

\revised{\begin{remark}
    Employing stochasticity and leveraging collaboration between multiple agents to empirically and provably achieve global convergence of numerical algorithms and to avoid convergence to local minima, is not just of particular relevance when it comes to the efficiency and success of zero-order methods, but also an emerging paradigm in the field of gradient-based optimization, see, e.g., \cite{dong2021replica,lu2022swarm,chizat2022meanfield}.
    Recent work~\cite{riedl2023all,riedl2023all2} even suggests
    %, at the specific example of CBO and gradient descent,
    a connection between the worlds of derivative-free and gradient-based methods.
    Similar guiding principles are present also in sampling methods, such as Langevin sampling~\cite{chiang1987diffusion,roberts1996exponential,durmus2017nonasymptotic,chizat2022meanfield} or Stein Variational Gradient Descent \cite{liu2016stein}, which are designed to generate samples from an unknown target distribution.

    A promising way to gain a theoretical understanding of the behavior of these classes of algorithms is by taking a mean-field perspective, i.e., by analyzing the dynamics, as the number of particles goes to infinity, through an associated PDE.
    This typically involves Polyak-{\L}ojasiewicz-like conditions~\cite{karimi2016linear} or certain families of log-Sobolev inequalities~\cite{chizat2022meanfield} on the objective function~$\CE$, which are more restrictive than the assumptions under which the statements of this work hold.
    For a recent analysis of the mean-field Langevin dynamics we refer to~\cite{chizat2022meanfield} and references therein. 

    Lately and conceptually similar to the convexification of a highly nonconvex problem observed in this work, taking a mean-field perspective  has allowed the authors of \cite{mei2018mean,rotskoff2018trainability,chizat2018global,sirignano2020mean} to explain the generalization capabilities of over-parameterized neural networks.
    By leveraging that the mean-field description (w.r.t.\@ the number of neurons) of the SGD learning dynamics is captured by a nonlinear PDE, which admits a gradient flow structure on~$\left(\CP_2(\bbR^d), W_2\right)$, these works show that original complexities of the loss landscape are alleviated.
    Together with a quantification of the fluctuations of the empirical neuron distribution around this mean-field limit (i.e., a mean-field approximation), convergence results are derived for SGD for sufficiently large networks with optimal generalization error.
    These results, however, are structurally different from the ones obtained in this paper for CBO.
    In particular, the individual particles in \cite{mei2018mean,rotskoff2018trainability,chizat2018global,sirignano2020mean} are associated with the different neurons of a two-layer \revisedTwo{or deep} neural network and the objective function is a specific empirical risk,
    \revisedTwo{which itself is subject to the mean-field limit and gains convexity as the number of neurons tends to infinity.
    In contrast,} in our setting each particle itself is a competitor for minimization of a general \revisedTwo{fixed nonconvex} objective function $\CE$ \revisedTwo{and the convexification of the problem emerges from the CBO dynamics when its mean-field limit behavior is studied.}
    For this reason, the two resulting mean-field limits are different.
    
    Let us further point out that, as far as the community could understand up to now, the Fokker-Planck equation \eqref{eq:fokker_planck} describing the mean-field behavior of CBO cannot be understood as a gradient flow of any energy on $\left(\CP_2(\bbR^d), W_2\right)$.
    Yet, and perhaps surprisingly, the analysis of our present paper shows that the Wasserstein-$2$ distance from the global minimizer is the correct Lyapunov functional to be analyzed.
\end{remark}}

\subsection{Organization} \label{subsec:organization}
In Section~\ref{sec:blueprints} we first discuss state-of-the-art global convergence results for CBO methods with a detailed account of the utilized proof technique, including potential weaknesses.
The second part of Section~\ref{sec:blueprints} then motivates an alternative proof strategy and explains how it can remedy the weaknesses of prior proofs under minimalistic assumptions.
In Section~\ref{sec:cbo_envelope_min} we first provide additional details about the well-posedness of the macroscopic SDE~\eqref{eq:dyn_macro}, respectively, the Fokker-Planck equation \eqref{eq:fokker_planck}, before presenting and discussing the main result about the convergence of the dynamics~\eqref{eq:dyn_macro} and~\eqref{eq:fokker_planck} to the global minimizer in mean-field law.
In order to demonstrate the relevance of such statement in establishing a holistic convergence guarantee for the numerical scheme~\eqref{eq:dyn_micro_discrete}, we conclude the section with a probabilistic quantitative result about the mean-field approximation.
Sections~\ref{sec:proof_main_theorem} and~\ref{sec:proof_probabilistic_MFA} comprise the proof details of the convergence result in mean-field law and the result about the mean-field approximation, respectively.
Section~\ref{sec:conclusions} concludes the paper.

For the sake of reproducible research, in the GitHub repository \url{https://github.com/KonstantinRiedl/CBOGlobalConvergenceAnalysis} we provide the Matlab code implementing the CBO algorithm analyzed in this work and used to create all visualizations.
Python and Julia code for CBO~\cite{Bailo2024} can be also found in the GitHub repositories \url{https://github.com/PdIPS/CBXpy} and \url{https://github.com/PdIPS/ConsensusBasedX.jl}.

\subsection{Notation}
\label{subsec:notation}
Euclidean balls are denoted as \mbox{$B_{r}(u) := \{v \in \bbR^d\!:\! \Nnormal{v-u}_2 \leq r\}$}.
For the space of continuous functions~$f:X\rightarrow Y$ we write $\CC(X,Y)$, with $X\subset\bbR^n$ and a suitable topological space $Y$.
For an open set $X\subset\bbR^n$ and for $Y=\bbR^m$ the spaces~$\CC^k_{c}(X,Y)$ and~$\CC^k_{b}(X,Y)$ contain functions~$f\in\CC(X,Y)$ that are $k$-times continuously differentiable and have compact support or are bounded, respectively.
We omit $Y$ in the real-valued case.
The operators $\nabla$ and $\Delta$ denote the gradient and Laplace operator of a function on~$\bbR^d$.
The main objects of study are laws of stochastic processes, $\rho\in\CC([0,T],\CP(\bbR^d))$, where the set $\CP(\bbR^d)$ contains all Borel probability measures over $\bbR^d$.
With $\rho_t\in\CP(\bbR^d)$ we refer to a snapshot of such law at time~$t$.
In case we refer to some fixed distribution, we write~$\indivmeasure$.
Measures~$\indivmeasure \in \CP(\bbR^d)$ with finite $p$-th moment $\int \Nnormal{v}_2^p\,d\indivmeasure(v)$ are collected in $\CP_p(\bbR^d)$.
For any $1\leq p<\infty$, $W_p$ denotes the \mbox{Wasserstein-$p$} distance between two Borel probability measures~$\indivmeasure_1,\indivmeasure_2\in\CP_p(\bbR^d)$, see, e.g., \cite{savare2008gradientflows}.
% is defined by
%\begin{align} \label{def:wassersteindistance}
%	W_p(\indivmeasure_1,\indivmeasure_2) = \left(\inf_{\pi\in\Pi(\indivmeasure_1,\indivmeasure_2)}\int\N{v_1-v_2}_2^pd\pi(v_1,v_2)\right)^{1/p},
%\end{align}
%where $\Pi(\indivmeasure_1,\indivmeasure_2)$ denotes the set of all couplings of $\indivmeasure_1$ and $\indivmeasure_2$, i.e., the collection of all Borel probability measures over $\mathbb{R}^d\times\mathbb{R}^d$ with marginals $\indivmeasure_1$ and $\indivmeasure_2$ on the first and second component, respectively.
$\bbE(\indivmeasure)$ denotes the expectation of a probability measure $\indivmeasure$.

\section{Blueprints for the analysis of CBO methods} \label{sec:blueprints}
In this section we provide intuitive descriptions of two approaches to the analysis of the convergence of CBO methods to global minimizers.
We first recall \cite{carrillo2018analytical}, and related works \cite{ha2020convergenceHD,ha2021convergence}, which prove convergence as a consequence of a monotonous decay of the variance of $\rho_t$ and by employing the asymptotic Laplace principle~\eqref{eq:laplace_principle}.
This proof strategy incurs a restrictive condition about the parameters $\alpha,\lambda,\sigma$ and the initial configuration $\rho_0$, which implies that a small optimization gap $\CE(\widetilde v) - \CE(\globmin)$ can only be achieved for initial configurations~$\rho_0$ already well-concentrated near the optimizer $\globmin$.
We then motivate an alternative proof idea to remedy this weakness based on the intuition that $\rho_t$ monotonically minimizes the squared Euclidean distance to the global minimizer~$\globmin$.

\subsection{State of the art: variance-based convergence analysis} \label{subsec:variance_based_analysis}
We now recall the blueprint proof strategy from \cite{carrillo2018analytical}, which has been adapted in other works, e.g.,  \cite{ha2020convergenceHD,ha2021convergence,fornasier2020consensus_sphere_convergence}, to prove consensus formation and convergence to the global minimum.

A successful application of the CBO framework underlies the premise that the induced particle density $\rho_t$ converges to a Dirac delta $\delta_{\widetilde v}$ for some $\widetilde v$ close to $\globmin$.
The analysis in \cite{carrillo2018analytical} proves this under certain assumptions by first showing that $\rho_t$ converges to a Dirac delta around \emph{some} $\widetilde v \in \bbR^d$ and then concluding $\widetilde v \approx \globmin$ in a subsequent step.
% in two steps. First it is shown
% that $\rho_t$ converges to a Dirac delta around \emph{any} $\widetilde v \in \bbR^d$,
% and then that $\widetilde v$ is close to $\globmin$, thus resulting in a convergence in mean-field law to $\globmin$.
Regarding the first step, the authors of \cite{carrillo2018analytical} study the variance of $\rho_t$, defined as
% Regarding the first step, the convergence of $\rho_t$ to a Dirac delta requires that the particle density variance, defined by
% \begin{align*}
% 	\Var{\rho_t} := \frac{1}{2}\int\N{v-\bbE(\rho_t)}_2^2d\rho_t(v),
% 	\quad \text{where}\quad
% 	\bbE(\rho_t) := \int v\,d\rho_t(v),
% \end{align*}
$\Var{\rho_t} := \frac{1}{2}\int\N{v-\bbE(\rho_t)}_2^2d\rho_t(v)$,
where $\bbE(\rho_t) := \int v\,d\rho_t(v)$,
and show that $\Var{\rho_t}$ decays exponentially fast in $t$ under a well-preparedness assumption about the initial condition $\rho_0$.
%As a matter of fact, exponential decay is a stronger statement compared to $\Var{\rho_t}\rightarrow 0$ for $t\rightarrow \infty$, as required per $\rho_t\rightarrow \delta_{\widetilde v}$, and the numerical test in Example~\ref{ex:numerical_example} indeed confirms that the two statements are not equivalent.
%The first analysis step is motivated by this technical expectation, and aims for proving a monotonous exponential decay of $\Var{\rho_t}$.
%We note that this is a stronger statement compared to $\Var{\rho_t}\rightarrow 0$ for $t\rightarrow \infty$, and the numerical test in Example~\ref{ex:numerical_example} confirms that the two statements are indeed not equivalent.
%Let us make the description of the analysis more precise now.
More precisely, in \cite[Section~4.1]{carrillo2018analytical} the authors use It\^o's lemma to derive for the time-evolution of $\Var{\rho_t}$ the expression
\begin{align}
	\frac{d}{dt}\Var{\rho_t}
	%&= -2\lambda \Var{\rho_t} + \frac{d\sigma^2}{2}\int \N{v-\conspoint{\rho_t}}_2^2d\rho_t(v) \label{eq:variance_evolution} \\
	&= -\left(2\lambda-d\sigma^2\right) \Var{\rho_t} + \frac{d\sigma^2}{2}\N{\bbE(\rho_t) - \conspoint{\rho_t}}^2_2. \label{eq:variance_evolution_2}
\end{align}
%where in the second line we used $\int \Nnormal{v-\conspoint{\rho_t}}_2^2\,d\rho_t = 2\Var{\rho_t} + \Nnormal{\bbE(\rho_t) - \conspoint{\rho_t}}^2_2$.
For parameter choices $2\lambda > d\sigma^2$, the first term in \eqref{eq:variance_evolution_2} is negative and one could \emph{almost} apply Gr\"onwall's inequality to obtain the asserted exponential decay of $\Var{\rho_t}$.
However, the second term can be problematic and the main difficulty is to control the distance $\Nnormal{\bbE(\rho_t) - \conspoint{\rho_t}}_2$ between the mean and the weighted mean. For \mbox{$\alpha\rightarrow 0$} the weight function $\omegaa(v) = \exp(-\alpha \CE(v))$ associated with $\conspoint{\rho_t}$ converges to $1$ pointwise and consequently $\conspoint{\rho_t} \rightarrow \bbE(\rho_t)$.
However, the second proof step, explained below, reveals that the crucial regime is $\alpha\gg 1$.
In this case $\conspoint{\rho_t}$ can be arbitrarily far from $\bbE(\rho_t)$ if we do not dispose of additional knowledge about the probability measure $\rho_t$.
To restrict the set of probability measures $\rho_t$ that need to be considered when bounding $\Nnormal{\bbE(\rho_t) - \conspoint{\rho_t}}_2$, the authors of \cite{carrillo2018analytical} compromise to assume that the initial distribution $\rho_0$ satisfies the well-preparedness assumptions
\begin{equation} \label{eq:carrillo_conditions}
%\begin{split}
	\alpha e^{-2\alpha \underbarscript{\CE}}(\sigma^2\!+\!2\lambda) \!<\! 3/8 \ \, \text{ and } \ \, 2\lambda\! \N{\omegaa}^{2}_{L_1(\rho_0)} \!-\! \Var{\rho_0} \!-\! 2d\sigma^2\N{\omegaa}_{L_1(\rho_{0})}e^{-\alpha \underbarscript{\CE}} \!\geq\! 0.
%\end{split}
\end{equation}
Since $\rho_t$ evolves from $\rho_0$ according to the Fokker-Planck equation~\eqref{eq:fokker_planck}, these conditions restrict $\rho_t$ and allow for bounding $\Nnormal{\bbE(\rho_t) - \conspoint{\rho_t}}_2$ by a suitable multiple of $\Var{\rho_t}$.
The exponential decay of $\Var{\rho_t}$ then follows from \eqref{eq:variance_evolution_2} after applying Gr\"onwall's inequality, see~\cite[Theorem~4.1]{carrillo2018analytical}.
Furthermore, the conditions in \eqref{eq:carrillo_conditions} also allow for proving convergence of $\rho_t$ to a stationary Dirac delta at $\widetilde v \in \bbR^d$.

Given convergence to a Dirac at $\widetilde v$, in a second step it is shown $\CE(\widetilde v) \approx \CE(\globmin)$.
%In order to prove this approximation, one first deduces that for any $\varepsilon > 0$, there exists $\alpha \gg 1$ such that $\Nnormal{\omegaa}_{L_1(\rho_t)} \geq \Nnormal{\omegaa}_{L_1(\rho_0)} - \frac{\varepsilon}{2}$ for all $t\geq 0$. 
%This follows as a consequence after deriving a lower bound for the evolution $d/dt \Nnormal{\omegaa}_{L_1(\rho_t)}$ for sufficiently large $\alpha > 0$, see \cite[Lemma~4.1]{carrillo2018analytical}.\KR{bitte citation doppelchecken! ggf. sollten hier Quadrate stehen?}\TK{entspricht das nicht eher dem was gezeigt wird? oder hatte ich eine falsche ref?}
In order to prove this approximation, one first deduces that for any $\varepsilon > 0$, there exists $\alpha \gg 1$ such that for all $t\geq 0$ it holds
% \begin{align*} %\label{eq:aux_evol_exp}
% 	-\frac{1}{\alpha}\log\!\left(\Nnormal{\omegaa}_{L_1(\rho_t)}\right) \leq
% 	-\frac{1}{\alpha}\log\!\left(\Nnormal{\omegaa}_{L_1(\rho_0)}\right)+\frac{\varepsilon}{2}.
% \end{align*}
$-\frac{1}{\alpha}\log\!\left(\Nnormal{\omegaa}_{L_1(\rho_t)}\right) \leq
	-\frac{1}{\alpha}\log\!\left(\Nnormal{\omegaa}_{L_1(\rho_0)}\right)+\frac{\varepsilon}{2}$.
This involves deriving a lower bound for the evolution $\frac{d}{dt} \Nnormal{\omegaa}_{L_1(\rho_t)}$ for sufficiently large $\alpha > 0$ as done in~\cite[Lemma~4.1]{carrillo2018analytical}, which is then combined with the formerly proven exponentially decaying variance, see~\cite[Proof of Theorem~4.2]{carrillo2018analytical}.
Then, by recognizing that the Laplace principle~\eqref{eq:laplace_principle} implies the existence of some $\alpha \gg 1$ with
\begin{align} \label{eq:aux_laplace_principle}
	-\frac{1}{\alpha}\log\!\left(\Nnormal{\omegaa}_{L_1(\rho_0)}\right) - \minobj < \frac{\varepsilon}{2},
\end{align}
and by establishing the convergence $\Nnormal{\omegaa}_{L_1(\rho_t)}\rightarrow \exp(-\alpha\CE(\widetilde v))$ as $t\rightarrow \infty$, one obtains the desired result $\CE(\widetilde v) - \minobj < \varepsilon$ in the limit $t\rightarrow \infty$, see~\cite[Lemma~4.2]{carrillo2018analytical}.
The gap $\CE(\widetilde v) - \minobj$ can be tightened by increasing $\alpha$, but it is impossible to establish an explicit relation $\alpha = \alpha(\varepsilon)$ due to the use of the asymptotic Laplace principle. 

This proof sketch unveils a tension on the role of the parameter $\alpha$.
Namely, the second step requires large $\alpha = \alpha(\varepsilon)$ to achieve $\CE(\widetilde v) -\minobj< \varepsilon$.
In fact, $\alpha(\varepsilon)$ may grow uncontrollably as we decrease the accuracy $\varepsilon$.
The first step, however, requires the conditions in~\eqref{eq:carrillo_conditions} which, in the most optimistic case, where $\sigma = 0$, imply
\begin{align}
\label{eq:localization_carrillo}
\Var{\rho_0} \leq \frac{3}{8\alpha} \left(\int \exp\big(-\alpha(\CE(v)-\minobj)\big)\,d\rho_0(v)\right)^2.
\end{align}
Therefore, $\rho_0$ needs to be increasingly concentrated as $\alpha$ increases, and should ideally be supported
on sets where $\CE(v) \approx \minobj$. Designing such distribution $\rho_0$ in
practice seems impossible in the absence of a good initial guess for $\globmin$. In particular, we cannot expect \eqref{eq:localization_carrillo} to hold for
generic choices such as a uniform distribution on a compact set.

We add that the works \cite{ha2020convergenceHD,ha2021convergence} conduct a similarly flavored analysis for the fully time-discretized microscopic system~\eqref{eq:dyn_micro_discrete}, with some differences in the details.
They first show an exponentially decaying variance under mild assumptions about $\lambda$ and $\sigma$, but provided that the same Brownian motion is used for all agents, i.e., $(B_{k\Delta t}^i)_{k=1,\dots,K} = (B_{k\Delta t})_{k=1,\dots,K}$ for all $i = 1,\dots,N$.
Such a choice leads to a less  explorative dynamics, but it simplifies the consensus formation analysis.
For proving $\CE(\widetilde v) \approx \minobj$ however, the authors again require an initial configuration $\rho_0$ that satisfies a technical concentration condition like \eqref{eq:aux_laplace_principle}, see for instance \cite[Remark~3.1]{ha2021convergence}.

\subsection{Alternative approach: CBO minimizes the squared distance to~$\globmin$} \label{subsec:convex_envelope_based_approach}
The approach described in the previous section might suggest that CBO only converges locally, which is in fact not what is observed in practice.
Instead, global optimization is actually expected.
%The variance-based analysis approach in Section~\ref{subsec:variance_based_analysis} has two drawbacks.
%Firstly, the analysis seems motivated by the technical expectation that the variance must vanish if the CBO method reaches \emph{any} consensus, which unfortunately does not shed light on the internal CBO mechanisms that lead to a successive minimization of the objective $\CE$.
%Secondly, technical conditions such as \eqref{eq:localization_carrillo}, which severely restrict the initial configuration $\rho_0$, are undesirable because they suggest that CBO methods are only successful if we have an informative initial guess for the global minimizer $\globmin$.
To remedy the locality requirements of the variance-based analysis, let us now sketch and motivate an alternative proof idea.
By averaging out the randomness associated with different realizations of Brownian motion paths, the macroscopic time-continuous SDE~\eqref{eq:dyn_macro}, in the case $H\equiv 1$, becomes
\begin{align} \label{eq:EVt}
	\frac{d}{dt}\bbE\!\left[\overbar V_t\big|\overbar V_0\right]
	&= -\lambda\bbE\!\left[\left(\overbar V_t - \conspoint{\rho_t}\right)\!\big|\overbar V_0\right]
	= -\lambda\bbE\!\left[\left(\overbar V_t - \globmin\right)\!\big|\overbar V_0\right] + \lambda\left(\conspoint{\rho_t} - \globmin\right)\!.
\end{align}
Furthermore, if $\CE$ is locally Lipschitz continuous and satisfies the coercivity condition
\begin{align} \label{eq:ICP_outside_def}
	\N{v-\globmin}_2 \leq \frac{1}{\eta}\big(\CE(v) - \CE(v^*)\big)^{\nu} = \frac{1}{\eta}\big(\CE(v) - \minobj\big)^{\nu}, \quad \text{for all } v \in \bbR^d,
\end{align}
and for some $\eta > 0$ and $\nu \in (0,\infty)$, the second term on the right-hand side of \eqref{eq:EVt} can be made arbitrarily small for sufficiently large $\alpha$, i.e., $\conspoint{\rho_t} \approx v^*$
%(more details follow below in Remark~\ref{rem:quantiative_laplace_principle} and in Proposition~\ref{lem:laplace_alt}).
(more details follow below).
In this case, the average dynamics of $\overbar{V}_t$ is well-approximated by
%In this case, the average dynamics of $\overbar{V}_t$ can be seen as a perturbative version of 
\begin{align} \label{eq:EVt_approx}
	\frac{d}{dt}\bbE\left[\overbar V_t|\overbar V_0\right]
	\approx -\lambda\bbE\left[\left(\overbar V_t - \globmin\right)\!|\overbar V_0\right],
\end{align}
which corresponds to the gradient flow of $v\mapsto \Nnormal{v-\globmin}_2^2$ with rate $2\lambda$.
In other words, each individual agent essentially performs a gradient-descent of $v\mapsto \Nnormal{v-\globmin}_2^2$ on average over all realizations of Brownian motion paths.
Figure~\ref{fig:intuition_ce} visualizes this phenomenon for three isolated agents on the Rastrigin function in two dimensions.

Inspired by this observation, our proof strategy is to show that CBO methods successively minimize the energy functional~$\CV:\CP(\bbR^d)\rightarrow \bbR_{\geq 0}$, given by
\begin{align} \label{def:J}
	\CV(\rho_t) := \frac{1}{2}\int\N{v-\globmin}_2^2d\rho_t(v).
\end{align}
Note that this functional essentially coincides with the Wasserstein distance in the sense that $	W_2^2(\rho_t,\delta_{\globmin}) =  2\CV(\rho_t)$.
%Therefore, the convergence~$\CV(\rho_t) \rightarrow 0$ as $t\rightarrow \infty$ simultaneously shows consensus formation and the convergence of $\rho_t$ to  the Dirac delta~$\delta_{\globmin}$ with respect to the Wasserstein distance.
Therefore $\CV(\rho_t) \rightarrow 0$ in particular implies that $\rho_t$ converges weakly to $\delta_{\globmin}$, see~\cite[Chapter~7]{savare2008gradientflows}.

This novel approach does not suffer a tension on the parameter $\alpha$ like the variance-based analysis from the previous section.
Roughly speaking (see Lemma~\ref{lem:evolution_of_objective} for details), $\CV(\rho_t)$ follows an evolution similar to  \eqref{eq:variance_evolution_2}, with $\Var{\rho_t}$ being replaced by $\CV(\rho_t)$.
However, we can now bound $\int \Nnormal{v-\conspoint{\rho_t}}_2^2\,d\rho_t(v) \leq 4\CV(\rho_t) + 2 \Nnormal{\conspoint{\rho_t}-\globmin}_2^2$, so that it just remains to control the second term.
In comparison to bounding $\Nnormal{\conspoint{\rho_t}-\bbE(\rho_t)}_2^2$ in terms of $\Var{\rho_t}$ for the variance-based analysis, this requires to bound $\Nnormal{\conspoint{\rho_t}-\globmin}_2^2$ in terms of $\CV(\rho_t)$.
Fortunately, this is a much easier task: the Laplace principle generally asserts $\Nnormal{\conspoint{\rho_t}-\globmin}_2 \rightarrow 0$ under \eqref{eq:ICP_outside_def} as $\alpha \rightarrow \infty$ and we can even establish (see Proposition~\ref{lem:laplace_alt} for details) the quantitative estimate
\begin{align*}
	\N{\conspoint{\indivmeasure} - \globmin}_2
	\leq \frac{(2Lr)^{\nu}}{\eta}+ \frac{\exp\left(-\alpha L r\right)}{\indivmeasure(B_{r}(\globmin))}\int\N{v-\globmin}_2d\indivmeasure(v)
\end{align*}
%see Remark~\ref{rem:quantiative_laplace_principle} and Proposition~\ref{lem:laplace_alt},
for an arbitrary probability measure~$\indivmeasure$ and assuming that $\CE$ is $L$-Lipschitz in a ball of radius~$r>0$.
This allows to estimate $\Nnormal{\conspoint{\rho_t} - \globmin}_2^2$ in terms of $\CV(\rho_t)$ as desired.

Finally, we note that $\CV(\rho_t)$ majorizes $\Var{\rho_t}$ because $u\mapsto \frac{1}{2}\!\int \Nnormal{v-u}_2^2\,d\rho_t(v)$ is minimized by the expectation $\bbE(\rho_t)$.
%More precisely, we have
%\begin{align} \label{eq:variance_J_relation}
%	\Var{\rho_t}
%	= \frac{1}{2}\int\N{v-\bbE(\rho_t)}_2^2 d\rho_t(v) = \CV(\rho_t) -\frac{1}{2}\N{\bbE(\rho_t)-\globmin}_2^2 \leq \CV(\rho_t).
%\end{align}
This relation may be a source of concern, as it shows that proving $\CV(\rho_t)\rightarrow 0$ implies $\Var{\rho_t}\rightarrow 0$.
We emphasize however that this does not imply a majorization for the corresponding time derivatives.
In fact, Example~\ref{ex:numerical_example} suggests that $\CV(\rho_t)$ can decay exponentially while $\Var{\rho_t}$ increases~initially.

\begin{example} \label{ex:numerical_example}
	We consider the Rastrigin function~$\CE(v)=v^2+2.5(1-\cos(2\pi v))$ with global minimum at~$\globmin = 0$ and various local minima, see Figure~\ref{fig:objective1d}.
	For different initial configurations $\rho_0 = \CN(\mu,0.8)$ with $\mu\in\{1,2,3,4\}$, we evolve the discretized system~\eqref{eq:dyn_micro_discrete} using $N = 320000$ agents, discrete time step size~$\Delta t = 0.01$ and parameters $\alpha = 10^{15}$ (i.e., the consensus point is the $\argmin$ of the agents), $\lambda = 1$ and $\sigma = 0.5$.
	By considering different means from $\mu = 1$ to $\mu = 4$, we push the global minimizer~$\globmin$ into the tails of the initial configuration~$\rho_0$.
	Figure~\ref{fig:comparison} shows that the decreasing initial probability mass around  $\globmin$ eventually causes the variance $\mathrm{Var}(\empmeasure{t})$ (dashed lines) to increase in the beginning of the dynamics.
	In contrast, $\CV(\empmeasure{t})$ always decays exponentially fast with convergence speed~$(2\lambda-d\sigma^2)$, independently of the initial condition~$\rho_0$.
	From a theoretical perspective, this means proving global convergence using a variance-based analysis as in Section~\ref{subsec:variance_based_analysis} must require assumptions about $\rho_0$ such as Condition~\eqref{eq:localization_carrillo}, whereas using $\CV(\rho_t)$ does not suffer from this issue.
	The convergence speed $(2\lambda-d\sigma^2)$ coincides with the result in Theorem~\ref{thm:global_convergence_main}.
	\begin{figure}[!ht]
		\centering
		\subcaptionbox{The Rastrigin function and the map $v\mapsto\N{v-\globmin}_2^2$\label{fig:objective1d}}{\includegraphics[height=5.5cm, width=0.344\textwidth, trim=120 228 124 244,clip]{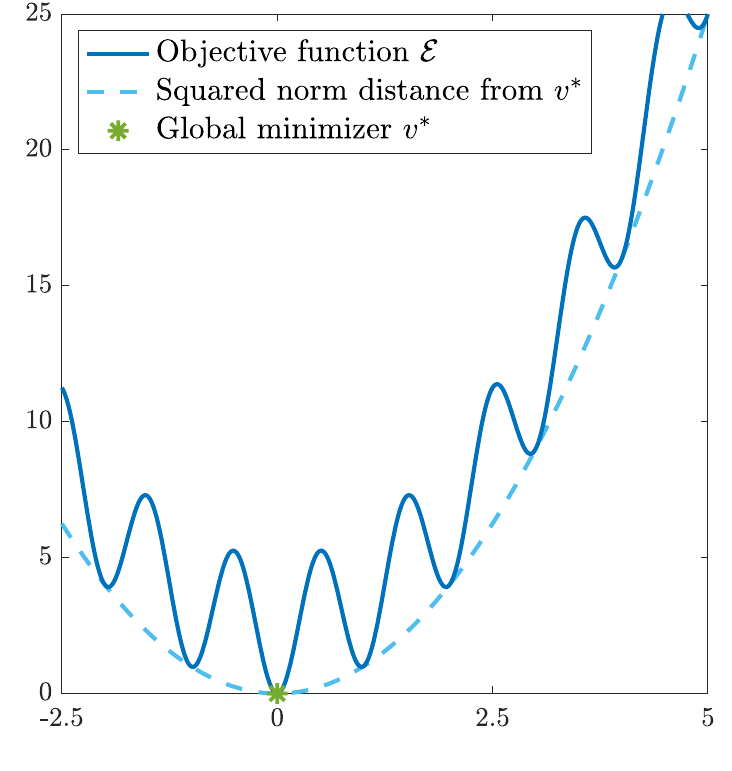}}
		\hspace{3em}
		\subcaptionbox{Evolution of the variance~$\mathrm{Var}(\empmeasure{t})$ and $\CV(\empmeasure{t})$ for different initial conditions~$\rho_0$\label{fig:comparison}}{\includegraphics[height=5.5cm, width=0.516\textwidth, trim=50 228 40 244,clip]{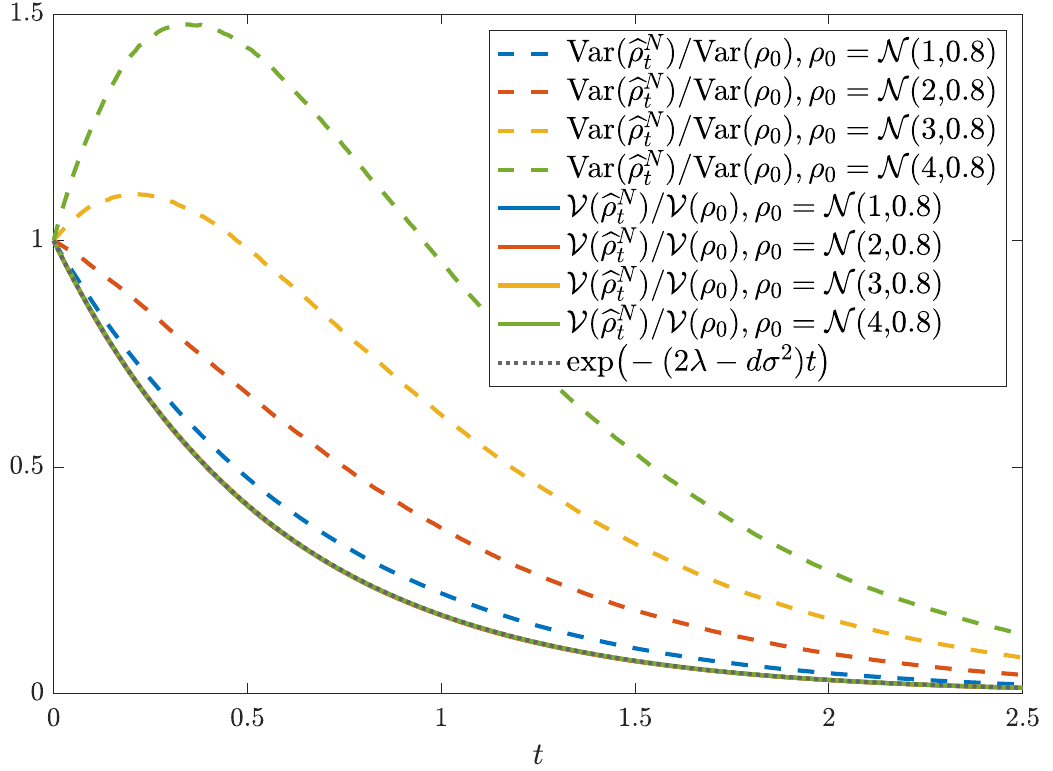}}
		\caption{ \textbf{(a)}~The Rastrigin function as objective function~$\CE$ and the squared Euclidean distance from $v^*$.
		\textbf{(b)}~The evolution of the variance~$\mathrm{Var}(\empmeasure{t})$ and the functional~$\CV(\empmeasure{t})$ for different initial conditions~$\rho_0 = \CN(\mu, 0.8)$ with $\mu\in\{1,2,3,4\}$.
		The measure~$\empmeasure{t}$ is the empirical agent density that is evolved using~\eqref{eq:dyn_micro_discrete} with $N = 320000$ agents, discrete time step size~$\Delta t=0.01$ and parameters~$\alpha = 10^{15}$, $\lambda = 1$ and $\sigma = 0.5$.
		As we move the mean of the initial configuration~$\rho_0$ away from the global optimizer $\globmin = 0$, and thereby push $v^*$ into the tails of~$\rho_0$, $\mathrm{Var}(\empmeasure{t})$ increases in the starting phase of the dynamics. $\CV(\empmeasure{t})$ on the other hand always decreases exponentially at a rate $(2\lambda-d\sigma^2)$, independently of the initial condition~$\rho_0$.}
	\label{figure:VarianceVSJ}
	\end{figure}
\end{example}

\section{Global convergence of consensus-based optimization} \label{sec:cbo_envelope_min}
In the first part of this section we recite and extend well-posedness results about the nonlinear macroscopic SDE~\eqref{eq:dyn_macro}, respectively, the associated Fokker-Planck equation~\eqref{eq:fokker_planck}.
At the beginning of the second part we introduce the class of studied objective functions, which is followed by the presentation of the main result about the convergence of the dynamics~\eqref{eq:dyn_macro} and~\eqref{eq:fokker_planck} to the global minimizer in mean-field law.
%In the final part of this section we present a holistic convergence proof of the numerical scheme~\eqref{eq:dyn_micro_discrete} by combining the latter statement with a probabilistic result about the mean-field approximation.
In the final part we then highlight the relevance of this result by presenting a holistic convergence proof of the numerical scheme~\eqref{eq:dyn_micro_discrete} to the global minimizer.
This combines the latter statement with a probabilistic quantitative  result about the mean-field approximation.

\subsection{Definition of weak solutions and well-posedness} \label{subsec:well_posedness}
We begin by rigorously defining weak solutions of the Fokker-Planck equation~\eqref{eq:fokker_planck}.

\begin{definition} \label{def:fokker_planck_weak_sense}
	Let $\rho_0 \in \CP(\bbR^d)$, $T > 0$.
	We say $\rho\in\CC([0,T],\CP(\bbR^d))$ satisfies the Fokker-Planck equation~\eqref{eq:fokker_planck} with initial condition $\rho_0$ in the weak sense in the time interval $[0,T]$, if we have for all $\phi \in \CC_c^{\infty}(\bbR^d)$ and all $t \in (0,T)$
	\begin{align} \label{eq:weak_solution_identity}
	\begin{split}
		\frac{d}{dt}\int \phi(v) \,d\rho_t(v) =
		&- \lambda\int \cutoff{\CE(v)-\CE(\conspoint{\rho_t})} \left\langle v - \conspoint{\rho_t}, \nabla \phi(v) \right \rangle d\rho_t(v)\\
		&+ \frac{\sigma^2}{2} \int \N{v-\conspoint{\rho_t}}_2^2 \Delta \phi(v) \,d\rho_t(v)
	\end{split}
	\end{align}
	and $\lim_{t\rightarrow 0}\rho_t = \rho_0$ pointwise.
\end{definition}

If the cutoff function $\cutoffnoarg$ in the dynamics~\eqref{eq:dyn_macro} is inactive, i.e., satisfies $\cutoffnoarg \equiv 1$, the authors of \cite{carrillo2018analytical} prove the following well-posedness result.

\begin{theorem}[{\cite[Theorems~3.1, 3.2]{carrillo2018analytical}}] \label{thm:well-posedness_FP}
	Let $T > 0$, $\rho_0 \in \CP_4(\bbR^d)$.
	Let $\cutoffnoarg \equiv 1$ and consider $\CE : \bbR^d\rightarrow \bbR$ with $\minobj > -\infty$, which, for constants $C_1,C_2 > 0$, satisfies
	\begin{align}
		\SN{\CE(v)-\CE(w)}
		&\leq C_1(\N{v}_2 + \N{w}_2)\N{v-w}_2,
		\quad \text{for all } v,w \in \bbR^d,
		\label{eq:lipschitz_condition_car}\\
		\CE(v) - \minobj
		&\leq C_2 (1+\N{v}_2^2),
		\quad \text{for all } v \in \bbR^d.
		\label{eq:quadratic_boundedness_condition_car}
	\end{align}
	If in addition, either $\sup_{v \in \bbR^d}\CE(v) < \infty$,
	or $\CE$ satisfies for some constants $C_3,C_4 > 0$
	\begin{align} \label{eq:quadratic_growth_condition_car}
		\CE(v) - \minobj
		\geq C_3\N{v}_2^2,
		\quad \text{for all } \N{v}_2 \geq C_4,
	\end{align}
	then there exists a unique nonlinear process $\overbar V \in \CC([0,T],\bbR^d)$ satisfying \eqref{eq:dyn_macro} \revisedTwo{in the strong sense.}
	The associated law $\rho = \Law(\overbar V)$ has regularity $\rho \in \CC([0,T], \CP_4(\bbR^d))$ and is a weak solution to the Fokker-Planck equation~\eqref{eq:fokker_planck}.
\end{theorem}

\begin{remark}
\label{rem:test_functions_redefine}
The regularity $\rho \in \CC([0,T], \CP_4(\bbR^d))$ stated in Theorem~\ref{thm:well-posedness_FP}, and also obtained in Theorem~\ref{thm:well-posedness_FP_cutoff} below,
is a consequence of the regularity of the initial condition $\rho_0 \in \CP_4(\bbR^d)$.
Despite not indicated explicitly in~\cite[Theorems~3.1, 3.2]{carrillo2018analytical}, it follows from their proofs.
In particular, it allows for extending the test function space $\CC^{\infty}_{c}(\bbR^d)$ in Definition~\ref{def:fokker_planck_weak_sense}.
%as can be seen in the proof sketch of Theorem~\ref{thm:well-posedness_FP_cutoff}.
Namely, if $\rho \in \CC([0,T], \CP_4(\bbR^d))$ solves \eqref{eq:fokker_planck} in the weak sense, Identity~\eqref{eq:weak_solution_identity} holds for all $\phi \in \CC^2(\bbR^d)$ with
% \begin{enumerate}[label=(\roman*),labelsep=10pt,leftmargin=35pt]
% 	\item $\sup_{v\in \bbR^d}|\Delta \phi(v)| < \infty$,
% 	\item $\Nnormal{\nabla \phi(v)}_2 \leq C(1+\Nnormal{v}_2)$ \ for some $C > 0$ and for all $v \in \bbR^d$.
% \end{enumerate}
(i) $\sup_{v\in \bbR^d}|\Delta \phi(v)| < \infty$, and (ii) $\Nnormal{\nabla \phi(v)}_2 \leq C(1+\Nnormal{v}_2)$ \ for some $C > 0$ and for all $v \in \bbR^d$.
We denote the corresponding function space by $\CC^2_*(\bbR^d)$.
\end{remark}

Under minor modifications of the proof for Theorem~\ref{thm:well-posedness_FP},
we can extend the existence of solutions to an active Lipschitz-continuous cutoff function $\cutoffnoarg$.

\begin{theorem}
\label{thm:well-posedness_FP_cutoff}
Let $\cutoffnoarg\not\equiv1$ be $L_H$-Lipschitz continuous. Then, under the assumptions of Theorem~\ref{thm:well-posedness_FP}, there exists a nonlinear process $\overbar V \in \CC([0,T],\bbR^d)$ satisfying \eqref{eq:dyn_macro} \revisedTwo{in the strong sense.} The associated law $\rho = \Law(\overbar V)$ has regularity $\rho \in \CC([0,T], \CP_4(\bbR^d))$ and is a weak solution to the Fokker-Planck equation~\eqref{eq:fokker_planck}.
\end{theorem}

\begin{proof}[Proof sketch]
	The proof is based on the Leray-Schauder fixed point theorem and follows the steps taken in~\cite[Theorems~3.1, 3.2]{carrillo2018analytical}.

	\noindent\textbf{Step 1:} For a given function $u\in\mathcal{C}([0,T],\mathbb{R}^d)$ and an initial measure~$\rho_0\in\CP_4(\bbR^d)$, according to standard SDE theory~\cite[Chapter~6]{arnold1974stochasticdifferentialequations}, we can uniquely solve the auxiliary SDE
		\begin{align*}
			d\widetilde{V}_t = -\lambda\big(\widetilde{V}_t-u_t\big)\cutoffnoarg\big(\mathcal{E}(\widetilde{V}_t)-\mathcal{E}(u_t)\big)\,dt + \sigma\Nbig{\widetilde{V}_t-u_t}_2\,dB_t \quad \text{with}\quad \widetilde{V}_0 \sim\rho_0,
		\end{align*}
		as the coefficients are locally Lipschitz continuous and have at most linear growth, due to the assumptions on $\CE$ and $\cutoffnoarg$.
		This induces $\widetilde\rho_t=\Law(\widetilde{V}_t)$.
		Moreover, the regularity of the initial distribution $\rho_0 \in \CP_4(\bbR^d)$ allows for a fourth-order moment estimate of the form~$\mathbb{E}\Nnormal{\widetilde{V}_t}_2^4\leq \big(1+\mathbb{E}\Nnormal{\widetilde{V}_0}_2^4\big)e^{ct}$, see, e.g.\@~\cite[Chapter~7]{arnold1974stochasticdifferentialequations}.
		So, in particular, $\widetilde\rho\in\CC([0,T],\CP_4(\bbR^d))$.

		\noindent\textbf{Step 2:} Let us now define, for some constant $C>0$, the test function space
		\begin{align}
		\label{eq:function_space}
			\CC^2_{*}(\bbR^d):=\big\{\phi\in\CC^2(\bbR^d):\Nnormal{\nabla\phi(v)}_2\leq C(1+\Nnormal{v}_2) \text{ and } \sup_{v\in \bbR^d}|\Delta \phi(v)| < \infty\big\}.
		\end{align}
		For some $\phi\in\CC^2_{*}(\bbR^d)$, by It\^o's formula, we derive
		\begin{align*}
		\begin{split}
			d\phi(\widetilde{V}_t) &= \nabla\phi(\widetilde{V}_t)\cdot\big(-\lambda\big(\widetilde{V}_t-u_t\big)\cutoffnoarg\big(\mathcal{E}(\widetilde{V}_t)-\mathcal{E}(u_t)\big)\,dt + \sigma\Nbig{\widetilde{V}_t-u_t}_2\,dB_t\big)\\
			&\qquad\quad+\frac{1}{2}\sigma^2\Delta\phi(\widetilde{V}_t)\Nbig{\widetilde{V}_t-u_t}_2^2\,dt.
		\end{split}
		\end{align*}
		After taking the expectation, applying Fubini's theorem and observing that the stochastic integral~$\bbE\int_0^t\nabla\phi(\widetilde{V}_t)\cdot\Nnormal{\widetilde{V}_t-u_t}_2\,dB_t$ vanishes as a consequence of \cite[Theorem 3.2.1(iii)]{oksendal2013stochastic} due to the established regularity $\widetilde\rho\in\CC([0,T],\CP_4(\bbR^d))$ and $\phi\in\CC^2_{*}(\bbR^d)$, we obtain
		\begin{align*}
			\frac{d}{dt}\bbE\phi(\widetilde{V}_t) = -\lambda\bbE\nabla\phi(\widetilde{V}_t)\cdot\big(\widetilde{V}_t-u_t\big)\cutoffnoarg\big(\mathcal{E}(\widetilde{V}_t)-\mathcal{E}(u_t)\big) + \frac{\sigma^2}{2}\mathbb{E}\Delta\phi(\widetilde{V}_t)\Nbig{\widetilde{V}_\tau-u_t}_2^2
		\end{align*}
		according to the fundamental theorem of calculus. This shows that
		the measure~$\widetilde\rho\in\CC([0,T],\CP_4(\bbR^d))$ satisfies the Fokker-Planck equation
		\begin{align} \label{proof:auxiliary_Fokker-Planck}
		\begin{aligned}
			\frac{d}{dt}\int\phi(v)\,d\widetilde\rho_t(v) =& -\lambda\int\cutoff{\CE(v)-\CE(u_t)}\left\langle v-u_t,\nabla\phi(v)\right\rangle d\widetilde\rho_t(v) \\
			&+ \frac{\sigma^2}{2}\int\N{v-u_t}_2^2\Delta\phi(v)\,d\widetilde\rho_t(v).
		\end{aligned}
		\end{align}
		\noindent The remainder is identical to the cited paper and is summarized briefly for completeness.

		\noindent\textbf{Step 3:} Setting $\CT u:=\conspoint{\widetilde\rho}\in\mathcal{C}([0,T],\mathbb{R}^d)$ provides the self-mapping property of the map
		\begin{align*}
			\CT:\CC([0,T],\mathbb{R}^d)\rightarrow\CC([0,T],\bbR^d), \quad u\mapsto\mathcal{T}u=\conspoint{\widetilde\rho},
		\end{align*}
		which is compact as a consequence of both a stability estimate for the consensus point showing that $\Nnormal{\conspoint{\widetilde\rho_t}-\conspoint{\widetilde\rho_s}}_2 \lesssim W_2(\widetilde\rho_t,\widetilde\rho_s)$ for $\widetilde\rho_t,\widetilde\rho_s\in\CP_4(\bbR^d)$~\cite[Lemma~3.2]{carrillo2018analytical} and the \mbox{H\"{o}lder-$1/2$} continuity of the Wasserstein-$2$ distance~$W_2(\widetilde\rho_t,\widetilde\rho_s)$.

		\noindent\textbf{Step 4:} Finally, for $u=\vartheta\CT u$ with $\vartheta\in[0,1]$, there exists $\rho\in\CC([0,T],\CP_4(\bbR^d))$ satisfying~\eqref{proof:auxiliary_Fokker-Planck} such that $u=\vartheta \conspoint{\rho}$, for which a uniform bound can be obtained either due to the boundedness or the growth condition of~$\CE$.
		An application of the Leray-Schauder fixed point theorem concludes the proof by providing a solution to~\eqref{eq:dyn_macro}.
\end{proof}

\subsection{Global convergence in mean-field law} \label{subsec:main_result}
We now present the main result about global convergence in mean-field law for objectives satisfying the following.

\begin{definition}[Assumptions] \label{def:assumptions}
	Throughout we are interested in objective functions $\CE \in \CC(\bbR^d)$, for which
	\begin{enumerate}[label=A\arabic*,labelsep=10pt,leftmargin=35pt]
		\item\label{asm:zero_global} there exists $v^*\in\bbR^d$ such that $\CE(\globmin)=\inf_{v\in\bbR^d} \CE(v)=:\underbar{\CE}$, and
		\item\label{asm:icp} there exist $\CE_{\infty},R_0,\eta > 0$, and $\nu \in (0,\infty)$ such that
		\begin{align}
			\label{eq:asm_icp_vstar}
			\N{v-v^*}_2 &\leq (\CE(v)-\minobj)^{\nu}/\eta \quad \text{ for all } v \in B_{R_0}(\globmin),\\
			\label{eq:asm_icp_farfield}
			\CE(v)-\minobj &> \CE_{\infty} \quad \text{ for all } v \in \big(B_{R_0}(\globmin)\big)^c.
		\end{align}
	\end{enumerate}
	Furthermore, for the case $H\not\equiv 1$, we additionally require that $\CE$ fulfills a local Lipschitz continuity-like condition, i.e.,
	\begin{enumerate}[label=A\arabic*,labelsep=10pt,leftmargin=35pt]
		\setcounter{enumi}{2}
		\item\label{asm:local_lipschitz} there exist $L_{\CE} > 0$ and $\gamma \geq 0$ such that
		\begin{align}
			\CE(v) - \minobj \leq L_{\CE}(1+\N{v-\globmin}_2^{\gamma})\N{v-\globmin}_2 \quad \text{ for all } v \in\bbR^d.
		\end{align}
	\end{enumerate}
\end{definition}

\begin{remark} \label{rem:ICP_E}
	The analyses in \cite{carrillo2018analytical} and related works require $\CE \in \CC^2(\bbR^d)$ and an additional boundedness assumptions on the Laplacian $\Delta \CE$.
	We relax these regularity requirements and use the conditions \revisedTwo{in Definition~\ref{def:assumptions} on $\CE$ instead.}
\end{remark}

Assumption~\ref{asm:zero_global} just states that the continuous objective $\CE$ attains its infimum $\minobj$ at some $\globmin \in \bbR^d$.
The continuity itself can be further relaxed at the cost of additional technical details because it is  only required in a small neighborhood of $\globmin$.

Assumption~\ref{asm:icp} should be interpreted as a tractability condition of the landscape of $\CE$ around $\globmin$ and in the farfield.
The first part, Equation~\eqref{eq:asm_icp_vstar}, describes  the local coercivity of $\CE$, which implies that there is a unique minimizer $\globmin$ on $B_{R_0}(\globmin)$ and that $\CE$ grows like  $v\mapsto \Nnormal{v-\globmin}_2^{1/\nu}$.
This condition is also known as the inverse continuity condition from \cite{fornasier2020consensus_sphere_convergence}, as a quadratic growth condition in the case $\nu= 1/2$ from \cite{anitescu2000degenerate, necoara2019linear}, or as the H\"olderian error bound condition in the case $\nu \in (0,1]$ \cite{bolte2017error}.
In \cite[Theorem~4]{necoara2019linear} and \cite[Theorem~2]{karimi2016linear} many equivalent or stronger conditions are identified to imply Equation~\eqref{eq:asm_icp_vstar} globally on $\bbR^d$.
Furthermore, in \cite{xu2017adaptive,fornasier2020consensus_sphere_convergence}, \eqref{eq:asm_icp_vstar} is shown to hold globally for objectives related to various machine learning problems.
The second part of \ref{asm:icp}, Equation~\eqref{eq:asm_icp_farfield}, describes the behavior of $\CE$ in the farfield and prevents $\CE(v) \approx \minobj$ for some $v \in \bbR^d$ far away from $\globmin$.
We introduce it for the purpose of covering functions that tend to a constant just above $\CE_{\infty}$ as $\Nnormal{v}_{2}\!\rightarrow\!\infty$, because such functions do not satisfy the growth condition \eqref{eq:asm_icp_vstar} globally.
However, whenever \eqref{eq:asm_icp_vstar} holds globally, we take $R_0 = \infty$, i.e.,  $B_{R_0}(\globmin) = \bbR^d$ and \eqref{eq:asm_icp_farfield} is void.
We also note that \eqref{eq:asm_icp_vstar} and \eqref{eq:asm_icp_farfield} imply the uniqueness of the global minimizer $\globmin$ on $\bbR^d$.

Finally, to cover the active cutoff case $H\not\equiv 1$, we additionally require \ref{asm:local_lipschitz}.
The condition is weaker than local Lipschitz-continuity on any compact ball around $\globmin$, with Lipschitz constant growing with the size of the ball.

% \begin{example} \label{ex:example_function}
% 	A prototypical objective function that satisfies the assumptions in Definition~\ref{def:assumptions} is the Rastrigin function, which we already discussed in Example~\ref{ex:numerical_example}, see also Figures~\ref{fig:objective2d} and \ref{fig:objective1d}.
% 	In particular, it satisfies \eqref{eq:asm_icp_vstar} globally with $\nu = 1/2$.
% %	Namely, the Rastrigin function is locally Lipschitz continuous on any compact ball around $\globmin$ and it is minorized by an appropriately scaled quadratic polynomial $v \mapsto \N{v-\globmin}_2^2$.
% %	Hence, it satisfies \eqref{eq:asm_icp_vstar} globally with $\nu = 1/2$.
% \end{example}

We are now ready to state the main result. The proof is deferred to Section~\ref{sec:proof_main_theorem}.

\begin{theorem} \label{thm:global_convergence_main}
	Let $\CE \in \CC(\bbR^d)$ satisfy \ref{asm:zero_global}--\ref{asm:icp}.
	Moreover, let $\rho_0 \in \CP_4(\bbR^d)$ be such that $\globmin\in\supp(\rho_0)$.
	% \begin{align} \label{eq:condition_initial_measure}
	% 	%\rho_0(B_{r}(\globmin)) > 0\quad \text{ for all } r > 0.
 %        \globmin\in\supp(\rho_0).
	% \end{align}
	%Define the functional $\CV(\rho_t) := 1/2 \int \Nnormal{v-\globmin}_2^2 \,d\rho_t(v)$.
    Define $\CV(\rho_t)$ as given in \eqref{def:J}.
	Fix any $\varepsilon \in (0,\CV(\rho_0))$ and $\vartheta \in (0,1)$,
    \revised{choose} parameters $\lambda,\sigma > 0$ with $2\lambda > d \sigma^2 $, and
    \revised{define} the time horizon
	\begin{align} \label{eq:end_time_star_statement}
		T^* := \frac{1}{(1-\vartheta)\big(2\lambda-d\sigma^2\big)}\log\left(\frac{\CV(\rho_0)}{\varepsilon}\right).
	\end{align}
	Then there exists $\alpha_0 > 0$, depending (among problem dependent quantities) on $\varepsilon$ and $\vartheta$, such that for all $\alpha > \alpha_0$, if $\rho \in \CC([0,T^*], \CP_4(\bbR^d))$ is a weak solution to the Fokker-Planck equation~\eqref{eq:fokker_planck} on the time interval $[0,T^*]$ with initial condition $\rho_0$,
    %we have $\min_{t \in [0,T^*]}\CV(\rho_t) \leq \varepsilon$.
    we have %$\CV(\rho_T) = \varepsilon$ with
    \begin{align}
        \label{eq:end_time_statement}
        \revised{\CV(\rho_T) = \varepsilon
        \quad\text{ with }\quad
        T\in\left[\frac{1-\vartheta}{(1+\vartheta/2)}\;\!T^*,T^*\right].}
    \end{align}
	Furthermore, \revised{on the time interval $[0,T]$, $\CV(\rho_t)$ decays at least exponentially fast. More precisely, for all $t\in[0,T]$, it holds}
	%\begin{align*}
		%\CV(\rho_t) \leq \CV(\rho_0) \exp\left(-(1-\vartheta)\left(2\lambda- d\sigma^2\right) t\right)
	%\end{align*}
	%and
	%\begin{align*}
	%	W_2^2(\rho_t,\delta_{\globmin}) \leq 2\CV(\rho_0) \exp\left(-(1-\vartheta)\left(2\lambda- d\sigma^2\right) t\right).
	%\end{align*}
	\begin{align} \label{eq:thm:global_convergence_main:V}
		W_2^2(\rho_t,\delta_{\globmin}) = 2 \CV(\rho_t) \leq 2 \CV(\rho_0) \exp\left(-(1-\vartheta)\left(2\lambda- d\sigma^2\right) t\right).
        %\quad\text{ for all } t\in[0,T].
	\end{align}
	If $\CE$ additionally satisfies \ref{asm:local_lipschitz}, the same conclusion holds for any $H: \bbR^d\rightarrow [0,1]$ that satisfies $H(x) = 1$ whenever $x\geq 0$.
\end{theorem}

% \begin{remark}
% 	Theorem~\ref{thm:global_convergence_main} is valid for any~$\rho \in \CC([0,T^*], \CP_4(\bbR^d))$ weakly solving the Fokker-Planck equation~\eqref{eq:fokker_planck} with initial datum~$\rho_0$.
% 	Sufficient conditions for the existence of such~$\rho$ are provided by the assumptions of Theorems~\ref{thm:well-posedness_FP} and~\ref{thm:well-posedness_FP_cutoff}, respectively.
% \end{remark}

%Theorem~\ref{thm:global_convergence_main} proves an exponential convergence to the global minimizer of $\CE \in \CC(\bbR^d)$ in mean-field law under Assumptions~\ref{asm:zero_global}--\ref{asm:icp} and the minimalistic assumption~\eqref{eq:condition_initial_measure} about the initial configuration $\rho_0$.
%The latter assumption is by no means a restriction, as it would anyhow hold immediately for $\rho_t$ for any $t>0$ in view of the diffusive character of the dynamics~\eqref{eq:fokker_planck}.
The assumption $\globmin\in\supp(\rho_0)$ about the initial configuration~$\rho_0$ is \revisedTwo{not really} a restriction, as it would anyhow hold immediately for $\rho_t$ for any $t>0$ in view of the diffusive character of the dynamics~\eqref{eq:fokker_planck}, see Remark~\ref{remark:lem:lower_bound_probability}.
Additionally, as we clarify in the next section, this condition does neither mean nor require that, for finite particle approximations, some particle needs to be in the vicinity of the minimizer~$\globmin$ at time~$t=0$.
It is actually sufficient that the empirical measure~$\empmeasure{t}$ weakly approximates the law~$\rho_t$ uniformly in time. We rigorously explain this mechanism in Section~\ref{subsec:convergence_probability}.

A lower bound on the rate of convergence in~\eqref{eq:thm:global_convergence_main:V} is $(1-\vartheta)(2\lambda-d\sigma^2)$, which can be made arbitrarily close to the numerically observed rate $(2\lambda-d\sigma^2)$ (see, e.g., Figure~\ref{fig:comparison}) at the cost of taking $\alpha \rightarrow \infty$ to allow for $\vartheta \rightarrow 0$.
The condition $2\lambda > d\sigma^2$ is necessary, both in theory and practice, to avoid overwhelming the dynamics by the random exploration term.
The dependency on $d$ can be eased by replacing the isotropic Brownian motion in the dynamics with an anisotropic one~\cite{carrillo2019consensus,fornasier2021convergence}.

\subsection{Global convergence  in probability} \label{subsec:convergence_probability}

To stress the relevance of the main result of this paper, Theorem~\ref{thm:global_convergence_main}, we now show how Estimate~\eqref{eq:thm:global_convergence_main:V} plays a fundamental role in establishing a quantitative convergence result for the numerical scheme~\eqref{eq:dyn_micro_discrete} to the global minimizer~$\globmin$.
%With Theorem~\ref{thm:full_global_convergence} below we provide the first, and so far unique, holistic proof of convergence of the CBO method on the plane~$\bbR^d$.
By paying the price of having a probabilistic statement about the convergence of CBO as in Theorem~\ref{thm:full_global_convergence}, we gain provable polynomial complexity.
For simplicity, we present the results of this section for the case of an inactive cutoff function, i.e., $H\equiv1$.

\begin{theorem} \label{thm:full_global_convergence}
	Fix $\varepsilon_{\mathrm{total}}>0$ and $\delta \in(0,1/2)$.
	Then, under the assumptions of Theorem~\ref{thm:global_convergence_main} and Proposition~\ref{prop:MFL}, \revised{and with $K:=T/ \Delta t$, where $T$ is as in \eqref{eq:end_time_statement},} the iterations~$((V_{k\Delta t}^i)_{k=0,\ldots,K})_{i=1,\ldots,N}$ generated by the numerical scheme~\eqref{eq:dyn_micro_discrete} converge in probability to $\globmin$. More precisely, the empirical mean of the final iterations fulfills %the quantitative error estimate
	\begin{align} \label{eq:thm:full_global_convergence:error}
		\N{\frac{1}{N}\sum_{i=1}^N V_{K\Delta t}^i - \globmin}_2^2 
		\leq \varepsilon_{\mathrm{total}}
	\end{align}
	with probability larger than $1-\left(\delta + \varepsilon_{\mathrm{total}}^{-1}(6C_{\mathrm{NA}}(\Delta t)^{2m} + 3C_{\mathrm{MFA}}N^{-1} + 12\varepsilon)\right)$.
	Here, $m$ denotes the order of accuracy of the numerical scheme (for the Euler-Maruyama scheme $m=1/2$) and $\varepsilon$ is the error from Theorem~\ref{thm:global_convergence_main}.
	Moreover, besides problem-dependent constants, $C_{\mathrm{NA}}>0$ depends linearly on the dimension~$d$ and the number of particles~$N$, exponentially on the time horizon~$T$, and on $\delta^{-1}$; $C_{\mathrm{MFA}}>0$ depends exponentially on the parameters~$\alpha$, $\lambda$ and $\sigma$, on $T$, and on $\delta^{-1}$.
\end{theorem}
%
%\begin{proof}[Proof sketch]
%	We have the error decomposition
%	\begin{align*} %\label{eq:proof:thm:full_global_convergence:error_decomposition}
%	\begin{split}
%		&\bbE\left[\N{\frac{1}{N}\sum_{i=1}^N V_{K\Delta t}^i - \globmin}_2^2 \Bigg|\; \Omega_M\right]
%		\lesssim
%			\frac{1}{1-\delta}\,\bbE\N{\frac{1}{N}\sum_{i=1}^N \left(V_{K\Delta t}^i - V_{T^*}^i\right)}_2^2 \\
%			&\qquad\qquad\qquad\,+\bbE\left[\N{\frac{1}{N}\sum_{i=1}^N \left(V_{T^*}^i - \overbar{V}_{T^*}^i\right)}_2^2 \Bigg|\; \Omega_M\right]
%			+\frac{1}{1-\delta}\,\bbE\N{\frac{1}{N}\sum_{i=1}^N \overbar{V}_{T^*}^i - \globmin}_2^2.
%		%\leq
%		%C_1(\Delta t)^{2m} + C_2N^{1-} + \varepsilon,
%	\end{split}
%	\end{align*}
%	%which divides the overall error into an approximation error of the numerical scheme, the mean-field approximation error and the optimization error in the mean-field limit.
%	By combining a classical result about the convergence of numerical approximation schemes~\cite{platen1999introduction} with the mean-field approximation in form of Proposition~\ref{prop:MFL}, and the convergence result in mean-field law, Theorem~\ref{thm:global_convergence_main}, we obtain the statement.
%\end{proof}

\revised{Let us briefly discuss in the following remark the computational complexity of the numerical scheme~\eqref{eq:dyn_micro_discrete} together with some implementational aspects which allow to reduce the overall runtime of the algorithm in practice.}

\revised{\begin{remark}[Computational complexity]
    To achieve 
    %an accuracy of $\varepsilon_{\mathrm{total}}$ as in
    Estimate \eqref{eq:thm:full_global_convergence:error}
    with probability of at least $(1-2\delta)$, the implementable CBO scheme~\eqref{eq:dyn_micro_discrete} has to be run using
    % \begin{equation*}
    %     N \geq \frac{9C_{\mathrm{MFA}}}{\delta\varepsilon_{\mathrm{total}}}
    % \end{equation*}
    \mbox{$N \geq 9C_{\mathrm{MFA}}/(\delta\varepsilon_{\mathrm{total}})$} agents
    and with
    time step size $\Delta t \leq \!\!\sqrt[2m]{\delta\varepsilon_{\mathrm{total}}/(18 C_{\mathrm{NA}})}$
    % \begin{equation*}
    %     \Delta t \leq \sqrt[2m]{\frac{\delta\varepsilon_{\mathrm{total}}}{18 C_{\mathrm{NA}}}}
    % \end{equation*}
    for 
    \begin{equation*}
        K \geq \frac{1}{(1-\vartheta)\big(2\lambda-d\sigma^2\big)}\frac{1}{\Delta t}\log\left(\frac{36\CV(\rho_0)}{\delta\varepsilon_{\mathrm{total}}}\right)
    \end{equation*}
    iterations.
    Here, the parameter dependence of $C_{\mathrm{NA}}$ and $C_{\mathrm{MFA}}$ is as described in Theorem~\ref{thm:full_global_convergence}.
    The computational complexity (counted in terms of the number of evaluations of the objective~$\CE$) of the CBO method is therefore given by $\CO(KN)$.
    % \begin{equation*}
    %     \CO\left(\frac{1}{(1-\vartheta)\big(2\lambda-d\sigma^2\big)} \frac{C_{\mathrm{MFA}}}{\delta\varepsilon_{\mathrm{total}}} \sqrt[2m]{\frac{C_{\mathrm{NA}}}{\delta\varepsilon_{\mathrm{total}}}}\log\left(\frac{36\CV(\rho_0)}{\delta\varepsilon_{\mathrm{total}}}\right) \right)
    % \end{equation*}

    When working in the setting of large-scale applications arising, for instance, in machine learning and signal processing (therefore, with $\CE$ being expensive to compute), several considerations allow to reduce the overall runtime of the algorithm~\eqref{eq:dyn_micro_discrete} and thereby make the method feasible and more competitive.
    First of all, it may be recommendable to leverage that the evaluations of the objective function~$\CE$ for each of the $N$ particles can be performed in parallel.
    Furthermore, random mini-batch sampling ideas as proposed in \cite{carrillo2019consensus,fornasier2021convergence} may be employed when evaluating the objective function and/or computing the consensus point.
    I.e., at each time step, $\CE$ is evaluated only on a random subset of the available data, and $\conspointnoarg$ is computed only from a subset of the $N$ particles.
    Besides immediately reducing the computational and communication complexity of CBO methods, such ideas motivate communication-efficient parallelization of the algorithm by evolving disjoint subsets of particles independently for some time with separate consensus points, before aligning the dynamics through a global communication step.
    This, however, is so far largely unexplored, both from a theoretical and practical point of view.
    Lastly, taking inspiration from genetic algorithms, a variance-based particle reduction technique as suggested in \cite{fornasier2020consensus_sphere_convergence} may be used to reduce the number of optimizing agents (and therefore the required evaluations of $\CE$) during the algorithm in case concentration of the particles is observed.
\end{remark}}

The proof of Theorem~\ref{thm:full_global_convergence}, which we report below, combines our main result about the convergence in mean-field law, a quantitative mean-field approximation and classical results of numerical approximation of SDEs.
To this end, we establish in what follows the result about the quantitative mean-field approximation on a restricted set of bounded processes.
For this purpose, let us introduce the common probability space~$(\Omega, \CF, \bbP)$ over which all considered stochastic processes get their realizations, and define a subset $\Omega_M$ of $\Omega$ of suitably bounded  processes according to
\begin{align*}
    \textstyle
	\Omega_M:=\left \{\omega \in \Omega: \sup_{t\in[0,T]} \frac{1}{N}\sum_{i=1}^N \max\left\{\N{V_t^i(\omega)}_2^4,\N{\overbar{V}_t^i(\omega)}_2^4\right\} \leq M  \right \}.
\end{align*}
Throughout this section, $M>0$ denotes a constant which we shall adjust at the end of the proof of Theorem~\ref{thm:full_global_convergence}. Before stating the mean-field approximation result, Proposition~\ref{prop:MFL}, let us estimate the measure of the set~$\Omega_M$ in Lemma~\ref{lem:PBND}.
The proofs of both statements are deferred to Section~\ref{sec:proof_probabilistic_MFA}.

\begin{lemma} \label{lem:PBND}
	Let $T > 0$, $\rho_0 \in \CP_{4}(\bbR^d)$ and let $N\in\bbN$ be fixed.
	Moreover, let $((V_t^i)_{t\geq0})_{i=1,\dots,N}$ denote the \revised{strong} solution to system~\eqref{eq:dyn_micro} and let $((\overbar{V}_t^i)_{t\geq0})_{i=1,\dots,N}$ be $N$ independent copies of the \revised{strong} solution to the mean-field dynamics~\eqref{eq:dyn_macro}.
	Then, under the assumptions of Theorem~\ref{thm:well-posedness_FP}, for any $M>0$ we have
	\begin{align} \label{eq:prop:MFL:OmegaM}
		\bbP\left(\Omega_M\right) =
		\bbP\left(\sup_{t\in[0,T]} \frac{1}{N}\sum_{i=1}^N \max\left\{\N{V_t^i}_2^4,\N{\overbar{V}_t^i}_2^4\right\} \leq M\right)
		\geq 1 - \frac{2K}{M},
	\end{align}
	where $K=K(\lambda,\sigma,d,T, b_1,b_2)$ is a constant, which is in particular independent of $N$.
	Here, $b_1$ and $b_2$ denote the problem-dependent constants from~\cite[Lemma~3.3]{carrillo2018analytical}.
\end{lemma}

Lemma~\ref{lem:PBND} proves that the processes are bounded with high probability uniformly in time.
Therefore, by restricting the analysis to $\Omega_M$, we can obtain the following quantitative mean-field approximation result \revised{by proving pointwise propoagation of chaos through the coupling method~\cite{chaintron2021propagation,chaintron2022propagation} using a synchronous coupling between the stochastic processes $V^i$ and $\overbar{V}^i$, see, e.g., \cite[Section~4.1.2]{chaintron2022propagation}.}

\begin{proposition} \label{prop:MFL}
	Let $T > 0$, $\rho_0 \in \CP_{4}(\bbR^d)$ and let $N\in\bbN$ be fixed.
	Moreover, let $((V_t^i)_{t\geq0})_{i=1,\dots,N}$ denote the \revised{strong} solution to system~\eqref{eq:dyn_micro} and let $((\overbar{V}_t^i)_{t\geq0})_{i=1,\dots,N}$ be $N$ independent copies of the \revised{strong} solution to the mean-field dynamics~\eqref{eq:dyn_macro}.
	Further consider valid the assumptions of Theorem~\ref{thm:well-posedness_FP}.
	If $(V_t^i)_{t\geq0}$ and $(\overbar{V}_t^i)_{t\geq0}$ share the initial data as well as the Brownian motion paths~$(B_t^i)_{t\geq0}$ for all $i=1,\dots,N$, then we have %a probabilistic mean-field approximation of the form
	\begin{align} \label{eq:prop:MFL:mean-field approximation}
		\max_{i=1,\dots,N}\sup_{t\in[0,T]}\bbE\left[\Nnormal{V^i_t-\overbar V^i_t}_2^2 \,\big|\, \Omega_M\right]\leq C_{\mathrm{MFA}}N^{-1}
	\end{align}
	with $C_{\mathrm{MFA}}=C_{\mathrm{MFA}}(\alpha,\lambda,\sigma,T,C_1,C_2,M,K,\CM_2,b_1,b_2)$, where $K$ is as in Lemma~\ref{lem:PBND} and $\CM_2$ denotes a second-order moment bound of~$\rho$.
\end{proposition}

A quantitative mean-field approximation was left as an open problem in~\cite[Remark 3.2]{carrillo2018analytical} \revisedTwo{due to a lack of global Lipschitz continuity of the SDE coefficients and approached since then in several steps, see Remark~\ref{rem:convergence_details}.
While} the restriction to bounded processes, which reflects the typical behavior in view of Lemma~\ref{lem:PBND}, \revisedTwo{already} allows to obtain an estimate of the type \eqref{eq:prop:MFL:mean-field approximation}, \revisedTwo{which is sufficient to} prove convergence in probability \revisedTwo{in what follows,
the recent work~\cite{gerber2023propagation} improves \eqref{eq:prop:MFL:mean-field approximation} by firstly showing a non-probabilistic mean-field approximation, i.e., removing the necessity of conditioning on the set $\Omega_M$ as done in \eqref{eq:prop:MFL:mean-field approximation}, and secondly by obtaining a pathwise estimate, see \cite[Theorem~2.6]{gerber2023propagation}.
Hence, in the light of \cite{gerber2023propagation}, the role of the constant $M$ can be regarded as merely an auxiliary technical tool.}

Equipped with Lemma~\ref{lem:PBND} and Proposition~\ref{prop:MFL}, we are now able to prove Theorem~\ref{thm:full_global_convergence}.

\begin{proof}[Proof of Theorem \ref{thm:full_global_convergence}]
	We have the error decomposition
	\revised{\begin{align} \label{eq:proof:thm:full_global_convergence:error_decomposition}
	\begin{split}
		&\bbE\left[\N{\frac{1}{N}\sum_{i=1}^N V_{K\Delta t}^i - \globmin}_2^2 \Bigg|\; \Omega_M\right]
		\leq
			3 \underbrace{\bbE\left[\N{\frac{1}{N}\sum_{i=1}^N \left(V_{K\Delta t}^i - V_{T}^i\right)}_2^2\Bigg|\; \Omega_M\right]}_{\!\!\!\substack{\leq C_{\mathrm{NA}}(\Delta t)^{2m} \text{ by applying classical convergence} \\ \text{results for numerical schemes for SDEs~\cite{platen1999introduction}}}\!\!\!} \\
			&\qquad\quad\,+3\underbrace{\bbE\left[\N{\frac{1}{N}\sum_{i=1}^N \left(V_{T}^i - \overbar{V}_{T}^i\right)}_2^2 \Bigg|\; \Omega_M\right]}_{\!\!\!\substack{\leq C_{\mathrm{MFA}}N^{-1} \text{ using the quantitative mean-field} \\ \text{approximation in form of Proposition~\ref{prop:MFL}}}\!\!\!}
			+\frac{3}{1-\delta} \underbrace{\bbE\N{\frac{1}{N}\sum_{i=1}^N \overbar{V}_{T}^i - \globmin}_2^2}_{\substack{\revisedTwo{\leq \bbE\N{\overbarscript{V}_{\!T}^1-\globmin\!}_2^2}\leq 2\CV(\rho_{T}) \leq 2\varepsilon \\ \text{by means of Theorem~\ref{thm:global_convergence_main}}}} \\
            &\quad\,\leq 6C_{\mathrm{NA}} (\Delta t)^{2m} + 3C_{\mathrm{MFA}} N^{-1} + 12\varepsilon
	\end{split}
	\end{align}}
	dividing the overall error into an approximation error of the numerical scheme, the mean-field approximation error and the optimization error in the mean-field limit.
 
	Denoting now by $K_{\varepsilon_\mathrm{total}}^{N}\subset\Omega$ the set, where \eqref{eq:thm:full_global_convergence:error} does not hold, we can estimate
	\begin{align*}
	\begin{split}
		\bbP\left(K_{\varepsilon_\mathrm{total}}^{N}\right)
		& = \bbP\left(K_{\varepsilon_\mathrm{total}}^{N}\cap\Omega_M\right) + \bbP\left(K_{\varepsilon_\mathrm{total}}^{N}\cap\Omega_M^c\right)\leq \bbP\left(K_{\varepsilon_\mathrm{total}}^{N}\big|\,\Omega_M\right)\bbP\left(\Omega_M\right) + \bbP\left(\Omega_M^c\right) \\
		%&\leq \bbP\left(K_{\varepsilon_\mathrm{total}}^{N}\big|\,\Omega_M\right) + \frac{2K}{M} \\
		& \leq \varepsilon_{\mathrm{total}}^{-1}\left(6C_{\mathrm{NA}} (\Delta t)^{2m} + 3C_{\mathrm{MFA}} N^{-1} + 12\varepsilon\right) + \delta,
	\end{split}
	\end{align*}
	where in the last step we employ %the conditional version of 
	Markov's inequality together with \eqref{eq:proof:thm:full_global_convergence:error_decomposition} to bound the first term.
	For the second it suffices to choose the $M$ from~\eqref{eq:prop:MFL:OmegaM} large enough.
\end{proof}

As a consequence of Theorem~\ref{thm:global_convergence_main}, the hardness of any optimization problem is necessarily encoded in the mean-field approximation.
%, or, more precisely, in the way how the empirical measure~$\empmeasurenoarg$ of the finite particle dynamics~\eqref{eq:dyn_micro} is used to approximate the mean-field limit~\eqref{eq:fokker_planck}.
Proposition~\ref{prop:MFL} addresses precisely this question, ensuring that, with arbitrarily high probability, the finite particle dynamics~\eqref{eq:dyn_micro} keeps close to the mean-field dynamics~\eqref{eq:dyn_macro}.
Since the rate of this convergence is of order $N^{-1/2}$ in the number of particles~$N$, the hardness of the problem is fully captured by the constant~$C_{\mathrm{MFA}}$ in~\eqref{eq:prop:MFL:mean-field approximation}, which does not depend explicitly on the dimension~$d$.
Therefore, the mean-field approximation is, in general, not affected by the curse of dimensionality. 
Nevertheless, as our assumptions on the objective function~$\CE$ do not exclude the class of NP-hard problems, it cannot be expected that CBO solves \textit{any} problem, howsoever hard, with polynomial complexity.
This is reflected by the exponential dependence of~$C_{\mathrm{MFA}}$ on the parameter~$\alpha$ and its possibly worst-case linear dependence on the dimension~$d$, as we discuss in what follows.
However, several numerical experiments~\cite{carrillo2019consensus,fornasier2020consensus_sphere_convergence,fornasier2021anisotropic,fornasier2021convergence} in high dimensions confirm that in typical applications CBO performs comparably to state-of-the-art methods without the necessity of an exponentially large amount of particles.
As mentioned before, characterizing $\alpha_0$ in more detail is crucial in view of the mean-field approximation result, Proposition~\ref{prop:MFL}.
We did not precisely specify $\alpha_0$ in Theorem~\ref{thm:global_convergence_main} since it seems challenging to provide informative bounds in all generality.
In Remark~\ref{rem:alpha_0}, however, we devise an informal derivation in the case $H\equiv 1$ for objectives $\CE$ that are locally $L$-Lipschitz continuous on some ball $B_R(\globmin)$ and satisfy the coercivity condition~\eqref{eq:asm_icp_vstar} globally for $\nu = 1/2$.
For a parameter-dependent constant $c = c(\vartheta,\lambda,\sigma)$, we obtain
\begin{align} \label{eq:alpha_requirement_optimistic_foreshadow}
	\alpha
		> \alpha_0
		= \frac{-8}{c^2\eta^2\varepsilon}\log\left(\frac{c}{2\sqrt{2}}\,\rho_0\left(B_{\min\{R,\,c^2\eta^2\varepsilon/(8L)\}}(\globmin)\right)\right)
\end{align}
provided that the probability mass $t\mapsto \rho_t\left(B_{c^2\eta^2 \varepsilon/(8L)}(\globmin)\right)$ is minimized at time $t = 0$.
The latter assumption is motivated by numerical observations of typical successful CBO runs, where the particle density around the global minimizer tends to be minimized initially and steadily increases over time.
We note that the argument of the $\log$ in \eqref{eq:alpha_requirement_optimistic_foreshadow} may induce a dependence of $\alpha_0$ on the ambient dimension $d$, if we do not dispose of an informative initial configuration~$\rho_0$.
For instance, if $\rho_0$ is measure-theoretically equivalent to the Lebesgue measure on a compact set in $\bbR^d$, we have $\alpha_0 \in \CO(d\log(\varepsilon)/\varepsilon)$ as $d,1/\varepsilon\rightarrow \infty$ by \eqref{eq:alpha_requirement_optimistic_foreshadow}.
If we interpreted $\rho_0$ as the uncertainty about the location of the global minimizer $\globmin$, we could thus consider low-uncertainty regimes, where $\rho_0$ actually concentrates around $\globmin$ and $\alpha_0$ may be dimension-free, or a high-uncertainty regime, where $\rho_0$ does not concentrate and $\alpha_0$ may depend on $d$.

\section{Proof details for Section~\ref{subsec:main_result}} \label{sec:proof_main_theorem}
In this section we provide the proof details for the global convergence result of CBO in mean-field law, Theorem~\ref{thm:global_convergence_main}.
In Section~\ref{subsec:proof_sketch} we give a proof sketch to outline the main steps.
Sections~\ref{subsec:evolution_convex}--\ref{subsec:lower_bound_prob_mass} provide auxiliary results, which might be of independent interest.
In Section~\ref{subsec:proof_main} we complete the proof of Theorem~\ref{thm:global_convergence_main}.
%\begin{remark} \label{rem:wlog_minobj_zero}
	%For simplicity we assume $\minobj = 0$ throughout the proofs.
	%This is without loss of generality since a constant offset to the objective $\CE$ does not change the location of the consensus point $\conspoint{\rho_t}$ or the evolution of the agents in the CBO dynamics.
%\end{remark}
Throughout we assume $\minobj = 0$, which is w.l.o.g.\@ since a constant offset to $\CE$ does not change the CBO dynamics.

\subsection{Proof sketch} \label{subsec:proof_sketch}
Let us first concentrate on the case $H \equiv 1$ and recall the definition of  the energy functional $\CV(\rho_t) = 1/2\int \N{v-\globmin}_2^2d\rho_t(v)$.
The main idea is to show that $\CV(\rho_t)$ satisfies the  differential inequality
\begin{align} \label{eq:proof_sketch_ODE_ineq}
	\frac{d}{dt}\CV(\rho_t)
	\leq -(1-\vartheta)\left(2\lambda-d\sigma^2\right)\CV(\rho_t)
\end{align}
until the time $T = T^*$ or until $\CV(\rho_T) \leq \varepsilon$.
In the case $T=T^*$, it is easy to check that the definition of $T^*$ in Theorem~\ref{thm:global_convergence_main} implies $\CV(\rho_{T^*}) \leq \varepsilon$, thus giving $\CV(\rho_T)\leq \varepsilon$ in either case.

The first step towards \eqref{eq:proof_sketch_ODE_ineq} is to derive a differential inequality for the evolution of $\CV(\rho_t)$ by using the dynamics of $\rho$.
Namely, by using the \eqref{eq:weak_solution_identity} with test function $\phi(v) = 1/2\N{v-\globmin}_2^2$, we derive in Lemma~\ref{lem:evolution_of_objective} the inequality
\begin{equation} \label{eq:proof_sketch_ODE_ineq_2}
\begin{aligned}
	\frac{d}{dt}\CV(\rho_t) \leq \,
	&-\left(2\lambda-d\sigma^2\right) \CV(\rho_t) + \sqrt{2}\left(\lambda+d\sigma^2\right) \sqrt{\CV(\rho_t)} \N{\conspoint{\rho_t}-\globmin}_2 \\
	&+ \frac{d\sigma^2}{2	} \N{\conspoint{\rho_t}-\globmin}_2^2.
\end{aligned}
\end{equation}
To find suitable bounds on the second and third term in \eqref{eq:proof_sketch_ODE_ineq_2}, we need to control the quantity~$\N{\conspoint{\rho_t}-\globmin}_2$.
This is achieved by the quantitative Laplace principle. Namely, under the inverse continuity property~\ref{asm:icp}, Proposition~\ref{lem:laplace_alt} in Section~\ref{subsec:quant_laplace} shows
\begin{align*}
	\N{\conspoint{\rho_t}-\globmin}_2\lesssim \ell(r) + \frac{\exp(-\alpha r)}{\rho_t(B_{r}(\globmin))},\quad \text{for sufficiently small } r > 0,
\end{align*}
where $\ell$ is a strictly positive but monotonically decreasing function with $\ell(r)\rightarrow 0$ as $r\rightarrow 0$.
As long as we can guarantee $\rho_t(B_{r}(\globmin)) > 0$, the choice
\begin{align*}
	\alpha > \frac{-\log(\rho_t(B_{r}(\globmin)))-\log(\ell(r))}{r}
\end{align*}
implies $\N{\conspoint{\rho_t}-\globmin}_2\lesssim \ell(r)$, which can be made arbitrarily small by suitable choices of~$r\ll 1$ and $\alpha \gg 1$.

Consequently, the last ingredient to the proof is to show $\rho_t(B_{r}(\globmin)) > 0$ for all $r > 0$.
To this end, we prove in Proposition~\ref{lem:lower_bound_probability} that the initial mass $\rho_0(B_{r}(\globmin)) > 0$ can decay at most at an exponential rate for any $r > 0$, but remains strictly positive in any finite time window $[0,T]$.
We emphasize that a key requirement to this result is an active Brownian motion term, i.e., $\sigma > 0$, that counteracts the deterministic movement of the drift term by inducing randomness.

The proof for the case $H\not \equiv 1$ differs slightly, because the evolution inequality~\eqref{eq:proof_sketch_ODE_ineq_2} contains an additional term involving $\N{\conspoint{\rho_t}-\globmin}_2$, see Lemma~\ref{lem:evolution_of_objective_H}.
The idea, however, stays the same and the remaining steps of the proof can be conducted in a similar fashion.

\subsection{Evolution of the mean-field limit} \label{subsec:evolution_convex}
We now derive evolution inequalities of the energy functional $\CV(\rho_t)$ for the cases~$H\equiv 1$ and $H\not \equiv 1$, respectively.

\begin{lemma} \label{lem:evolution_of_objective}
Let $\CE : \bbR^d \rightarrow \bbR$, $\cutoffnoarg \equiv 1$, and fix $\alpha,\lambda,\sigma > 0$.
Moreover, let $T>0$ and let $\rho \in \CC([0,T], \CP_4(\bbR^d))$ be a weak solution to the Fokker-Planck equation~\eqref{eq:fokker_planck}.
%Then the functional $\CV(\rho_t) := 1/2\int \N{v-\globmin}_2^2 d\rho_t(v)$ satisfies
Then the functional $\CV(\rho_t)$ satisfies\vspace{-0.2cm}
\begin{align}
    \label{eq:evolution_of_objective}
\begin{aligned}
    \frac{d}{dt}\CV(\rho_t)
	\leq
	&-\left(2\lambda - d\sigma^2\right) \CV(\rho_t)
    + \sqrt{2}\left(\lambda + d\sigma^2\right) \sqrt{\CV(\rho_t)} \N{\conspoint{\rho_t}-\globmin}_2 \\
    &+ \frac{d\sigma^2}{2} \N{\conspoint{\rho_t}-\globmin}_2^2.
\end{aligned}
\end{align}
\end{lemma}

\begin{proof}
	%We note that the function $v\mapsto \psi(v) := 1/2\N{v-\globmin}_2^2$ is in $\CC_*^2(\bbR^d)$ and $\rho$ satisfies the weak solution identity~\eqref{eq:weak_solution_identity} for all test functions in $\CC_*^2(\bbR^d)$, see Remark~\ref{rem:test_functions_redefine}.
	We note that the function $\phi(v) = 1/2\Nnormal{v-\globmin}_2^2$ is in $\CC_*^2(\bbR^d)$ and recall that $\rho$ satisfies the weak solution identity~\eqref{eq:weak_solution_identity} for all test functions in $\CC_*^2(\bbR^d)$, see Remark~\ref{rem:test_functions_redefine}.
	%Hence, by using \eqref{eq:weak_solution_identity} with $\phi$, the evolution of $\CV(\rho_t)$ reads
	By applying \eqref{eq:weak_solution_identity} with $\phi$ as above, we obtain for the evolution of $\CV(\rho_t)$
	\begin{align*}
		\frac{d}{dt} \CV(\rho_t)
		%&= -\lambda \int \langle \nabla \phi(v), v-\conspoint{\rho_t}\rangle \,d\rho_t(v) +  \frac{\sigma^2}{2} \int \N{v-\conspoint{\rho_t}}_2^2 \Delta \phi(v) \,d\rho_t(v)\\
		&= \underbrace{-\lambda \int \langle  v-\globmin, v-\conspoint{\rho_t}\rangle \,d\rho_t(v)}_{=:T_1} + \underbrace{\frac{d\sigma^2}{2} \int \N{v-\conspoint{\rho_t}}_2^2 d\rho_t(v)}_{=:T_2},
		%&= -\lambda \int \langle  v-\globmin, v-\conspoint{\rho_t}\rangle \,d\rho_t(v) + \frac{d\sigma^2}{2} \int \N{v-\conspoint{\rho_t}}_2^2 d\rho_t(v)=: T_1 + T_2,
	\end{align*}
	where we used $\nabla \phi(v) = v-\globmin$ and $\Delta \phi(v) = d$.
	Expanding the right-hand side of the scalar product in the integrand of $T_1$ by subtracting and adding $\globmin$ yields
	\begin{align*}
		T_1
		&=-\lambda \int \langle v-\globmin, v-\globmin\rangle \,d\rho_t(v) + \lambda  \left\langle \int (v-\globmin) \,d\rho_t(v), \conspoint{\rho_t} - \globmin\right\rangle\\
		%T_1 &=-\lambda \int \N{v-\globmin}_2^2 d\rho_t(v) + \lambda  \left\langle \int (v-\globmin) \,d\rho_t(v), \conspoint{\rho_t} - \globmin\right\rangle\\
		&\leq -2\lambda \CV(\rho_t) + \lambda \N{\bbE(\rho_t) - \globmin}_2 \N{\conspoint{\rho_t}-\globmin}_2
	\end{align*}
	with Cauchy-Schwarz inequality being used in the last step.
	Similarly, again by subtracting and adding $v^*$, for the term $T_2$ we have with Cauchy-Schwarz inequality
	\begin{align} \label{eq:bound_t2_aux}
	\begin{split}
		T_2
		%&= \frac{d\sigma^2}{2} \int \N{v-\conspoint{\rho_t}}_2^2  d\rho_t(v) \\
		%&\leq \frac{d\sigma^2}{2} \left(\int \N{v-\globmin}_2^2 d\rho_t(v) + 2 \N{\conspoint{\rho_t}-\globmin}_2 \int \N{v-\globmin}_2 d\rho_t(v)+  \N{\conspoint{\rho_t}-\globmin}_2^2\right)\\
		%&= \frac{d\sigma^2}{2} \left(\int \N{v-\globmin}_2^2 d\rho_t(v) - 2 \left\langle \int (v-\globmin) \,d\rho_t(v), \conspoint{\rho_t}-\globmin \right\rangle+  \N{\conspoint{\rho_t}-\globmin}_2^2\right)\\
		& \leq d\sigma^2 \left(\CV(\rho_t) + \int \N{v-\globmin}_2 d\rho_t(v)\N{\conspoint{\rho_t}-\globmin}_2 + \frac{1}{2}\N{\conspoint{\rho_t}-\globmin}_2^2\right).
	\end{split}
	\end{align}
	%The result follows after applying Jensen's inequality to get
	%\begin{align*}
		%\N{\bbE(\rho_t) - \globmin}_2 \leq \int \N{v-\globmin}_2 d\rho_t(v) \leq \sqrt{\int \N{v-\globmin}_2^2 d\rho_t(v)}=\sqrt{2\CV(\rho_t)}.
	%\end{align*}
	The result now follows by noting that $\Nnormal{\bbE(\rho_t) - \globmin}_2 \leq \int \Nnormal{v-\globmin}_2 \,d\rho_t(v) \leq \sqrt{2\CV(\rho_t)}$ as a consequence of Jensen's inequality.
\end{proof}

%\begin{remark} \label{rem:evo_lemma_1}
%	Lemma~\ref{lem:evolution_of_objective} actually holds for an arbitrary but fixed point $\globmin \in \bbR^d$ and we do not require \ref{asm:zero_global}--\ref{asm:icp} in this step of the analysis.
%\end{remark}

\revised{\begin{lemma}
\label{lem:evolution_of_objective_lower}
    Under the assumptions of Lemma~\ref{lem:evolution_of_objective}, the functional $\CV(\rho_t)$ satisfies
    \begin{align}
        \label{eq:evolution_of_objective_lower}
        \frac{d}{dt}\CV(\rho_t)
        &\geq
        -\left(2\lambda - d\sigma^2\right) \CV(\rho_t)
        - \sqrt{2}\left(\lambda + d\sigma^2\right) \sqrt{\CV(\rho_t)} \N{\conspoint{\rho_t}-\globmin}_2.
\end{align}
\end{lemma}

\begin{proof}
    By following the lines of the proof of Lemma~\ref{lem:evolution_of_objective} and noticing that by Cauchy-Schwarz inequality it holds $\left\langle \int (v-\globmin) \,d\rho_t(v), \conspoint{\rho_t} - \globmin\right\rangle \geq -\N{\bbE(\rho_t) - \globmin}_2 \N{\conspoint{\rho_t}-\globmin}_2$ and since moreover $\N{\conspoint{\rho_t}-\globmin}_2^2\geq0$, the lower bound is immediate. 
\end{proof}}

\begin{lemma} \label{lem:evolution_of_objective_H}
	Let $\CE \in \CC(\bbR^d)$ satisfy \ref{asm:zero_global}--\ref{asm:local_lipschitz} \revised{and w.l.o.g.\@~assume $\minobj = 0$.}
	Let $\cutoffnoarg : \bbR^d \rightarrow [0,1]$ be such that $\cutoff{x} = 1$ whenever $x \geq 0$
	and fix $\alpha,\lambda,\sigma > 0$.
	Moreover, let $T>0$ and let $\rho \in \CC([0,T],\CP_4(\bbR^d))$ be a weak solution to the Fokker-Planck equation~\eqref{eq:fokker_planck}.
	Then, provided $\max_{t \in [0,T]}\CE(\conspoint{\rho_t}) \leq \CE_{\infty}$,
	%the functional $\CV(\rho_t) := 1/2\int \N{v-\globmin}_2^2 d\rho_t(v)$ satisfies
	the functional $\CV(\rho_t)$ satisfies
	\begin{align*}
		\frac{d}{dt} \CV(\rho_t) \leq
		& -\left(2\lambda-d\sigma^2\right)\CV(\rho_t) + \sqrt{2}\left(\lambda+d\sigma^2\right) \sqrt{\CV(\rho_t)} \N{\conspoint{\rho_t}-\globmin}_2 \\
		&+ \frac{\lambda}{\eta^2}  \big(L_{\CE}\left(1+\N{\conspoint{\rho_t}-\globmin}_2^{\gamma}\right)\N{\conspoint{\rho_t}-\globmin}_2\big)^{2\nu} +\frac{d\sigma^2}{2} \N{\conspoint{\rho_t}-\globmin}_2^2.
	\end{align*}
\end{lemma}

\begin{proof}
	Let us write $\cutoffnoarg^*(v) := \cutoff{\CE(v)-\CE(\conspoint{\rho_t})}$.
	%As in the proof of Lemma~\ref{lem:evolution_of_objective}, we take $\phi(v) = 1/2\N{v-\globmin}_2^2$.
	%Since $\rho$ satisfies the weak solution identity~\eqref{eq:weak_solution_identity} for all test functions $\CC_*^2(\bbR^d)$, the evolution of $\CV(\rho_t)$ reads
	%Leveraging \eqref{eq:weak_solution_identity}, the evolution of $\CV(\rho_t)$ reads
	Taking $\phi(v) = 1/2\Nnormal{v-\globmin}_2^2$ as test function in~\eqref{eq:weak_solution_identity} as in the proof of Lemma~\ref{lem:evolution_of_objective} yields for the evolution of $\CV(\rho_t)$
	\begin{align} \label{eq:H_initial_decomp_aux}
		\frac{d}{dt} \CV(\rho_t)
		&=  \underbrace{-\lambda \!\int\! H^*(v) \langle  v-\globmin, v-\conspoint{\rho_t}\rangle \,d\rho_t(v)}_{=:\widetilde T_1}
		+ \frac{d\sigma^2}{2}\!\int \!\N{v-\conspoint{\rho_t}}_2^2  \,d\rho_t(v).\!\!
	\end{align}
	For the second term on the right-hand side, we proceed as in Equation~\eqref{eq:bound_t2_aux}. The term $\widetilde T_1$ on the other hand can be rewritten as
	\begin{align} \label{eq:T1_initial_decomp}
	\begin{split}
		\widetilde T_1 
			= -2\lambda\CV(\rho_t)
			&- \lambda \int H^*(v) \langle v-\globmin, \globmin-\conspoint{\rho_t}\rangle \,d\rho_t(v) \\
			&+ \lambda\int (1-H^*(v)) \N{v-\globmin}_2^2 d\rho_t(v).
	\end{split}
	\end{align}
	Let us now bound the latter two terms individually. For the second term in \eqref{eq:T1_initial_decomp}, noting that $0\leq H^*\leq1$, Cauchy-Schwarz inequality and Jensen's inequality give
	\begin{align*}
		-\lambda\int H^*(v) \langle v-\globmin, \globmin-\conspoint{\rho_t}\rangle \,d\rho_t(v)
		&\leq \lambda \sqrt{2\CV(\rho_t)} \N{\conspoint{\rho_t}-\globmin}_2.
	\end{align*}
	For the third term in \eqref{eq:T1_initial_decomp}, let us first note that $(1-H^*(v)) \neq 0$ implies $H^*(v) \neq 1$ and thus $\CE(v)<\CE(\conspoint{\rho_t})$.
	Furthermore, $\CE(\conspoint{\rho_t}) \leq \CE_{\infty}$ implies $v \in B_{R_0}(\globmin)$ by the second part of \ref{asm:icp}.
	By the first part of \ref{asm:icp} and $0\leq1-H^*\leq1$, we therefore have
	\begin{align*}
		\lambda\int(1-H^*(v)) \N{v-\globmin}_2^2 \,d\rho_t(v)
		&\leq \lambda \int \frac{(1-H^*(v))}{\eta^2}\CE(v)^{2\nu} \,d\rho_t(v)
		\leq  \frac{\lambda}{\eta^2}\CE(\conspoint{\rho_t})^{2\nu} \\
		&\leq \frac{\lambda}{\eta^2} \big(L_{\CE}\left(1+\N{\conspoint{\rho_t}-\globmin}_2^{\gamma}\right)\N{\conspoint{\rho_t}-\globmin}_2\big)^{2\nu},
	\end{align*}
	where the last step used \ref{asm:local_lipschitz}.
	% Employing the last two inequalities in \eqref{eq:T1_initial_decomp} yields
	% \begin{align} \label{eq:bound_t1tilde_aux}
	% \begin{split}
	% 	\widetilde T_1
	% 	\leq -2\lambda\CV(\rho_t)
	% 	&+ \lambda \sqrt{2\CV(\rho_t)} \N{\conspoint{\rho_t}-\globmin}_2 \\
	% 	&+ \frac{\lambda}{\eta^2}  \big(L_{\CE}\left(1+\N{\conspoint{\rho_t}-\globmin}_2^{\gamma}\right)\N{\conspoint{\rho_t}-\globmin}_2\big)^{2\nu}.
	% \end{split}
	% \end{align}
	% The result now follows when inserting \eqref{eq:bound_t1tilde_aux} and \eqref{eq:bound_t2_aux} into \eqref{eq:H_initial_decomp_aux}.
    Employing the last two inequalities in \eqref{eq:T1_initial_decomp} and inserting the result together with \eqref{eq:bound_t2_aux} into \eqref{eq:H_initial_decomp_aux}, gives the result.
\end{proof}

\revised{\begin{lemma}
\label{lem:evolution_of_objective_H_lower}
    Under the assumptions of Lemma~\ref{lem:evolution_of_objective_H}, the functional $\CV(\rho_t)$ satisfies
    \begin{align}
        \label{eq:evolution_of_objective_H_lower}
    \begin{aligned}
        \frac{d}{dt} \CV(\rho_t) \geq
		& -\left(2\lambda-d\sigma^2\right)\CV(\rho_t) - \sqrt{2}\left(\lambda+d\sigma^2\right) \sqrt{\CV(\rho_t)} \N{\conspoint{\rho_t}-\globmin}_2.
    \end{aligned}
    \end{align}
\end{lemma}

\begin{proof}
    Analogously to the proof of Lemma~\ref{lem:evolution_of_objective_lower}, by following the lines of the proof of Lemma~\ref{lem:evolution_of_objective_H} and noticing that for $\widetilde{T}_1$ it holds $-\int H^*(v) \langle v-\globmin, \globmin-\conspoint{\rho_t}\rangle \,d\rho_t(v) \geq -\N{\bbE(\rho_t) - \globmin}_2 \N{\conspoint{\rho_t}-\globmin}_2$ as well as $\int (1-H^*(v)) \N{v-\globmin}_2^2 \,d\rho_t(v) \geq0$ as a consequence of $0\leq H^*\leq1$, the lower bound is immediate.
\end{proof}}

\subsection{Quantitative Laplace principle}
\label{subsec:quant_laplace}
The Laplace principle \eqref{eq:laplace_principle} asserts that $-\log(\Nnormal{\omegaa}_{L_1(\indivmeasure)})/\alpha \rightarrow \minobj$ as $\alpha\rightarrow \infty$ as long as the global minimizer $\globmin$ is in the support of $\indivmeasure$.
However, it cannot be used to characterize the proximity of $\conspoint{\indivmeasure}$ to the global minimizer $v^*$ in general. 
For instance, if $\CE$ had two minimizers with similar objective value $\minobj$, and half of the probability mass of $\indivmeasure$ concentrates around each associated location, $\conspoint{\indivmeasure}$ is located halfway on the line that connects the two minimizing locations.
\revised{The} inverse continuity property~\ref{asm:icp}, \revised{by design,} excludes such cases, so that we can refine the Laplace principle under \ref{asm:icp} in the following sense.

\begin{proposition} \label{lem:laplace_alt}
Let $\indivmeasure \in \CP(\bbR^d)$ and fix $\alpha  > 0$.
For any $r > 0$ define $\CE_{r} := \sup_{v \in B_{r}(\globmin)}\CE(v)$.
Then, under the inverse continuity property~\ref{asm:icp} and \revised{assuming w.l.o.g.\@~$\minobj = 0$,} for any $r \in (0,R_0]$ and $q > 0$  such that $q + \CE_{r} \leq \CE_{\infty}$, we have
\begin{align*}
\N{\conspoint{\indivmeasure} - \globmin}_2 \leq \frac{(q + \CE_{r})^\nu}{\eta} + \frac{\exp(-\alpha q)}{\indivmeasure(B_{r}(\globmin))}\int\N{v-\globmin}_2d\indivmeasure(v).
\end{align*}
\end{proposition}
\begin{proof}
	For any $a > 0$ it holds $\Nnormal{\omegaa}_{L_1(\indivmeasure)} \geq a \indivmeasure(\{v : \exp(-\alpha \CE(v)) \geq a\})$ due to Markov's inequality.
	By choosing $a = \exp(-\alpha\CE_{r})$ and noting that
	\begin{align*}
		\indivmeasure\left(\left\{v \in \bbR^d: \exp(-\alpha \CE(v)) \geq \exp(-\alpha\CE_{r})\right\}\right)
		&= \indivmeasure\left(\left\{ v \in \bbR^d: \CE(v) \leq \CE_{r} \right\}\right) \geq \indivmeasure(B_{r}(\globmin)),
	\end{align*}
	we get $\N{\omegaa}_{L_1(\indivmeasure)} \geq \exp(-\alpha \CE_{r})\indivmeasure(B_{r}(\globmin))$.
	Now let $\widetilde r \geq r > 0$.
	Using the definition of the consensus point $\conspoint{\indivmeasure} = \int v \omegaa(v)/\Nnormal{\omegaa}_{L_1(\indivmeasure)}\,d\indivmeasure(v)$ we can decompose
	\begin{align*}
		\N{\conspoint{\indivmeasure} \!-\! \globmin}_2
		&\leq \int_{B_{\widetilde r}(\globmin)} \!\!\N{v\!-\!\globmin}_2\!\frac{\omegaa(v)}{\N{\omegaa}_{L_1(\indivmeasure)}}\,d\indivmeasure(v) + \int_{\left(B_{\widetilde r}(\globmin)\right)^c} \!\!\N{v\!-\!\globmin}_2\!\frac{\omegaa(v)}{\N{\omegaa}_{L_1(\indivmeasure)}}\,d\indivmeasure(v).
	\end{align*}
	The first term is bounded by $\widetilde r$ since $ \Nnormal{v-\globmin}_2\leq \widetilde r$ for all $v \in B_{\widetilde r}(\globmin)$.
	For the second term we use $\N{\omegaa}_{L_1(\indivmeasure)} \geq \exp(-\alpha \CE_{r})\indivmeasure(B_{r}(\globmin))$ from above to get
%	\begin{align*}
%		&\int_{\left(B_{\widetilde r}(\globmin)\right)^c} \!\!\!\!\!\N{v\!-\!\globmin}_2\!\frac{\omegaa(v)}{\N{\omegaa}_{L_1(\indivmeasure)}}\,d\indivmeasure(v)\!
%		\leq\! \frac{1}{ \exp(-\alpha \CE_{r})\indivmeasure(B_{r}(\globmin))}\!\int_{(B_{\widetilde r}(\globmin))^c} \!\!\!\!\!\N{v\!-\!\globmin}_2 \omegaa(v)\,d\indivmeasure(v)\\
%		&\qquad\qquad\qquad\qquad\leq\!\frac{\exp\left(-\alpha \inf_{v \in (B_{\widetilde r}(\globmin))^c}\CE(v)\right)}{ \exp\left(-\alpha \CE_{r})\indivmeasure(B_{r}(\globmin)\right)}\!\int\!\N{v\!-\!\globmin}_2d\indivmeasure(v)\\
%		&\qquad\qquad\qquad\qquad=\! \frac{\exp\left(-\alpha \left(\inf_{v \in (B_{\widetilde r}(\globmin))^c}\CE(v) \!-\! \CE_{r}\right)\right)}{\indivmeasure(B_{r}(\globmin))}\!\int\!\N{v\!-\!\globmin}_2d\indivmeasure(v).
%	\end{align*}
	\begin{align*}
		&\int_{\left(B_{\widetilde r}(\globmin)\right)^c} \!\!\!\!\!\N{v\!-\!\globmin}_2\!\frac{\omegaa(v)}{\N{\omegaa}_{L_1(\indivmeasure)}}\,d\indivmeasure(v)\!
		\leq\! \frac{1}{ \exp(-\alpha \CE_{r})\indivmeasure(B_{r}(\globmin))}\!\int_{(B_{\widetilde r}(\globmin))^c} \!\!\!\!\!\N{v\!-\!\globmin}_2 \omegaa(v)\,d\indivmeasure(v)\\
		&\qquad\qquad\qquad\qquad\leq\! \frac{\exp\left(-\alpha \left(\inf_{v \in (B_{\widetilde r}(\globmin))^c}\CE(v) \!-\! \CE_{r}\right)\right)}{\indivmeasure(B_{r}(\globmin))}\!\int\!\N{v\!-\!\globmin}_2d\indivmeasure(v).
	\end{align*}
	Thus, for any $\widetilde r \geq r > 0$ we obtain
	\begin{align} \label{eq:aux_laplace_1}
		\N{\conspoint{\indivmeasure} - \globmin}_2
		\leq \widetilde r+ \frac{\exp\left(-\alpha \left(\inf_{v \in (B_{\widetilde r}(\globmin))^c}\CE(v) - \CE_{r}\right)\right)}{\indivmeasure(B_{r}(\globmin))}\int\N{v-\globmin}_2d\indivmeasure(v).
	\end{align}
	Let us now choose $\widetilde r = (q+\CE_r)^{\nu}/\eta$.
	This choice satisfies $\widetilde r \leq \CE_{\infty}^{\nu}/\eta$ by the assumption $q+\CE_r \leq \CE_{\infty}$, and furthermore $\widetilde r  \geq r$, since \ref{asm:icp} with $\minobj = 0$ and $r \leq R_0$ implies
	\begin{align*}
		\widetilde r
		= \frac{(q+\CE_r)^{\nu}}{\eta}
		\geq \frac{\CE_r^{\nu}}{\eta}
		= \frac{\left(\sup_{v \in B_{r}(\globmin)}\CE(v)\right)^{\nu}}{\eta} \geq \sup_{v \in B_{r}(\globmin)}\N{v-\globmin}_2
		= r.
	\end{align*}
	% Using again \ref{asm:icp} with $\minobj = 0$ we thus have
	% \begin{align*}
	% 	\inf_{v \in (B_{\widetilde r}(\globmin))^c}\CE(v) - \CE_{r}
	% 	\geq \min\big\{\CE_{\infty}, (\eta \widetilde r)^{1/\nu}\big\} - \CE_{r}
	% 	= (\eta \widetilde r)^{1/\nu} - \CE_{r}
	% 	= q + \CE_r - \CE_{r}
	% 	= q.
	% \end{align*}
    Thus, using again \ref{asm:icp} with $\minobj = 0$, $\inf_{v \in (B_{\widetilde r}(\globmin))^c}\CE(v) - \CE_{r} \geq \min\big\{\CE_{\infty}, (\eta \widetilde r)^{1/\nu}\big\} - \CE_{r} = (\eta \widetilde r)^{1/\nu} - \CE_{r} = q$.
	Inserting this and the definition of $\widetilde r$ into \eqref{eq:aux_laplace_1}, we obtain the result.
\end{proof}

\subsection{A lower bound for the probability mass  around $v^*$} \label{subsec:lower_bound_prob_mass}
In this section we bound the probability mass $\rho_t(B_{r}(\globmin))$ for an arbitrary small radius $r > 0$ from below.
By defining a smooth mollifier $\phi_r : \bbR^d \rightarrow [0,1]$ with $\supp\phi_r = B_{r}(\globmin)$ according to
\begin{align} \label{eq:mollifier}
	\phi_{r}(v) :=
	\begin{cases}
		\exp\left(1-\frac{r^2}{r^2-\N{v-\globmin}_2^2}\right),& \text{ if } \N{v-\globmin}_2 < r,\\
		0,& \text{ else,}
	\end{cases}
\end{align}
it holds $\rho_t(B_{r}(\globmin)) \!\geq\! \int \phi_r(v)\,d\rho_t(v)$.
%Then the weak solution property of $\rho$ in Definition~\ref{def:fokker_planck_weak_sense} allows to study the evolution of $\int \phi_r(v)\,d\rho_t(v)$ and control its decay rate.
From there, the evolution of the right-hand side can be studied by using the weak solution property of $\rho$ as in Definition~\ref{def:fokker_planck_weak_sense}, since $\phi_{r} \in \CC_c^{\infty}(\bbR^d)$.

To do so, we compute the derivatives
\begin{align}
	\label{eq:mollifier_gradient}
	\nabla \phi_{r}(v) &= -2r^2 \frac{v-\globmin}{\left(r^2-\N{v-\globmin}_2^2\right)^2}\phi_{r}(v),\\
	\label{eq:mollifier_laplace}
	\Delta \phi_{r}(v) &= 2r^2 \left(\frac{2\left(2\N{v-\globmin}_2^2-r^2\right)\N{v-\globmin}_2^2 - d\left(r^2-\N{v-\globmin}_2^2\right)^2}{\left(r^2-\N{v-\globmin}_2^2\right)^4}\right)\phi_{r}(v).
\end{align}

\begin{proposition} \label{lem:lower_bound_probability}
	Let $\cutoffnoarg: \bbR \rightarrow [0,1]$ be arbitrary, $T > 0,\ r > 0$, and fix parameters $\alpha,\lambda,\sigma > 0$.
	Assume $\rho\in\CC([0,T],\CP(\bbR^d))$ weakly solves the Fokker-Planck equation~\eqref{eq:fokker_planck} in the sense of
	Definition~\ref{def:fokker_planck_weak_sense} with initial condition $\rho_0 \in \CP(\bbR^d)$ and for $t \in [0,T]$.
	Then, for all $t\in[0,T]$ we have
	\begin{align} 
		\rho_t\left(B_{r}(\globmin)\right)
		%= \rho_t\left(\left\{v \in \bbR^d : \N{v-v^*}_2 \leq r\right\}\right)
		&\geq \left(\int\phi_{r}(v)\,d\rho_0(v)\right)\exp\left(-pt\right), \text{ where} \label{eq:lower_bound_probability_rate} \\
        p &:= \max\left\{\frac{2\lambda(\sqrt{c}r+B)\sqrt{c} }{(1-c)^2r}+\frac{2\sigma^2(cr^2+B^2)(2c+d)}{(1-c)^4r^2},\frac{4\lambda^2}{(2c-1)\sigma^2}\right\}
        \label{eq:def_p}
	\end{align}
	for any $B<\infty$ with $\sup_{t \in [0,T]}\N{\conspoint{\rho_{t}}-v^*}_2 \leq B$
	and for any $c \in (1/2,1)$ satisfying
	\begin{align} \label{eq:def_c}
		(2c-1)c \geq d(1-c)^2.
	\end{align}
\end{proposition}

\begin{remark} \label{remark:lem:lower_bound_probability}
	Let us comment in what follows on two technical details of Proposition~\ref{lem:lower_bound_probability}.
	\begin{enumerate}[label=(\roman*),labelsep=10pt,leftmargin=35pt]
        \item Note that neither the definition of $\conspoint{\rho_{t}}$ nor $\globmin$ play a significant role in the proof.
			The same result holds for an arbitrary $v \in \bbR^d$ so that $\sup_{t \in [0,T]}\Nnormal{\conspoint{\rho_{t}}-v}_2 \leq B < \infty$ and for arbitrary continuous maps $u \in \CC([0,T], \bbR^d)$ to replace $\conspoint{\rho_{t}}$ as long as $\rho$ weakly solves the Fokker-Planck equation~\eqref{eq:fokker_planck} with $u$ in place of $\conspoint{\rho_{t}}$.	
        \item In case the reader may have wondered about the crucial role of the stochastic terms in \eqref{eq:dyn_micro_discrete} and \eqref{eq:dyn_micro}, or the diffusion in the macroscopic models~\eqref{eq:dyn_macro} and~\eqref{eq:fokker_planck},  Proposition~\ref{lem:lower_bound_probability} precisely explains where positive diffusion $\sigma>0$ is actually used to ensure mass around the minimizer $\globmin$ (compare Proposition \ref{lem:laplace_alt}).
			We require $\sigma > 0$ in Proposition~\ref{lem:lower_bound_probability} to ensure a finite decay rate $q < \infty$, see the definition in Equation \eqref{eq:lower_bound_probability_rate}.
			Intuitively, we can understand the measure $\rho_t$ as having a deterministic component, which evolves according to the drift term in the Fokker-Planck equation~\eqref{eq:fokker_planck} and whose associated mass may leave $B_r(\globmin)$ in finite time, convolved with an exponentially decaying kernel from the diffusion term.
			This convolution ensures that the mass leaves at most exponentially fast, leading to the lower bound.
			The statement does not hold in general in the case $\sigma = 0$.
			%We conjecture that the statement is either not true in the case $\sigma = 0$ or it requires a different proof strategy.
    \end{enumerate}
\end{remark}

\begin{proof}[Proof of Proposition~\ref{lem:lower_bound_probability}]
	By definition of the mollifier $\phi_{r}$ in \eqref{eq:mollifier} we have $0\leq \phi_{r}(v) \leq 1$ and $\supp(\phi_{r}) = B_{r}(\globmin)$.
	This implies
	\begin{align} \label{eq:bound_massaroundv*_initial}
		\rho_t\left(B_{r}(\globmin)\right)
		= \rho_t\left(\left\{v \in \bbR^d: \N{v-v^*}_2\leq r\right\}\right)
		\geq \int \phi_{r}(v)\,d\rho_t(v).
	\end{align}
	Our strategy is to derive a lower bound for the right-hand side of this inequality.
	Using the weak solution property of $\rho$ and the fact that  $\phi_{r}\in \CC^{\infty}_c(\bbR^d)$, we obtain
	\begin{align} \label{eq:initial_evolution}
		\frac{d}{dt}\int\phi_{r}(v)\,d\rho_t(v)
		&= \int \left(T_1(v) + T_2(v)\right) d\rho_t(v)
	\end{align}
    with $T_1(v) := -\lambda\cutoffnoarg^*(v) \langle v-\conspoint{\rho_t}, \nabla \phi_{r}(v)\rangle$ and $T_2(v) := \frac{\sigma^2}{2}\N{v-\conspoint{\rho_t}}_2^2 \Delta \phi_{r}(v)$,
	and where we abbreviate $\cutoffnoarg^*(v) := \cutoff{\CE(v)-\CE(\conspoint{\rho_t})}$ to keep the notation concise.
	We now aim for showing $T_1(v) + T_2(v) \geq -p\phi_r(v)$ uniformly on $\bbR^d$ for $p>0$ as given in~\eqref{eq:def_p} in the statement.
	Since the mollifier~$\phi_r$ and its first and second derivatives vanish outside of $\Omega_r := \{v \in \bbR^d : \Nnormal{v-\globmin}_2 < r\}$ we can restrict our attention to the open ball $\Omega_r$.
	To achieve the lower bound over $\Omega_r$, we introduce the subsets
	% \begin{align*}
	% 	K_1 &:= \left\{v \in \bbR^d : \N{v-v^*}_2 > \sqrt{c}r\right\}
	% \end{align*}
    $K_1 := \left\{v \in \bbR^d : \N{v-v^*}_2 > \sqrt{c}r\right\}$
	and
    \begin{align*}
	\begin{aligned}
		&K_2 := \left\{v \in \bbR^d : -\lambda\cutoffnoarg^*(v)\langle v-\conspoint{\rho_t},v-v^*\rangle \left(r^2-\N{v-v^*}_2^2\right)^2\right. \\
		&\qquad\qquad\qquad\qquad\qquad\left. > \tilde{c}r^2\frac{\sigma^2}{2}\N{v-\conspoint{\rho_t}}_2^2\N{v-v^*}_2^2\right\},
	\end{aligned}
	\end{align*}
	where $c$ adheres to \eqref{eq:def_c}, and $\tilde{c} := 2c-1\in(0,1)$.\!
	We now decompose $\Omega_r$ according to
	% \begin{align} \label{eq:disjoint_sets}
	% 	\Omega_r = \left(K_1^c \cap \Omega_r\right)\cup \left(K_1 \cap K_2^c \cap \Omega_r\right) \cup \left(K_1 \cap K_2 \cap \Omega_r\right),
	% \end{align}
    $\Omega_r = \left(K_1^c \cap \Omega_r\right)\cup \left(K_1 \cap K_2^c \cap \Omega_r\right) \cup \left(K_1 \cap K_2 \cap \Omega_r\right)$,
	which is illustrated in Figure~\ref{fig:decomposition_support} for different positions of $\conspoint{\rho_t}$ and values of $\sigma$.
	\begin{figure}[!ht]
		\centering
		\subcaptionbox{\footnotesize Decomposition for $\conspoint{\rho_t}\in\Omega_r$ and $\sigma=0.2$\label{fig:decomposition1}}{\includegraphics[width=0.25\textwidth, trim=119 252 123 250,clip]{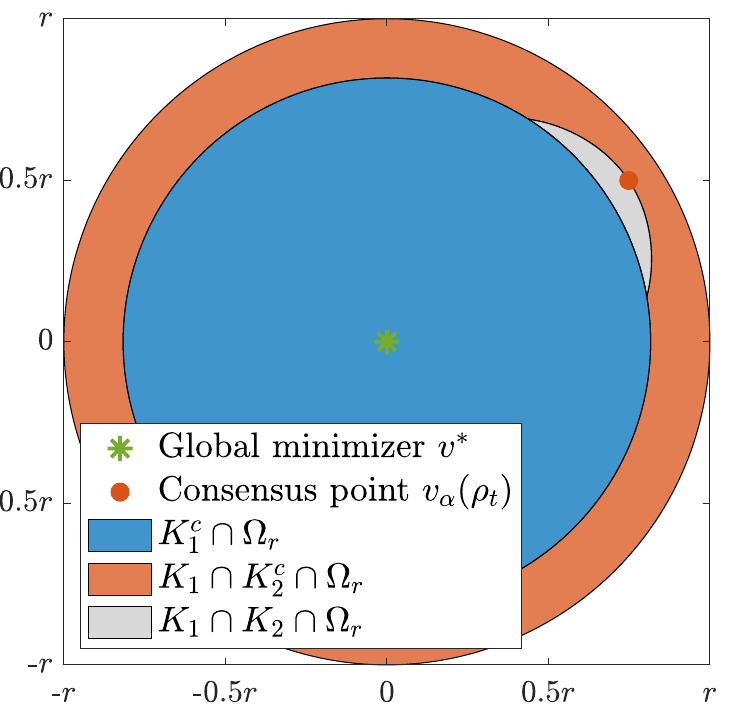}}
		\hspace{1.5em}
		\subcaptionbox{\footnotesize Decomposition for $\conspoint{\rho_t}\not\in\Omega_r$ and $\sigma=0.2$\label{fig:decomposition3}}{\includegraphics[width=0.25\textwidth, trim=119 252 123 250,clip]{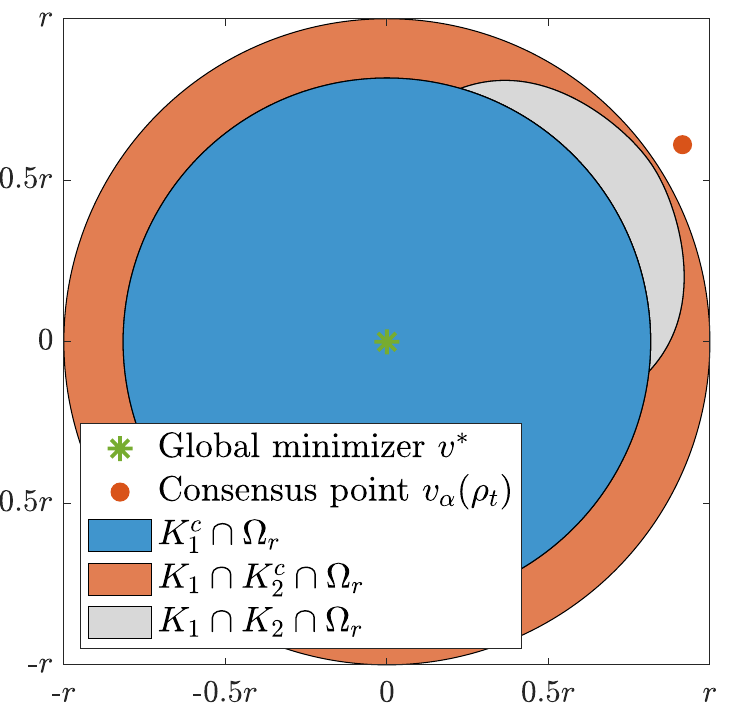}}
		\hspace{1.5em}
		\subcaptionbox{\footnotesize Decomposition for $\conspoint{\rho_t}\not\in\Omega_r$ and $\sigma=1$\label{fig:decomposition2}}{\includegraphics[width=0.25\textwidth, trim=119 252 123 250,clip]{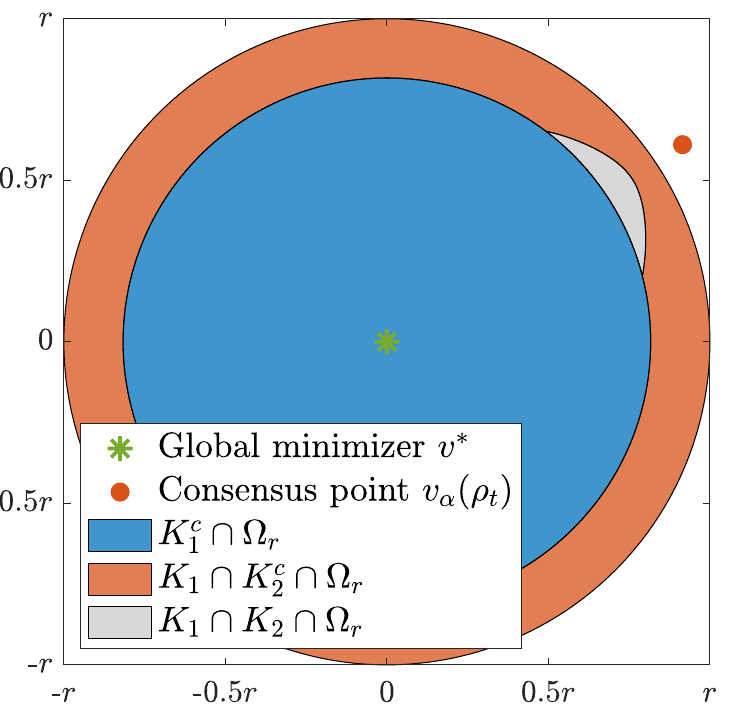}}
		\caption{Visualization of the decomposition of $\Omega_r$ for different positions of $\conspoint{\rho_t}$ and values of $\sigma$ in the setting $\cutoffnoarg\equiv1$.
		In the proof of Proposition~\ref{lem:lower_bound_probability} we limit the rate of the mass loss induced by both consensus drift and noise term for the set $K_1^c \cap \Omega_r$, which is colored blue.
		On the set $K_1 \cap K_2^c \cap \Omega_r$, inked orange, the noise term counterbalances any potential mass loss induced by the drift, while on the gray set $K_1 \cap K_2 \cap \Omega_r$ mass can be lost at an exponential rate~$-4\lambda^2/((2c-1)\sigma^2)$.}
		\label{fig:decomposition_support}
	\end{figure}

	\noindent
	In the following we treat each of these three subsets separately.

	\noindent
	\textbf{Subset $K_1^c \cap \Omega_r$:}
	We have $\Nnormal{v-v^*}_2 \leq \sqrt{c}r$ for each $v \in K_1^c$, which can be used to independently derive lower bounds for both $T_1$ and $T_2$.
	Recalling the expression for~$\phi_r$ from~\eqref{eq:mollifier}, for $T_1$ we get by using Cauchy-Schwarz inequality and $H^*\leq1$
	\begin{align*}
		T_1(v)
		&= -\lambda\cutoffnoarg^*(v) \langle v-\conspoint{\rho_t}, \nabla \phi_{r}(v)\rangle
		= - \lambda\cutoffnoarg^*(v) \left\langle v-\conspoint{\rho_t},\frac{-2r^2(v-\globmin)\phi_{r}(v)}{\left(r^2-\N{v-\globmin}_2^2\right)^2}\right\rangle \\
		&\geq -2 r^2 \lambda \frac{\N{v-\conspoint{\rho_t}}_2\N{v-v^*}_2}{\left(r^2-\N{v-\globmin}_2^2\right)^2}\phi_r(v)
		\geq - \frac{2\lambda(\sqrt{c}r+B)\sqrt{c} }{(1-c)^2r}\phi_r(v) =: -p_1\phi_r(v),
	\end{align*}
	where the last bound is due to $\Nnormal{v-\conspoint{\rho_t}}_2 \leq \N{v-v^*}_2+\Nnormal{v^*-\conspoint{\rho_t}}_2\leq \sqrt{c}r+B$.
	Similarly, by computing $\Delta \phi_{r}$ and inserting it, for $T_2$ we obtain
	\begin{align*}
		T_2(v)
		\!&=\! \sigma^2r^2 \N{v-\conspoint{\rho_t}}_2^2 \frac{2\left(2\N{v-\globmin}_2^2\!-\!r^2\right)\N{v-\globmin}_2^2\!-\!d\left(r^2-\N{v-\globmin}_2^2\right)^2}{\left(r^2-\N{v-\globmin}_2^2\right)^4}\phi_{r}(v)\\
		&\geq - \frac{2\sigma^2(cr^2+B^2)(2c+d)}{(1-c)^4r^2}\phi_{r}(v) =: -p_2\phi_r(v),
	\end{align*}
	where we used $\N{v-\conspoint{\rho_t}}_2^2 \leq 2\big(\!\N{v-v^*}_2^2+\N{v^*-\conspoint{\rho_t}}_2^2\!\big)\leq 2(cr^2+B^2)$.

	\noindent
	\textbf{Subset $K_1 \cap K_2^c \cap \Omega_r$:}
	By the definition of $K_1$ and $K_2$ we have $\Nnormal{v-\globmin}_2 > \sqrt{c} r$ and
	\begin{align} \label{eq:vinK2c}
		-\lambda\cutoffnoarg^*(v)\langle v\!-\!\conspoint{\rho_t},v\!-\!v^*\rangle \left(r^2\!-\!\N{v\!-\!v^*}_2^2\right)^2
		\!\leq \tilde{c}r^2\frac{\sigma^2}{2}\N{v\!-\!\conspoint{\rho_t}}_2^2\N{v\!-\!v^*}_2^2.
	\end{align}
	Our goal now is to show $T_1(v)+T_2(v) \geq 0$ for all $v$ in this subset.
	We first compute
	\begin{align*}
		&\frac{T_1(v) + T_2(v)}{2r^2\phi_r(v)}
		= \,\lambda \cutoffnoarg^*(v)\frac{\langle v-\conspoint{\rho_t}, v-\globmin\rangle\left(r^2-\N{v-\globmin}_2^2\right)^2}{\left(r^2-\N{v-\globmin}_2^2\right)^4}\\
		&\qquad\qquad+ \frac{\sigma^2}{2}\N{v-\conspoint{\rho_t}}_2^2\frac{2\left(2\N{v-\globmin}_2^2-r^2\right)\N{v-\globmin}_2^2 - d\left(r^2-\N{v-\globmin}_2^2\right)^2}{\left(r^2-\N{v-\globmin}_2^2\right)^4}.
	\end{align*}
	Therefore we have $T_1(v) + T_2(v) \geq 0$ whenever we can show
	\begin{equation} \label{eq:aux_term_3}
	\begin{aligned}
		&\left(-\lambda\cutoffnoarg^*(v) \langle v-\conspoint{\rho_t}, v-\globmin\rangle + \frac{d\sigma^2}{2} \N{v-\conspoint{\rho_t}}_2^2\right)\left(r^2-\N{v-\globmin}_2^2\right)^2 \\
		&\qquad\qquad\qquad\qquad\qquad\quad\;\; \leq \sigma^2 \N{v-\conspoint{\rho_t}}_2^2\left(2\N{v-\globmin}_2^2-r^2\right)\N{v-\globmin}_2^2.
	\end{aligned}
	\end{equation}
	Now note that the first summand on the left-hand side in \eqref{eq:aux_term_3} can be upper bounded by means of Condition \eqref{eq:vinK2c} and by using the relation $\tilde{c} = 2c-1$.
	More precisely,
	\begin{align*}
		&-\!\lambda\cutoffnoarg^*(v) \langle v\!-\!\conspoint{\rho_t}, v\!-\!\globmin\rangle \left(r^2\!-\!\N{v\!-\!\globmin}_2^2\right)^2
		\!\!\leq\! \tilde{c}r^2\frac{\sigma^2}{2}\N{v\!-\!\conspoint{\rho_t}}_2^2\N{v\!-\!v^*}_2^2\\
		&\ =\! (2c\!-\!1)r^2\frac{\sigma^2}{2}\N{v\!-\!\conspoint{\rho_t}}_2^2\N{v\!-\!v^*}_2^2
		\!\leq\!\! \left(2\N{v\!-\!\globmin}_2^2-r^2\right)\frac{\sigma^2}{2}\N{v\!-\!\conspoint{\rho_t}}_2^2\N{v\!-\!v^*}_2^2,
	\end{align*}
	where the last inequality follows since $v\in K_1$.
	For the second term on the left-hand side in \eqref{eq:aux_term_3} we can use $d(1-c)^2 \leq (2c-1)c$ as per \eqref{eq:def_c}, to get
	\begin{align*}
		&\frac{d\sigma^2}{2} \N{v-\conspoint{\rho_t}}_2^2 \left(r^2-\N{v-\globmin}_2^2\right)^2
		\leq \frac{d\sigma^2}{2} \N{v-\conspoint{\rho_t}}_2^2 (1-c)^2r^4 \\
		&\quad\leq \frac{\sigma^2}{2}\! \N{v-\conspoint{\rho_t}}_2^2 (2c-1)r^2cr^2
		\leq \frac{\sigma^2}{2}\! \N{v-\conspoint{\rho_t}}_2^2 \left(2\N{v-\globmin}_2^2-r^2\right)\!\N{v-\globmin}_2^2.
	\end{align*}
	Hence, \eqref{eq:aux_term_3} holds and we have $T_1(v) + T_2(v) \geq 0$ uniformly on this subset.

	\noindent
	\textbf{Subset $K_1 \cap K_2 \cap \Omega_r$:}
	On this subset, we have $\Nnormal{v-\globmin}_2 > \sqrt{c}r$ and
	\begin{align} \label{eq:vinK2}
		-\lambda\cutoffnoarg^*(v)\langle v\!-\!\conspoint{\rho_t},v\!-\!v^*\rangle \left(r^2\!-\!\N{v\!-\!v^*}_2^2\right)^2 \!> \tilde{c}r^2\frac{\sigma^2}{2}\N{v\!-\!\conspoint{\rho_t}}_2^2\N{v\!-\!v^*}_2^2.
	\end{align}
	We first note that $T_1(v) = 0$ whenever $\sigma^2\Nnormal{v-\conspoint{\rho_t}}_2^2 = 0$, provided that $\sigma>0$, so nothing needs to be done for the point $v = \conspoint{\rho_t}$.
	On the other hand, if $\sigma^2\Nnormal{v-\conspoint{\rho_t}}_2^2 > 0$, we can use $\cutoffnoarg^*\leq1$, two applications of Cauchy-Schwarz inequalities, and Condition \eqref{eq:vinK2} to get
	\begin{align*}
	\begin{split}
		\frac{\cutoffnoarg^*(v)\left\langle v-\conspoint{\rho_t},v-v^*\right\rangle}{\left(r^2-\N{v-v^*}_2^2\right)^2}
		&\geq \frac{-\N{v-\conspoint{\rho_t}}_2\N{v-v^*}_2}{\left(r^2-\N{v-v^*}_2^2\right)^2}\\
		&> \frac{2\lambda\cutoffnoarg^*(v)\langle v-\conspoint{\rho_t},v-v^*\rangle}{\tilde{c}r^2\sigma^2\N{v-\conspoint{\rho_t}}_2\N{v-v^*}_2}
		\geq -\frac{2\lambda}{\tilde{c}r^2\sigma^2}.
	\end{split}
	\end{align*}
	Using this, $T_1$ can be bounded from below by
	\begin{align*}
		T_1(v)
		&\!=\! 2\lambda r^2\cutoffnoarg^*(v) \left\langle v\!-\!\conspoint{\rho_t}, \frac{v\!-\!\globmin}{\left(r^2\!-\!\N{v\!-\!\globmin}_2^2\right)^2}\phi_{r}(v)\right\rangle
		\geq -\frac{4\lambda^2}{\tilde{c}\sigma^2}\phi_{r}(v)
        %= -\frac{4\lambda^2}{(2c-1)\sigma^2}\phi_{r}(v)
        =: -p_3\phi_r(v),
	\end{align*}
	where we made use of the relation $\tilde{c} = 2c-1$ in the last step.
	For $T_2$, we note that the nonnegativity of $\sigma^2\Nnormal{v-\conspoint{\rho_t}}_2$ implies $T_2(v) \geq 0$, whenever
	\begin{align*}
		2\left(2\N{v-\globmin}_2^2-r^2\right)\N{v-\globmin}_2^2\geq d\left(r^2-\N{v-\globmin}_2^2\right)^2.
	\end{align*}
	This is satisfied for all $v$ with $\Nnormal{v-\globmin}_2 \geq \sqrt{c}r$, provided $c$ satisfies $2(2c-1)c \geq (1-c)^2d$ as implied by \eqref{eq:def_c}.

	\noindent
	\textbf{Concluding the proof:}
	Using the evolution of $\phi_r$ as in \eqref{eq:initial_evolution}, we now get
	\begin{alignat}{2}
		\phantom{\frac{d}{dt}}
		&\begin{aligned}[c]\nonumber
			\mathllap{\frac{d}{dt}}
			&\int\phi_{r}(v)\,d\rho_t(v) =
			 \int_{K_1 \cap K_2^c \cap \Omega_r}\underbrace{(T_1(v) + T_2(v))}_{\geq 0}\,d\rho_t(v) \\
			&\qquad\qquad + \int_{K_1 \cap K_2 \cap \Omega_r}\underbrace{(T_1(v) + T_2(v))}_{\geq -p_3\phi_{r}(v)}\,d\rho_t(v) + \int_{K_1^c \cap \Omega_r}\underbrace{(T_1(v) + T_2(v))}_{\geq -(p_1+p_2)\phi_{r}(v)}d\rho_t(v)
		\end{aligned} \\
		&\qquad\qquad\begin{aligned}[c]\nonumber
			\mathllap{\geq}
			& -\max\left\{p_1+p_2,p_3\right\} \int\phi_{r}(v)\,d\rho_t(v) = -p \int\phi_{r}(v)\,d\rho_t(v)
		\end{aligned}
	\end{alignat}
	An application of Gr\"onwall's inequality gives $\int\phi_{r}(v)\,d\rho_t(v)\!\geq\! \int\phi_{r}(v)\,d\rho_0(v)\exp(-pt)$, which concludes the proof after recalling~\eqref{eq:bound_massaroundv*_initial}.
\end{proof}

\subsection{Proof of Theorem~\ref{thm:global_convergence_main}} \label{subsec:proof_main}
We now have all necessary tools at hand to present a detailed proof of the global convergence result in mean-field law.
%We first separately prove the case of an inactive cutoff function, i.e., $H\equiv1$.
We separately prove the cases of an inactive and active cutoff function, i.e., $H\equiv1$ and $H\not\equiv1$,~respectively.

\revised{\begin{proof}[Proof of Theorem~\ref{thm:global_convergence_main} when $H\equiv 1$]
	W.l.o.g.\@~we may assume $\minobj = 0$.
	Let us first choose the parameter~$\alpha$ such that
	\begin{align} \label{eq:alpha}
	\begin{split}
		\alpha>
		\alpha_0
		:= \frac{1}{q_\varepsilon}\Bigg(\log\left(\frac{4\sqrt{2\CV(\rho_0)}}{c\left(\vartheta,\lambda,\sigma\right)\sqrt{\varepsilon}}\right) 
		&+ \frac{p}{(1-\vartheta)\left(2\lambda-d\sigma^2\right)}\log\left(\frac{\CV(\rho_0)}{\varepsilon}\right) \\
        &- \log\rho_0\big(B_{r_\varepsilon/2}(v^*)\big)\!\Bigg),
	\end{split}
	\end{align}
	where we introduce the definitions
	\begin{align}
		c\left(\vartheta,\lambda,\sigma\right)
		:= \min\left\{
			\frac{\vartheta}{2}\frac{\left(2\lambda-d\sigma^2\right)}{\sqrt{2}\left(\lambda+d\sigma^2\right)}, 
			\sqrt{\vartheta\frac{\left(2\lambda-d\sigma^2\right)}{d\sigma^2}}
			\right\} 
	\end{align}
	as well as
	\begin{align*}
		q_\varepsilon := \frac{1}{2}\min\bigg\{\!\left(\eta\frac{c\left(\vartheta,\lambda,\sigma\right)\sqrt{\varepsilon}}{2}\right)^{1/\nu}\!,\CE_{\infty}\bigg\}
		\quad \text{and} \quad
		r_\varepsilon :=\max_{s \in [0,R_0]}\left\{\max_{v \in B_s(\globmin)}\CE(v) \leq q_\varepsilon\right\}.
	\end{align*}
	Moreover, $p$ is as defined in \eqref{eq:def_p} in Proposition~\ref{lem:lower_bound_probability} with $B=c(\vartheta,\lambda,\sigma)\sqrt{\CV(\rho_0)}$ and with $r=r_\varepsilon$.
	We remark that, by construction, $q_\varepsilon>0$ and $r_\varepsilon\leq R_0$.
	Furthermore, recalling the notation $\CE_{r}=\sup_{v \in B_{r}(\globmin)}\CE(v)$ from Proposition~\ref{lem:laplace_alt}, we have $q_\varepsilon+\CE_{r_\varepsilon} \leq 2q_\varepsilon \leq \CE_{\infty}$ as a consequence of the definition of~$r_\varepsilon$.
	Since $q_\varepsilon>0$, the continuity of $\CE$ ensures that there exists $s_{q_\varepsilon}>0$ such that $\CE(v)\leq q_\varepsilon$ for all $v\in B_{s_{q_\varepsilon}}(\globmin)$, thus yielding also $r_\varepsilon>0$.
	
	Let us now define the time horizon $T_\alpha \geq 0$, which may depend on $\alpha$, by
	\begin{align} \label{eq:endtime_T}
		T_\alpha := \sup\big\{t\geq0 : \CV(\rho_{t'}) > \varepsilon \text{ and } \N{\conspoint{\rho_{t'}}-\globmin}_2 < C(t') \text{ for all } t' \in [0,t]\big\}
	\end{align}
	with $C(t):=c(\vartheta,\lambda,\sigma)\sqrt{\CV(\rho_t)}$.
	Notice for later use that $C(0)=B$.
	
    %Our aim is to show that $\min_{t \in [0,T_\alpha]}\CV(\rho_t) \leq \varepsilon$ with $T_\alpha\leq T^*$ and that we have at least exponential decay until $\CV(\rho_t)$ reaches the prescribed accuracy~$\varepsilon$.
    Our aim now is to show $\CV(\rho_{T_\alpha}) = \varepsilon$ with $T_\alpha\in\big[\frac{1-\vartheta}{(1+\vartheta/2)}\;\!T^*,T^*\big]$ and that we have at least exponential decay of $\CV(\rho_t)$ until time $T_\alpha$, i.e., until accuracy~$\varepsilon$ is reached.
	
	First, however, we ensure that $T_\alpha>0$.
    With the mapping~$t\mapsto\CV(\rho_{t})$ being continuous as a consequence of the regularity $\rho\in\CC([0,T], \CP_4(\bbR^d))$ established in Theorem~\ref{thm:well-posedness_FP} and $t\mapsto\N{\conspoint{\rho_{t}}-\globmin}_2$ being continuous due to \cite[Lemma~3.2]{carrillo2018analytical} and $\rho\in\CC([0,T], \CP_4(\bbR^d))$, 
    $T_\alpha>0$ follows from the definition, since $\CV(\rho_{0}) > \varepsilon$ and $\N{\conspoint{\rho_{0}}-\globmin}_2 < C(0)$.
	While the former is immediate by assumption, applying Proposition~\ref{lem:laplace_alt} with $q_\varepsilon$ and $r_\varepsilon$
    %as defined in~\eqref{eq:q_and_r}
    gives the latter since
	\begin{align*} 
		\N{\conspoint{\rho_{0}} - \globmin}_2
		&\leq \frac{\left(q_\varepsilon+\CE_{r_\varepsilon}\right)^\nu}{\eta} + \frac{\exp\left(-\alpha q_\varepsilon\right)}{\rho_0(B_{r_\varepsilon}(v^*))}\int\N{v-\globmin}_2d\rho_{0}(v)\\
		&\leq \frac{c\left(\vartheta,\lambda,\sigma\right)\sqrt{\varepsilon}}{2} + \frac{\exp\left(-\alpha q_\varepsilon\right)}{\rho_0(B_{r_\varepsilon}(v^*))}\sqrt{2\CV(\rho_0)}\\
		&\leq c\left(\vartheta,\lambda,\sigma\right)\sqrt{\varepsilon}
		< c\left(\vartheta,\lambda,\sigma\right)\sqrt{\CV(\rho_0)} = C(0),
	\end{align*}
	where the first inequality in the last line holds by the choice of $\alpha$ in \eqref{eq:alpha}.
%	{\color{blue} % do not uncomment
%		To fulfill here the first inequality in the last line, we need to ensure that
%		\begin{equation*}
%			\frac{\exp\left(-\alpha q_\varepsilon\right)}{\rho_0(B_{r_\varepsilon}(v^*))}\sqrt{2\CV(\rho_0)}
%			\leq \frac{c\left(\vartheta,\lambda,\sigma\right)\sqrt{\varepsilon}}{2},
%		\end{equation*}
%		which requires us to choose
%		\begin{equation*}
%			\alpha
%			\geq \frac{1}{q_\varepsilon}\left(\log\left(\frac{2\sqrt{2}}{c\left(\vartheta,\lambda,\sigma\right)}\right) + \frac{1}{2}\log\left(\frac{\CV(\rho_0)}{\varepsilon}\right) - \log\rho_0(B_{r_\varepsilon}(v^*))\right).
%		\end{equation*}}
	
	Next, we show that the functional $\CV(\rho_t)$ decays essentially exponentially fast in time.
    More precisely, we prove that, up to time $T_\alpha$, $\CV(\rho_t)$ decays
    \begin{enumerate}[label=(\roman*),labelsep=10pt,leftmargin=35pt]
        \item at least exponentially fast (with rate $(1-\vartheta)(2\lambda-d\sigma^2)$), and \label{enumerate:proof:atleastexpdecay}
        \item at most exponentially fast (with rate $(1+\vartheta/2)(2\lambda-d\sigma^2)$).\label{enumerate:proof:atmostexpdecay}
    \end{enumerate}
	To obtain \ref{enumerate:proof:atleastexpdecay}, recall that Lemma~\ref{lem:evolution_of_objective} provides an upper bound on $\frac{d}{dt}\CV(\rho_t)$ given by
	\begin{equation}
	\begin{aligned} \label{eq:main_aux_1_no_H}
		\frac{d}{dt}\CV(\rho_t) \leq\,
		& -\left(2\lambda-d\sigma^2\right) \CV(\rho_t) + \sqrt{2}\left(\lambda+d\sigma^2\right) \sqrt{\CV(\rho_t)} \N{\conspoint{\rho_t}-\globmin}_2 \\
		&+ \frac{d\sigma^2}{2} \N{\conspoint{\rho_t}-\globmin}_2^2.
	\end{aligned}
	\end{equation}
	Combining this with the definition of $T_\alpha$ in \eqref{eq:endtime_T} we have by construction 
	\begin{align*}
		\frac{d}{dt}\CV(\rho_t)
		\leq -(1-\vartheta)\left(2\lambda-d\sigma^2\right)\CV(\rho_t)
		\quad \text{ for all } t \in (0,T_\alpha).
	\end{align*}
    Analogously, for \ref{enumerate:proof:atmostexpdecay}, by Lemma~\ref{lem:evolution_of_objective_lower}, we obtain a lower bound on $\frac{d}{dt}\CV(\rho_t)$ of the form
    \begin{align*}
    \begin{aligned}
        \frac{d}{dt}\CV(\rho_t)
        &\geq
        -\left(2\lambda - d\sigma^2\right) \CV(\rho_t)
        - \sqrt{2}\left(\lambda + d\sigma^2\right) \sqrt{\CV(\rho_t)} \N{\conspoint{\rho_t}-\globmin}_2 \\
        &\geq -(1+\vartheta/2)\left(2\lambda - d\sigma^2\right) \CV(\rho_t)
        \quad \text{ for all } t \in (0,T_\alpha),
    \end{aligned}
    \end{align*}
    where the second inequality again exploits the definition of $T_\alpha$.
	Gr\"onwall's inequality now implies for all $t \in [0,T_\alpha]$ the upper and lower bound
	\begin{align}
		\CV(\rho_t)
		&\leq \CV(\rho_0) \exp\left(- (1-\vartheta)\left(2\lambda-d\sigma^2\right) t\right), \label{eq:evolution_J_no_H}\\
		\CV(\rho_t)
		&\geq \CV(\rho_0) \exp\left(- (1+\vartheta/2)\left(2\lambda-d\sigma^2\right) t\right), \label{eq:evolution_J_no_H_lower}
	\end{align}
    thereby proving \ref{enumerate:proof:atleastexpdecay} and \ref{enumerate:proof:atmostexpdecay}.
	We further note that the definition of $T_\alpha$ in \eqref{eq:endtime_T} together with the definition of $C(t)$ and \eqref{eq:evolution_J_no_H} permits to control
	\begin{align}
		&\max_{t \in [0,T_\alpha]}\N{\conspoint{\rho_t}-\globmin}_2
		\leq \max_{t \in [0,T_\alpha]} C(t)\leq C(0).%,
		\label{eq:max_bound_distance_no_H}%\\
		%&\max_{t \in [0,T_\alpha]}\int\N{v-\globmin}_2d\rho_t(v)
		%\leq \max_{t \in [0,T_\alpha]} \sqrt{2\CV(\rho_t)} \leq \sqrt{2\CV(\rho_0)}
		%\label{eq:max_bound_moment_no_H}
	\end{align}

    To conclude, it remains to prove that $\CV(\rho_{T_\alpha}) = \varepsilon$ with $T_\alpha\in\big[\frac{1-\vartheta}{(1+\vartheta/2)}\;\!T^*,T^*\big]$.
    For this we distinguish the following three cases.
	
	\noindent
	\textbf{Case $T_\alpha \geq T^*$:}
	We can use the definition of $T^*$ in \eqref{eq:end_time_star_statement} and the time-evolution bound of $\CV(\rho_t)$ in \eqref{eq:evolution_J_no_H} to conclude that $\CV(\rho_{T^*}) \leq \varepsilon$.
	Hence, by definition of $T_\alpha$ in \eqref{eq:endtime_T} together with the continuity of $\CV(\rho_t)$, we find $\CV(\rho_{T_\alpha}) =\varepsilon$ with $T_\alpha = T^*$.
	
	\noindent
	\textbf{Case $T_\alpha < T^*$ and $\CV(\rho_{T_\alpha}) \leq \varepsilon$:}
	By continuity of $\CV(\rho_t)$, it holds for $T_\alpha$ as defined in \eqref{eq:endtime_T}, $\CV(\rho_{T_\alpha}) = \varepsilon$.
    Thus, $\varepsilon
        = \CV(\rho_{T_\alpha})
        \geq \CV(\rho_0) \exp\left(- (1+\vartheta/2)\left(2\lambda-d\sigma^2\right) T_\alpha\right)$ by \eqref{eq:evolution_J_no_H_lower}, which can be reordered as
    \begin{align*}
        \frac{1-\vartheta}{(1+\vartheta/2)} \, T^*
        =\frac{1}{(1+\vartheta/2)\left(2\lambda-d\sigma^2\right)}\log\left(\frac{\CV(\rho_0)}{\varepsilon}\right)
        \leq T_\alpha
        < T^*.
    \end{align*}
	
	\noindent
	\textbf{Case $T_\alpha < T^*$ and $\CV(\rho_{T_\alpha}) > \varepsilon$:}
	We shall show that this case can never occur by verifying that $\N{\conspoint{\rho_{T_\alpha}}-\globmin}_2 < C(T_\alpha)$ due to the choice of $\alpha$ in~\eqref{eq:alpha}.
	In fact, fulfilling simultaneously both $\CV(\rho_{T_\alpha})>\varepsilon$ and $\N{\conspoint{\rho_{T_\alpha}}-\globmin}_2 < C(T_\alpha)$ would contradict the definition of $T_\alpha$ in \eqref{eq:endtime_T} itself.
	To this end, by applying again Proposition~\ref{lem:laplace_alt} with $q_\varepsilon$ and $r_\varepsilon$,
    %as defined in~\eqref{eq:q_and_r}
    and recalling that $\varepsilon<\CV(\rho_{T_\alpha})$, we get
	\begin{align} \label{eq:proof_contradiction_1}
	\begin{split}
		\N{\conspoint{\rho_{T_\alpha}}-\globmin}_2
		&\leq \frac{\left(q_\varepsilon+\CE_{r_\varepsilon}\right)^\nu}{\eta} + \frac{\exp\left(-\alpha q_\varepsilon\right)}{\rho_{T_\alpha}\big(B_{r_\varepsilon}(v^*)\big)}\int\N{v-\globmin}_2d\rho_{T_\alpha}(v)\\
		%&\leq \frac{c\left(\vartheta,\lambda,\sigma\right)\sqrt{\varepsilon}}{2} + \frac{\exp\left(-\alpha q_\varepsilon\right)}{\rho_{T_\alpha}\big(B_{r_\varepsilon}(v^*)\big)}\sqrt{2\CV(\rho_{T_\alpha})}\\
		&< \frac{c\left(\vartheta,\lambda,\sigma\right)\sqrt{\CV(\rho_{T_\alpha})}}{2} + \frac{\exp\left(-\alpha q_\varepsilon\right)}{\rho_{T_\alpha}\big(B_{r_\varepsilon}(v^*)\big)}\sqrt{2\CV(\rho_{T_\alpha})}.
	\end{split}
	\end{align}
	Since, thanks to \eqref{eq:max_bound_distance_no_H}, we have the bound $\max_{t \in [0,T_\alpha]}\Nnormal{\conspoint{\rho_t}-\globmin}_2 \leq B$ for $B=C(0)$, which is in particular independent of $\alpha$, Proposition~\ref{lem:lower_bound_probability} guarantees that there exists a $p>0$ not depending on $\alpha$ (but depending on $B$ and $r_\varepsilon$) with
	\begin{align*}
		\rho_{T_\alpha}(B_{r_\varepsilon}(v^*))
		\geq \left(\int \phi_{r_\varepsilon}(v) \,d\rho_0(v)\right)\exp(-pT_\alpha)
		\geq \frac{1}{2}\,\rho_0\left(B_{r_\varepsilon/2}(v^*)\right) \exp(-pT^*) 
		> 0,
	\end{align*}
	where we used $\globmin\in\supp(\rho_0)$ for bounding the initial mass $\rho_0$ and the fact that $\phi_{r}$ (as defined in Equation~\eqref{eq:mollifier}) is bounded from below on $B_{r/2}(\globmin)$ by $1/2$.
	With this we can continue the chain of inequalities in~\eqref{eq:proof_contradiction_1} to obtain
	\begin{align} \label{eq:proof_contradiction_2}
	\begin{split}
		\N{\conspoint{\rho_{T_\alpha}}\!-\!\globmin}_2
		&< \frac{c\left(\vartheta,\lambda,\sigma\right)\!\sqrt{\CV(\rho_{T_\alpha})}}{2} \!+\! \frac{2\exp\left(-\alpha q_\varepsilon\right)}{\rho_0\big(B_{r_\varepsilon/2}(v^*)\big) \exp(-pT^*)}\sqrt{2\CV(\rho_{T_\alpha})}\\
		&\leq c\left(\vartheta,\lambda,\sigma\right)\sqrt{\CV(\rho_{T_\alpha})}
		= C(T_\alpha),
	\end{split}
	\end{align}
	where the first inequality in the last line holds by the choice of $\alpha$ in \eqref{eq:alpha}.
%	{\color{blue} % do not uncomment
%		To fulfill here the first inequality in the last line, we need to ensure that
%		\begin{equation*}
%			\frac{2\exp\left(-\alpha q_\varepsilon\right)}{\rho_0\big(B_{r_\varepsilon/2}(v^*)\big) \exp(-pT^*)}\sqrt{2\CV(\rho_{T_\alpha})}
%			\leq \frac{c\left(\vartheta,\lambda,\sigma\right)\sqrt{\CV(\rho_{T_\alpha})}}{2},
%		\end{equation*}
%		which requires us to choose
%		\begin{equation*}
%			\alpha
%			\geq \frac{1}{q_\varepsilon}\left(\log\left(\frac{4\sqrt{2}}{c\left(\vartheta,\lambda,\sigma\right)}\right) + pT^* - \log\rho_0\big(B_{r_\varepsilon/2}(v^*)\big)\right)
%		\end{equation*}}
	This establishes the desired contradiction, again as consequence of the continuity of the mappings~$t\mapsto\CV(\rho_{t})$ and~$t\mapsto\N{\conspoint{\rho_{t}}-\globmin}_2$.
\end{proof}}

We now consider the case of an active cutoff function~$\cutoffnoarg$ with $\cutoffnoarg(x) \!=\! 1$ whenever~\mbox{$x \!\geq\! 0$}.

% Proof of Theorem~\ref{thm:global_convergence_main} when $H\not\equiv 1$ short
\revised{\begin{proof}[Proof of Theorem~\ref{thm:global_convergence_main} when $H\not \equiv 1$]
	The proof follows the lines of the one for the inactive cutoff $H\equiv 1$, but requires some modifications since Lemmas~\ref{lem:evolution_of_objective} and~\ref{lem:evolution_of_objective_lower} need to be replaced by Lemmas~\ref{lem:evolution_of_objective_H} and~\ref{lem:evolution_of_objective_H_lower}, to derive bounds for the evolution of $\CV(\rho_t)$.
	
	As in the proof for the case $H\equiv1$ we first choose the parameter~$\alpha$ such that
	\begin{align} \label{eq:alpha_Hnot1}
	\begin{split}
		\alpha>
		\alpha_0
		:= \frac{1}{q_\varepsilon}\Bigg(\log\left(\frac{4\sqrt{2\CV(\rho_0)}}{C_\varepsilon}\right) 
		&+ \frac{p}{(1-\vartheta)\left(2\lambda-d\sigma^2\right)}\log\left(\frac{\CV(\rho_0)}{\varepsilon}\right) \\
		&- \log\rho_0\big(B_{r_\varepsilon/2}(v^*)\big)\!\Bigg),
	\end{split}
	\end{align}
	where $C_\varepsilon$ is obtained when replacing with $\varepsilon$ each $\CV(\rho_t)$ in $C(t)$ defined as
	\begin{equation} \label{eq:C_Hnot1}
	\begin{split}
		C(t) :=
		\min\Bigg\{
			&\frac{\CE_{\infty}}{2L_{\CE}}, \!
			\left(\frac{\CE_{\infty}}{2L_{\CE}}\right)^{\!1/(1+\gamma)}\!, 
			\frac{\vartheta}{4}\frac{\left(2\lambda-d\sigma^2\right)}{\sqrt{2}\left(\lambda+d\sigma^2\right)}\sqrt{\CV(\rho_t)},
			\sqrt{\frac{\vartheta}{2}\frac{\left(2\lambda-d\sigma^2\right)}{d\sigma^2}\CV(\rho_t)},\\
			&\!\!\!\!\left(\frac{\vartheta}{4}\frac{\eta^2}{L_{\CE}^{2\nu}}\frac{\left(2\lambda-d\sigma^2\right)}{\lambda}\CV(\rho_t)\right)^{\!1/(2\nu)}\!,
			\left(\frac{\vartheta}{4}\frac{\eta^2}{L_{\CE}^{2\nu}}\frac{\left(2\lambda-d\sigma^2\right)}{\lambda}\CV(\rho_t)\right)^{\!1/(2\nu(1+\gamma))}\!\Bigg\}.
	\end{split}
	\end{equation}
	Moreover, $r_\varepsilon$ is as defined before, $p$ as in \eqref{eq:def_p} with $B=C(0)$ and $r=r_\varepsilon$, and
    %$q_\varepsilon$ is given by
	\begin{align*}
		q_\varepsilon := \frac{1}{2}\min\bigg\{\!\left(\eta\frac{C_\varepsilon}{2}\right)^{1/\nu}\!,\CE_{\infty}\bigg\}.
	\end{align*}
	
	Let us now define again a time horizon~$T_\alpha$ according to \eqref{eq:endtime_T}, however with the modified definition of $C(t)$ from \eqref{eq:C_Hnot1}.
	It is straightforward to check that $T_\alpha>0$ by choice of $\alpha$ in \eqref{eq:alpha_Hnot1}.
	%Our aim is once again to show $\min_{t \in [0,T_\alpha]}\CV(\rho_t) \leq \varepsilon$ with $T_\alpha\leq T^*$ and that $\CV(\rho_t)$ decays at least exponentially fast until it reaches the prescribed accuracy~$\varepsilon$.
    Our aim is again to show $\CV(\rho_{T_\alpha}) = \varepsilon$ with $T_\alpha\in\big[\frac{1-\vartheta}{(1+\vartheta/2)}\;\!T^*,T^*\big]$ and that we have at least exponential decay of $\CV(\rho_t)$ until $T_\alpha$.
	
	Since due to Assumption~\ref{asm:local_lipschitz} and with the definition of $C(t)$ in \eqref{eq:C_Hnot1} it holds
	\begin{align} \label{eq:max_bound_distance_H}
		&\max_{t \in [0,T_\alpha]}\CE(\conspoint{\rho_t})
		\leq \max_{t \in [0,T_\alpha]}L_{\CE}(1+\N{\conspoint{\rho_t} - \globmin}_2^{\gamma})\N{\conspoint{\rho_t} - \globmin}_2
		\leq \CE_{\infty},
	\end{align}
	Lemmas~\ref{lem:evolution_of_objective_H} and \ref{lem:evolution_of_objective_H_lower} provide an upper and a lower bound for the time derivative of $\CV(\rho_t)$, which, when being combined with the definitions of $T_\alpha$ and $C(t)$ in \eqref{eq:C_Hnot1}, yield
 	\begin{align*}
		\frac{d}{dt}\CV(\rho_t)
		\leq -\left(1\!-\!\vartheta\right)\left(2\lambda\!-\!d\sigma^2\right)\CV(\rho_t)
		\quad
        \text{and}
        \quad
		\frac{d}{dt}\CV(\rho_t)
        \geq -\left(1\!+\!\vartheta/2\right)\left(2\lambda\!-\!d\sigma^2\right)\CV(\rho_t)
	\end{align*}
    for all $t \in (0,T_\alpha)$ as before.
	We can thus follow the lines of the proof for the case $H\equiv1$, since also here $C(t)$ is bounded.
	In particular, the choice of $\alpha$ in \eqref{eq:alpha_Hnot1} allows to derive the contradiction $\N{\conspoint{\rho_{T_\alpha}}-\globmin}_2 < C(T_\alpha)$ by employing Propositions~\ref{lem:laplace_alt}~and~\ref{lem:lower_bound_probability}.
\end{proof}}

\begin{remark}[Informal lower bound for $\alpha_0$] \label{rem:alpha_0}
	As mentioned in Section~\ref{subsec:convergence_probability}, insightful lower bounds on the required $\alpha_0$ in Theorem \ref{thm:global_convergence_main} may be interesting in view of better understanding the convergence of the microscopic system \eqref{eq:dyn_micro} to the mean-field limit~\eqref{eq:dyn_macro}.
	Let us therefore informally derive in what follows an instructive lower bound on the required $\alpha_0$ under the assumption that~$\CE$ satisfies Condition~\ref{asm:icp} globally with $\nu = 1/2$ and that $\CE$ is locally $L$-Lipschitz continuous around $\globmin$, i.e., in some ball $B_{R}(\globmin)$.
%	{\color{blue} % do not uncomment
%		Thus in particular $\abs{\CE(v)-\underbar{\CE}} \leq L\N{v-\globmin}_2$ for all $v\in B_{R}(\globmin)$.
%		It is straightforward to check that $\CE$ is $\widetilde{L}$-Lipschitz on $B_{q_\varepsilon/\widetilde{L}}(\globmin)$ with $\widetilde{L}:=\max\{q_\varepsilon/R,L\}$.}
	We restrict ourselves to the case of an inactive cutoff function~$H\equiv 1$.
	
	Recalling \eqref{eq:main_aux_1_no_H} in the proof of Theorem~\ref{thm:global_convergence_main}, $\alpha$ should be large enough to ensure
	\begin{align} \label{eq:inequality_aux}
		\N{\conspoint{\rho_{t}}-\globmin}_2
		&\leq c\left(\vartheta,\lambda,\sigma\right)\sqrt{\CV(\rho_{t})} \quad \text{ for all } t \in [0,T],
	\end{align}
	where $T$ is the time satisfying $\CV(\rho_{T}) = \varepsilon$.
	To achieve this, we recall that for $\indivmeasure \in \CP(\bbR^d)$ the quantitative Laplace principle in Proposition~\ref{lem:laplace_alt} with choices $q_\varepsilon := c\left(\vartheta,\lambda,\sigma\right)^2\eta^2\varepsilon/8$ and $r_\varepsilon:=\min\{R,q_\varepsilon/L\}$ for $q$ and $r$, respectively, yields
	\begin{align*}
		\N{\conspoint{\indivmeasure} - \globmin}_2
		\leq \frac{\sqrt{2q_\varepsilon}}{\eta} + \frac{\exp\left(-\alpha q_\varepsilon\right)}{\indivmeasure(B_{r_\varepsilon}(\globmin))}\int\N{v-\globmin}_2d\indivmeasure(v)
	\end{align*}
	provided that \ref{asm:icp} holds globally with $\nu=1/2$ and that $\CE$ is $L$-Lipschitz continuous on some ball $B_{R}(\globmin)$.
	It remains to choose $\alpha
		> \alpha_0$, where
	\begin{align} \label{eq:alpha_requirement}
		\alpha_0
		\!:=\! \sup_{t \in [0,T]}\frac{-8}{c\left(\vartheta,\lambda,\sigma\right)^2\eta^2\varepsilon}\log\left(\frac{c\left(\vartheta,\lambda,\sigma\right)}{2\sqrt{2}}\,\rho_t\!\left(B_{\min\{R,\,c\left(\vartheta,\lambda,\sigma\right)^2\eta^2\varepsilon/(8L)\}}(\globmin)\right)\!\right),
	\end{align}
	suggesting that $\alpha_0$ is strongly related to the time-evolution of the probability mass of $\rho_t$ around $\globmin$.
	Recalling Proposition \ref{lem:lower_bound_probability}, this mass adheres to the lower bound
	\begin{align*}
		\rho_t(B_{r}(\globmin))
		\geq \rho_0(B_{r/2}(\globmin))\exp(-pt)/2
		\quad \text{ for some } p > 0 \text{ and any } r > 0.
	\end{align*}
	However, this result is pessimistic due to its worst-case nature, and inserting it into \eqref{eq:alpha_requirement} with the corresponding $p$ as in \eqref{eq:def_p} leads to overly stringent requirements on $\alpha_0$, which are reflected by the respective second summands in~\eqref{eq:alpha} and \eqref{eq:alpha_Hnot1}.
	Rather, a successful application of the CBO method entails that the probability mass around the global minimizer increases over time, so that  $t\mapsto \rho_t(B_{r}(\globmin))$ is typically minimized at $t = 0$.
	In such case, the lower bound~\eqref{eq:alpha_requirement} becomes
	\begin{align} \label{eq:alpha_requirement_optimistic}
		\alpha_0
		= \frac{-8}{c\left(\vartheta,\lambda,\sigma\right)^2\eta^2\varepsilon}\log\left(\frac{c\left(\vartheta,\lambda,\sigma\right)}{2\sqrt{2}}\,\rho_0\!\left(B_{\min\{R,\,c\left(\vartheta,\lambda,\sigma\right)^2\eta^2\varepsilon/(8L)\}}(\globmin)\right)\!\right).
	\end{align}
\end{remark}

\section{Proof details for Section~\ref{subsec:convergence_probability}} \label{sec:proof_probabilistic_MFA}
In this section we provide the proof details for the result about the mean-field approximation of CBO, Proposition~\ref{prop:MFL}.
After giving the proof of the auxiliary Lemma~\ref{lem:PBND}, which ensures that the dynamics is to some extent bounded, we prove Proposition~\ref{prop:MFL}.

\begin{proof}[Proof of Lemma \ref{lem:PBND}]
	By combining the ideas of \cite[Lemma~3.4]{carrillo2018analytical} with a Doob-like inequality, we derive a bound for $\bbE\sup_{t\in[0,T]}\frac{1}{N}\sum_{i=1}^N \!\max\left\{\Nnormal{V_t^i}_2^4+\Nnormal{\overbar{V}_t^i}_2^4\right\}$, which ensures that $\empmeasure{t},\overbar\rho^N_t\in\CP_4(\bbR^d)$ with high probability.
	Here, $\overbar\rho^N$ denotes the empirical measure associated with the processes~$(\overbar{V}^i)_{i=1,\dots,N}$.
	%For notational simplicity, but without loss of generality, we restrict ourselves to the case~$\cutoffnoarg\equiv1$ in what follows.

	\noindent
	Employing standard inequalities shows
	\begin{align} \label{eq:proof:prop:MF:EsupVi}
	\begin{split}
		\bbE \sup_{t\in[0,T]}\N{V_t^i}_2^{4}
		\lesssim 
			\bbE \N{V_0^i}_2^{4}
			&+ \lambda^{4} \; \bbE \sup_{t\in[0,T]}\N{\int_0^t \left(V_\tau^i-\conspoint{\empmeasure{\tau}}\right) d\tau}_2^{4} \\
			&+ \sigma^{4} \; \bbE \sup_{t\in[0,T]}\N{\int_0^t \N{V_\tau^i-\conspoint{\empmeasure{\tau}}}_2 dB_\tau^i}_2^{4},
	\end{split}
	\end{align}
	where we note that the expression $\int_0^t \Nnormal{V_\tau^i-\conspoint{\empmeasure{\tau}}}_2 \,dB_\tau^i$ appearing in the third term of the right-hand side is a martingale, which is a consequence of~\cite[Corollary~3.2.6]{oksendal2013stochastic} combined with the regularity established in~\cite[Lemma~3.4]{carrillo2018analytical}.
	This allows to apply the Burkholder-Davis-Gundy inequality~\cite[Chapter~IV, Theorem~4.1]{RevuzYor1999martingales}, which yields
	\begin{align*}
	\begin{split}
		\bbE \sup_{t\in[0,T]}\N{\int_0^t \N{V_\tau^i-\conspoint{\empmeasure{\tau}}}_2 dB_\tau^i}_2^{4}
		&\lesssim \bbE \left(\int_0^T \N{V_\tau^i-\conspoint{\empmeasure{\tau}}}_2^2 d\tau\right)^2.
	\end{split}
	\end{align*}
	Let us stress that the constant appearing in the latter estimate depends on the dimension $d$.
	Further bounding this as well as the second term of the right-hand side in~\eqref{eq:proof:prop:MF:EsupVi} by means of Jensen's inequality and  utilizing~\cite[Lemma~3.3]{carrillo2018analytical} yields
	\begin{align} \label{eq:proof:prop:MF:EsupVi_2}
	\begin{split}
		\bbE \sup_{t\in[0,T]}\N{V_t^i}_2^{4}
		\leq 
			C \left(1 + \bbE \N{V_0^i}_2^{4}
			+ \bbE \int_0^T \N{V_\tau^i}_2^{4} + \int \N{v}_2^{4} d\empmeasure{\tau}(v) \, d\tau \right)
	\end{split}
	\end{align}
	with a constant~$C=C(\lambda,\sigma,d,T, b_1,b_2)$.
	Averaging~\eqref{eq:proof:prop:MF:EsupVi_2} over $i$ allows to bound
	\begin{align*} %\label{eq:proof:prop:MF:EsupVi_3}
	\begin{split}
		\bbE \!\sup_{t\in[0,T]} \int \!\N{v}_2^{4} d\empmeasure{t}(v)
		\leq 
			C \left(1 + \bbE \!\int\! \N{v}_2^{4} d\empmeasure{0}(v)
			+ 2\int_0^T \!\bbE \!\sup_{\widehat\tau\in[0,\tau]} \int \!\N{v}_2^{4} d\empmeasure{\widehat\tau}(v) \, d\tau \right),
	\end{split}
	\end{align*}
	which, after applying Gr\"onwall's inequality, ensures that the left-hand side is bounded independently of $N$ by a constant~$K\!=\!K(\lambda,\sigma,d,T, b_1,b_2)$.
	With analogous arguments we can show $\bbE \sup_{t\in[0,T]} \!\int\! \Nnormal{v}_2^{4}\, d\overbar\rho^N_{t}(v) \!\leq\! K$.
	Equation~\eqref{eq:prop:MFL:OmegaM} follows now from Markov's inequality.
\end{proof}

We now present the proof of Proposition~\ref{prop:MFL}.

\begin{proof}[Proof of Proposition~\ref{prop:MFL}]
	By exploiting the boundedness thanks to Lemma~\ref{lem:PBND} through a cutoff technique, we can follow the steps taken in~\cite[Theorem~3.1]{fornasier2020consensus_hypersurface_wellposedness}.

	\noindent
	Let us define the cutoff function
	\begin{align*} %\label{eq:proof:prop:MF:cutoff}
		I_M(t)=
		\begin{cases}
			1, & \text{ if } \frac{1}{N}\sum_{i=1}^N \max\left\{\N{V_\tau^i}_2^4,\N{\overbar{V}_\tau^i}_2^4\right\} \leq M \text{ for all } \tau\in[0,t],\\
			0, & \text{ else},
		\end{cases}
	\end{align*}
	which is adapted to the natural filtration and has the property $I_M(t)=I_M(t)I_M(\tau)$ for all $\tau\in[0,t]$.
	With Jensen's inequality and It\^o isometry this allows to derive
	\begin{align} \label{eq:proof:prop:MF:ENV-Vbar}
	\begin{split}
		\bbE\N{V_t^i-\overbar{V}_t^i}_2^2I_M(t)
		\lesssim %2\bbE\N{V_0^i-\overbar{V}_0^i}_2^2 +
			c \int_0^t \bbE\left(\N{V_\tau^i-\overbar{V}_\tau^i}_2^2 + \N{\conspoint{\empmeasure{\tau}}-\conspoint{\rho_\tau}}_2^2\right)I_M(\tau) \, d\tau
	\end{split}
	\end{align}
	for $c = \left(\lambda^2T+\sigma^2\right)$.
	Here we directly used that the processes $V_t^i$ and $\overbar{V}_t^i$ share the initial data as well as the Brownian motion paths.
	In what follows, let us denote by $\monopmeasure{\tau}$ the empirical measure of the processes~$\overbar{V}_\tau^i$.
	Then, by using the same arguments as in the proofs of~\cite[Lemma~3.2]{carrillo2018analytical} and \cite[Lemma~3.1]{fornasier2020consensus_hypersurface_wellposedness} with the care of taking into consideration the multiplication with the random variable~$I_M(\tau)$, we obtain
	\begin{align*}
	\begin{split}
		\bbE\!\N{\conspoint{\empmeasure{\tau}}\!-\!\conspoint{\rho_\tau}}_2^2\!I_M(\tau)
        &\!\!\lesssim \!\bbE\!\N{\conspoint{\empmeasure{\tau}}\!-\!\conspoint{\monopmeasure{\tau}}}_2^2 \!I_M(\tau) \!+\! \bbE\!\N{\conspoint{\monopmeasure{\tau}}\!-\!\conspoint{\rho_\tau}}_2^2 \!I_M(\tau)\! \\
		&\!\!\leq C\left(\max_{i=1,\dots,N}\bbE \N{V_\tau^i\!-\!\overbar{V}_\tau^i}_2^2I_M(\tau) + N^{-1}\right)
	\end{split}
	\end{align*}
	for a constant $C=C(\alpha,C_1,C_2,M,\CM_2,b_1,b_2)$.
	After plugging the latter into~\eqref{eq:proof:prop:MF:ENV-Vbar} and taking the maximum over $i$, the quantitative mean-field approximation result~\eqref{eq:prop:MFL:mean-field approximation} follows from an application of Gr\"onwall's inequality after recalling the definition of the conditional expectation and noting that $\mathbbm{1}_{\Omega_M}\leq I_M(t)$ pointwise and for all $t\in[0,T]$.
\end{proof}

\section{Conclusions} \label{sec:conclusions}
In this paper we establish the convergence of consensus-based optimization (CBO) methods to the global minimizer.
The proof technique is based on the novel insight that  the dynamics of individual agents follow, on average over all realizations of Brownian motion paths, the gradient flow dynamics associated with the map $v\mapsto \Nnormal{v-\globmin}_2^2$, where $\globmin$ is the global minimizer of the objective~$\CE$.
This implies that CBO methods are barely influenced by the local energy landscape of $\CE$, suggesting a high degree of robustness and versatility of the method.
As opposed to \revisedTwo{restrictive} concentration conditions on the initial agent configuration $\rho_0$ in the analyses in \cite{carrillo2018analytical,ha2020convergenceHD,ha2021convergence,fornasier2020consensus_sphere_convergence}, \revisedTwo{our} result holds under \revisedTwo{mild} assumptions about the initial distribution $\rho_0$.
Furthermore, we merely require local Lipschitz continuity \revisedTwo{and a certain tractability condition} about the objective $\CE$, relaxing the regularity requirement~$\CE\in\CC^2(\bbR^d)$ \revisedTwo{together with further assumptions} from prior works.
In order to demonstrate the relevance of the result of convergence in mean-field law for establishing a complete convergence proof of the original numerical scheme~\eqref{eq:dyn_micro_discrete}, we prove a probabilistic quantitative result about the mean-field approximation, which connects the finite particle regime with the mean-field limit.
With this we close the gap regarding the mean-field approximation of CBO and provide the first, and so far unique, holistic convergence proof of CBO on the plane.

We believe that the proposed analysis strategy can be adopted to other recently developed adaptations of the CBO algorithm, such as CBO methods tailored to manifold optimization problems \cite{fornasier2020consensus_hypersurface_wellposedness,fornasier2020consensus_sphere_convergence}, \revised{polarized CBO adjusted to identify multiple minimizers simultaneously~\cite{bungert2022polarized},} as well as related metaheuristics including, for instance, Particle Swarm Optimization~\cite{kennedy1995particle,grassi2020particle,qiu2022PSOconvergence}, \revised{which can be regarded as a second-order variant of CBO with inertia~\cite{grassi2020particle,cipriani2021zero}.}
For CBO with anisotropic Brownian motions, which are especially relevant in high-dimensional optimization problems~\cite{carrillo2019consensus}, \revised{for CBO with memory effects and gradient information, which can be beneficial in signal processing and machine learning applications~\cite{riedl2022leveraging,trillos2023FedCBO},} \revised{for CBO reconfigured for multi-objective optimization,} as well as for constrained CBO, this has already been done in~\cite{fornasier2021convergence}, \cite{riedl2022leveraging}, \cite{borghi2022adaptive}, and~\cite{borghi2021constrained}, respectively.

\section*{Acknowledgments} 
The authors would like to profusely thank Hui Huang for many fruitful and stimulating discussions about the topic.

This work has been funded by the German Federal Ministry of Education and Research and the Bavarian State Ministry for Science and the Arts.
The authors of this work take full responsibility for its content.
MF further acknowledges the support of the DFG Project ``Identification of Energies from Observations of Evolutions'' and the DFG SPP 1962 ``Non-smooth and Complementarity-Based Distributed Parameter Systems: Simulation and Hierarchical Optimization''.
TK acknowledges the support of the Technical University of Munich for hosting him while conducting the work on this manuscript.
KR acknowledges the financial support from the Technical University of Munich -- Institute for Ethics in Artificial Intelligence (IEAI).

\bibliographystyle{abbrv}
\bibliography{biblio.bib}

\newpage
\renewcommand*{\thesection}{\Alph{section}}
\setcounter{section}{1}
\section*{Appendix: Extended proof details of Section~\ref{subsec:convergence_probability}}
%\pagenumbering{gobble}

\begin{proof}[Extended proof of Lemma~\ref{lem:PBND}]
By combining the ideas of \cite[Lemma~3.4]{carrillo2018analytical} with a Doob-like inequality, we derive a bound for $\bbE\sup_{t\in[0,T]}\frac{1}{N}\sum_{i=1}^N \max\left\{\Nnormal{V_t^i}_2^4,\Nnormal{\overbar{V}_t^i}_2^4\right\}$, which ensures that $\empmeasure{t},\overbar\rho^N_t\in\CP_4(\bbR^d)$ with high probability.
Here, $\overbar\rho^N$ denotes the empirical measure associated with the processes~$(\overbar{V}^i)_{i=1,\dots,N}$.
For notational simplicity, but without loss of generality, we restrict ourselves to the case~$\cutoffnoarg\equiv1$ in what follows.

\noindent
By employing the inequality $(x+y)^q\leq 2^{q-1} (x^q+y^q)$, $q\geq1$ we note that
\begin{equation*}
\begin{split}
	\N{V_t^i}_2^{2p}
	\leq 
		2^{2p-1} \N{V_0^i}_2^{2p}
		&+ 2^{2(2p-1)}\lambda^{2p} \N{\int_0^t \left(V_\tau^i-\conspoint{\empmeasure{\tau}}\right) d\tau}_2^{2p} \\
		&+ 2^{2(2p-1)}\sigma^{2p} \N{\int_0^t \N{V_\tau^i-\conspoint{\empmeasure{\tau}}}_2 dB_\tau^i}_2^{2p}
\end{split}
\end{equation*}
for all~$i=1,\dots,N$.
Taking first the supremum over $t\in[0,T]$ and consecutively the expectation on both sides of the former inequality yields
\begin{equation} \label{app:eq:proof:prop:MF:EsupVi}
\begin{split}
	\bbE \sup_{t\in[0,T]}\N{V_t^i}_2^{2p}
	\leq 
		2^{2p-1} \bbE \N{V_0^i}_2^{2p}
		&+ 2^{2(2p-1)}\lambda^{2p} \bbE \sup_{t\in[0,T]}\N{\int_0^t \left(V_\tau^i-\conspoint{\empmeasure{\tau}}\right) d\tau}_2^{2p} \\
		&+ 2^{2(2p-1)}\sigma^{2p} \bbE \sup_{t\in[0,T]}\N{\int_0^t \N{V_\tau^i-\conspoint{\empmeasure{\tau}}}_2 dB_\tau^i}_2^{2p}.
\end{split}
\end{equation}
The second term on the right-hand side of~\eqref{app:eq:proof:prop:MF:EsupVi} can be further bounded by
\begin{equation} \label{app:eq:proof:prop:MF:EsupVi_term2}
\begin{split}
	\bbE \sup_{t\in[0,T]}\N{\int_0^t \left(V_\tau^i-\conspoint{\empmeasure{\tau}}\right) d\tau}_2^{2p}
	&\leq \max\{1,T^{2p-1}\} \bbE \sup_{t\in[0,T]}\int_0^t \N{V_\tau^i-\conspoint{\empmeasure{\tau}}}_2^{2p} d\tau \\
	&\leq \max\{1,T^{2p-1}\} \bbE \int_0^T \N{V_\tau^i-\conspoint{\empmeasure{\tau}}}_2^{2p} d\tau
\end{split}
\end{equation}
as a consequence of Jensen's inequality.
For the third term on the right-hand side of~\eqref{app:eq:proof:prop:MF:EsupVi} we first note that the expression $\int_0^t \N{V_\tau^i-\conspoint{\empmeasure{\tau}}}_2 dB_\tau^i$ is a martingale.
This is due to~\cite[Corollary~3.2.6]{oksendal2013stochastic} since its expected quadratic variation is finite as required by~\cite[Definition~3.1.4]{oksendal2013stochastic}.
The latter immediately follows from the regularity established in~\cite[Lemma~3.4]{carrillo2018analytical}.
Therefore we can apply the Burkholder-Davis-Gundy inequality~\cite[Chapter~IV, Theorem~4.1]{RevuzYor1999martingales}, which gives for a generic constant~$C_{2p}$ depending only on the dimension~$d$ the bound
\begin{equation} \label{app:eq:proof:prop:MF:EsupVi_term3}
\begin{split}
	\bbE \sup_{t\in[0,T]}\N{\int_0^t \N{V_\tau^i-\conspoint{\empmeasure{\tau}}}_2 dB_\tau^i}_2^{2p}
	&\leq C_{2p} \sup_{t\in[0,T]}\bbE \left(\int_0^t \N{V_\tau^i-\conspoint{\empmeasure{\tau}}}_2^2 d\tau\right)^p \\
	&\leq C_{2p} \max\{1,T^{p-1}\} \bbE \int_0^T \N{V_\tau^i-\conspoint{\empmeasure{\tau}}}_2^{2p} d\tau.
\end{split}
\end{equation}
Here, the latter step is again due to Jensen's inequality.
The right-hand sides of \eqref{app:eq:proof:prop:MF:EsupVi_term2} and \eqref{app:eq:proof:prop:MF:EsupVi_term3} can be further bounded since
\begin{equation} \label{app:eq:proof:prop:MF:EsupVi_terms_aux}
\begin{split}
	\bbE \int_0^T \N{V_\tau^i-\conspoint{\empmeasure{\tau}}}_2^{2p} d\tau
	&\leq 2^{2p-1}\bbE \int_0^T \left(\N{V_\tau^i}_2^{2p} + \N{\conspoint{\empmeasure{\tau}}}_2^{2p}\right) d\tau \\
	%&\leq 2^{2p-1}\bbE \int_0^T \left(\N{V_\tau^i}_2^{2p} + \left(b_1 + b_2 \int \N{v}_2^2 d\empmeasure{\tau}(v)\right)^{p}\right) d\tau \\
	&\leq 2^{2p-1}\bbE \int_0^T \left(\N{V_\tau^i}_2^{2p} + 2^{p-1}\left(b_1^p + b_2^p \int \N{v}_2^{2p} d\empmeasure{\tau}(v)\right)\right) d\tau,
\end{split}
\end{equation}
where in the last step we made use of~\cite[Lemma~3.3]{carrillo2018analytical}, which shows that
\begin{equation*}
\begin{split}
	\N{\conspoint{\empmeasure{\tau}}}_2^{2}
	&\leq \int \N{v}_2^2 \frac{\omegaa(v)}{\N{\omegaa}_{L_1(\empmeasure{t})}} \,d\empmeasure{\tau}(v)
	\leq b_1 + b_2 \int \N{v}_2^2 d\empmeasure{\tau}(v),
\end{split}
\end{equation*}
with $b_1=0$ and $b_2=e^{\alpha(\overline{\CE}-\underline{\CE})}$ in the case that $\CE$ is bounded, and
\begin{equation} \label{app:eq:proof:constantsb1b2}
	b_1=C_4^2+b_2
	\quad\text{and}\quad
	b_2=2\frac{C_2}{C_3}\left(1+\frac{1}{\alpha C_3}\frac{1}{C_4^2}\right)
\end{equation}
in the case that $\CE$ satisfies the coercivity assumption~\eqref{eq:quadratic_growth_condition_car}.
Inserting the upper bounds~\eqref{app:eq:proof:prop:MF:EsupVi_term2} and \eqref{app:eq:proof:prop:MF:EsupVi_term3} together with the estimate~\eqref{app:eq:proof:prop:MF:EsupVi_terms_aux} into~\eqref{app:eq:proof:prop:MF:EsupVi} yields 
%\begin{equation*}
%\begin{split}
%	\bbE \sup_{t\in[0,T]}\N{V_t^i}_2^{2p}
%	&\leq 
%		2^{2p-1} \bbE \N{V_0^i}_2^{2p}
%		+ 2^{2(2p-1)}\max\{1,T^{2p-1}\}\left(\lambda^{2p}+C_{2p} \sigma^{2p}\right) \cdot \\
%		&\qquad\quad\;\cdot \bbE \int_0^T \N{V_\tau^i-\conspoint{\empmeasure{\tau}}}_2^{2p} d\tau \\
%	&\leq 
%		2^{2p-1} \bbE \N{V_0^i}_2^{2p}
%		+ 2^{7p-4}\max\{1,T^{2p-1}\}\left(\lambda^{2p}+C_{2p} \sigma^{2p}\right) \cdot \\
%		&\qquad\quad\;\cdot \left(b_1^pT + \max\{1,b_2^p\} \bbE \int_0^T \left(\N{V_\tau^i}_2^{2p} + \int \N{v}_2^{2p} d\empmeasure{\tau}(v)\right) d\tau\right) \\
%\end{split}
%\end{equation*}
\begin{equation} \label{app:eq:proof:prop:MF:EsupVi_2}
\begin{split}
	\bbE \sup_{t\in[0,T]}\N{V_t^i}_2^{2p}
	\leq 
		C \left(1 + \bbE \N{V_0^i}_2^{2p}
		+ \bbE \int_0^T \N{V_\tau^i}_2^{2p} + \int \N{v}_2^{2p} d\empmeasure{\tau}(v) \, d\tau \right)
\end{split}
\end{equation}
with a constant~$C=C(p,\lambda,\sigma,d,T, b_1,b_2)$.
Averaging~\eqref{app:eq:proof:prop:MF:EsupVi_2} over $i$ allows to bound
\begin{equation*} %\label{app:eq:proof:prop:MF:EsupVi_3}
\begin{split}
	\bbE \sup_{t\in[0,T]} \int \N{v}_2^{2p} d\empmeasure{t}(v)
	&\leq 
		C \left(1 + \bbE \int \N{v}_2^{2p} d\empmeasure{0}(v)
		+ 2\int_0^T \bbE \int \N{v}_2^{2p} d\empmeasure{\tau}(v) \, d\tau \right), \\
	&\leq 
		C \left(1 + \bbE \int \N{v}_2^{2p} d\empmeasure{0}(v)
		+ 2\int_0^T \bbE \sup_{\widetilde\tau\in[0,\tau]} \int \N{v}_2^{2p} d\empmeasure{\widetilde\tau}(v) \, d\tau \right),
\end{split}
\end{equation*}
which ensures after an application of Gr\"onwall's inequality, that $\bbE \sup_{t\in[0,T]} \int \N{v}_2^{2p} d\empmeasure{t}(v)$ is bounded independently of $N$ provided $\rho_0\in\CP_{2p}(\bbR^d)$.
Since this holds by the assumption $\rho_0\in\CP_{4}(\bbR^d)$ for $p=2$, there exists a constant~$K=K(\lambda,\sigma,d,T, b_1,b_2)$, in particular independently of $N$, such that $\bbE \sup_{t\in[0,T]} \int \N{v}_2^{4} d\empmeasure{t}(v) \leq K$.

\noindent
Following analogous arguments for $\overbar{V}_t^i$ allows to derive
\begin{equation}
\begin{split}
	\bbE \sup_{t\in[0,T]}\N{\overbar{V}_t^i}_2^{2p}
	\leq 
		C \left(1 + \bbE \N{\overbar{V}_0^i}_2^{2p}
		+ \bbE \int_0^T \N{\overbar{V}_\tau^i}_2^{2p} + \int \N{v}_2^{2p} d\rho_\tau(v) \, d\tau \right)
\end{split}
\end{equation}
in place of~\eqref{app:eq:proof:prop:MF:EsupVi_2}.
Noticing that $\int \N{v}_2^{2p} d\rho_\tau(v) = \bbE\N{\overbar{V}_\tau^i}_2^{2p}$ for all $\tau\in[0,T]$ and averaging the latter over $i$ directly permits to prove $\bbE \sup_{t\in[0,T]} \int \N{v}_2^{2p} d\overbar\rho_{t}(v) \leq K$ by applying Gr\"onwall's inequality, again provided that $\rho_0\in\CP_{2p}(\bbR^d)$.
With this being the case for $p=2$ and by choosing $K$ sufficiently large for either estimate, the statement follows from a union bound and Markov's inequality.
More precisely,
\begin{equation*}
\begin{split}
	&\bbP\left(\sup_{t\in[0,T]} \frac{1}{N}\sum_{i=1}^N \max\left\{\N{V_t^i}_2^{4},\N{\overbar{V}_t^i}_2^{4}\right\} > M\right) \\
	&\qquad\qquad\qquad
	\leq \bbP\left(\sup_{t\in[0,T]} \frac{1}{N}\sum_{i=1}^N \N{V_t^i}_2^{4} > M\right) + \bbP\left(\sup_{t\in[0,T]} \frac{1}{N}\sum_{i=1}^N \N{\overbar{V}_t^i}_2^{4} > M\right) \\
	&\qquad\qquad\qquad
	\leq \frac{\bbE \sup_{t\in[0,T]} \frac{1}{N}\sum_{i=1}^N \Nnormal{V_t^i}_2^{4}}{M} + \frac{\bbE \sup_{t\in[0,T]} \frac{1}{N}\sum_{i=1}^N \Nnormal{\overbar V_t^i}_2^{4}}{M}
	\leq 2 \frac{K}{M}.
\end{split}
\end{equation*}
\end{proof}

\begin{proof}[Extended proof of Proposition~\ref{prop:MFL}]
By exploiting the boundedness of the dynamics established in Lemma~\ref{lem:PBND} through a cutoff technique, we can follow the steps taken in~\cite[Theorem~3.1]{fornasier2020consensus_hypersurface_wellposedness}.
For notational simplicity, we restrict ourselves to the case~$\cutoffnoarg\equiv1$ in what follows.
However, at the expense of minor technical modifications, the proof can be extended to the case of a Lipschitz-continuous active function~$\cutoffnoarg$.

\noindent
Let us define the cutoff function
\begin{equation} \label{app:eq:proof:prop:MF:cutoff}
	I_M(t)=
	\begin{cases}
		1, & \text{ if } \frac{1}{N}\sum_{i=1}^N \max\left\{\N{V_\tau^i}_2^4,\N{\overbar{V}_\tau^i}_2^4\right\} \leq M \text{ for all } \tau\in[0,t],\\
		0, & \text{ else},
	\end{cases}
\end{equation}
which is adapted to the natural filtration and has the property $I_M(t)=I_M(t)I_M(\tau)$ for all $\tau\in[0,t]$.
This allows to obtain for $\bbE\N{V_t^i-\overbar{V}_t^i}_2^2I_M(t)$ the inequality
%\begin{equation*} %\label{app:eq:proof:prop:MF:ENV-Vbar_1}
%\begin{split}
%	V_t^i-\overbar{V}_t^i
%	\leq \left(V_0^i-\overbar{V}_0^i\right)
%		&-\lambda \int_0^t \left(\left(V_\tau^i-\conspoint{\empmeasure{\tau}}\right)-\left(\overbar{V}^i_\tau - \conspoint{\rho_\tau}\right)\right) d\tau \\
%		&+\sigma \int_0^t \left(\N{V_\tau^i-\conspoint{\empmeasure{\tau}}}_2 - \N{\overbar{V}^i_\tau-\conspoint{\rho_\tau}}_2\right) dB_\tau^i \\
%	\leq \left(V_0^i-\overbar{V}_0^i\right)
%		&-\lambda \int_0^t \left(\left(V_\tau^i-\overbar{V}^i_\tau\right) - \left(\conspoint{\empmeasure{\tau}}-\conspoint{\rho_\tau}\right)\right) d\tau \\
%		&+\sigma \int_0^t \left(\N{V_\tau^i-\conspoint{\empmeasure{\tau}}}_2 - \N{\overbar{V}^i_\tau-\conspoint{\rho_\tau}}_2\right) dB_\tau^i
%\end{split}
%\end{equation*}
%\begin{equation*} %\label{app:eq:proof:prop:MF:ENV-Vbar_2}
%\begin{split}
%	\N{V_t^i-\overbar{V}_t^i}_2^2
%	\leq 2\N{V_0^i-\overbar{V}_0^i}_2^2
%		&+4\lambda^2 \N{\int_0^t \left(\left(V_\tau^i-\conspoint{\empmeasure{\tau}}\right)-\left(\overbar{V}^i_\tau - \conspoint{\rho_\tau}\right)\right) d\tau}_2^2 \\
%		&+4\sigma^2 \N{\int_0^t \left(\N{V_\tau^i-\conspoint{\empmeasure{\tau}}}_2 - \N{\overbar{V}^i_\tau-\conspoint{\rho_\tau}}_2\right) dB_\tau^i}_2^2 \\
%\end{split}
%\end{equation*}
\begin{equation*} %\label{app:eq:proof:prop:MF:ENV-Vbar_3}
\begin{split}
	\bbE\N{V_t^i-\overbar{V}_t^i}_2^2I_M(t)
	&\leq 2\bbE\N{V_0^i-\overbar{V}_0^i}_2^2 \\
		&\quad\, +4\lambda^2 \bbE\N{\int_0^t \left(\left(V_\tau^i-\conspoint{\empmeasure{\tau}}\right)-\left(\overbar{V}^i_\tau - \conspoint{\rho_\tau}\right)\right) I_M(\tau)\, d\tau}_2^2 \\
		&\quad\, +4\sigma^2 \bbE\N{\int_0^t \left(\N{V_\tau^i-\conspoint{\empmeasure{\tau}}}_2 - \N{\overbar{V}^i_\tau-\conspoint{\rho_\tau}}_2\right) I_M(\tau)\, dB_\tau^i}_2^2 \\
	&\leq 2\bbE\N{V_0^i-\overbar{V}_0^i}_2^2 \\
		&\quad\, +8\lambda^2T \bbE\int_0^t \left(\N{V_\tau^i-\overbar{V}^i_\tau}_2^2 + \N{\conspoint{\empmeasure{\tau}} - \conspoint{\rho_\tau}}_2^2\right)I_M(\tau)\, d\tau \\
		&\quad\, +8\sigma^2 \bbE\int_0^t \left(\N{V_\tau^i-\overbar{V}^i_\tau}_2^2 + \N{\conspoint{\empmeasure{\tau}}-\conspoint{\rho_\tau}}_2^2\right) I_M(\tau)\, d\tau,
\end{split}
\end{equation*}
where we used in the first step that the processes~$V_\tau^i$ and $\overbar{V}^i_\tau$ share the Brownian motion paths, and in the second step both It\^o isometry and Jensen's inequality.
Noting further that the processes also share the initial data, we are left with
\begin{equation} \label{app:eq:proof:prop:MF:ENV-Vbar}
\begin{split}
	\bbE\N{V_t^i-\overbar{V}_t^i}_2^2I_M(t)
	\leq %2\bbE\N{V_0^i-\overbar{V}_0^i}_2^2 +
		8\left(\lambda^2T+\sigma^2\right) \int_0^t \bbE\left(\N{V_\tau^i-\overbar{V}_\tau^i}_2^2 + \N{\conspoint{\empmeasure{\tau}}-\conspoint{\rho_\tau}}_2^2\right)I_M(\tau) \, d\tau,
\end{split}
\end{equation}
where it remains to control $\bbE\Nnormal{\conspoint{\empmeasure{\tau}}-\conspoint{\rho_\tau}}_2^2I_M(\tau)$.
By means of Lemmas~\ref{app:lem:Wstability_conspoint} and~\ref{app:lem:sampling_conspoint} below we have the bound
\begin{equation} \label{app:eq:proof:prop:MFL:conspoint:difference}
\begin{split}
	\bbE\N{\conspoint{\empmeasure{\tau}}-\conspoint{\rho_\tau}}_2^2I_M(\tau)
	&\leq 2\bbE\N{\conspoint{\empmeasure{\tau}}-\conspoint{\monopmeasure{\tau}}}_2^2 I_M(\tau) + 2\bbE\N{\conspoint{\monopmeasure{\tau}}-\conspoint{\rho_\tau}}_2^2 I_M(\tau) \\
	&\leq C\left(\frac{1}{N}\sum_{i=1}^N \bbE\N{V^i_\tau-\overbar{V}^i_\tau}_2^2I_M(\tau) + N^{-1}\right) \\
	&\leq C\left(\max_{i=1,\dots,N}\bbE \N{V_\tau^i-\overbar{V}_\tau^i}_2^2I_M(\tau) + N^{-1}\right)
\end{split}
\end{equation}
for a constant $C=C(\alpha,C_1,C_2,M,\CM_2,b_1,b_2)$.
After integrating the bound~\eqref{app:eq:proof:prop:MFL:conspoint:difference} into~\eqref{app:eq:proof:prop:MF:ENV-Vbar} and taking the maximum over $i$ we are left with 
\begin{equation} \label{app:eq:proof:prop:MF:ENV-Vbar_sup}
\begin{split}
	\max_{i=1,\dots,N}\bbE\N{V_t^i-\overbar{V}_t^i}_2^2I_M(t)
	\leq %2\bbE\N{V_0^i-\overbar{V}_0^i}_2^2 +
		C \int_0^t \max_{i=1,\dots,N}\bbE\N{V_\tau^i-\overbar{V}_\tau^i}_2^2I_M(\tau) \, d\tau + CTN^{-1},
\end{split}
\end{equation}
where $C$ depends additionally on $\lambda$, $\sigma$ and $T$, i.e., $C=C(\alpha,\lambda,\sigma,T,C_1,C_2,M,\CM_2,b_1,b_2)$.
The second part of the statement now follows from an application of Gr\"onwall's inequality and by noting that $\mathbbm{1}_{\Omega_M}\leq I_M(t)$ pointwise and for all $t\in[0,T]$.
\end{proof}

\begin{lemma} \label{app:lem:Wstability_conspoint}
	Let $I_M$ be as defined in~\eqref{app:eq:proof:prop:MF:cutoff}.
	Then, under the assumptions of Theorem~\ref{thm:well-posedness_FP}, it holds
	\begin{equation*} %\label{app:eq:proof:lem:Wstability_conspoint:difference_norm}
	\begin{split}
		\N{\conspoint{\empmeasure{\tau}}-\conspoint{\monopmeasure{\tau}}}_2^2I_M(\tau)
		&\leq C \frac{1}{N}\sum_{i=1}^N \N{V^i_\tau-\overbar{V}^i_\tau}_2^2I_M(\tau)
	\end{split}
	\end{equation*}
	for a constant $C=C(\alpha,C_1,C_2,M)$.
\end{lemma}

\begin{proof}
	The proof follows the steps taken in~\cite[Lemmas~3.1 and~3.2]{carrillo2018analytical}.
	
	\noindent
	Let us first note that by exploiting that the quantity~$\frac{1}{N}\sum_{i=1}^N \N{V^i_\tau}_2^4$ is bounded uniformly by $M$ due to the multiplication with~$I_M(\tau)$, we obtain with Jensen's inequality that
	\begin{equation} \label{app:eq:proof:lem:Wstability_conspoint:aux_statement}
	\begin{split}
		\frac{e^{-\alpha\underbarscript\CE}\,I_M(\tau)}{\frac{1}{N}\sum_{i=1}^N \omegaa(V^i_\tau)}
		%= \frac{I_M(\tau)}{\frac{1}{N}\sum_{i=1}^N \exp\big(-\alpha(\CE(V^i_\tau)-\underbar\CE)\big)}
		&\leq \frac{I_M(\tau)}{\exp\!\big(\!-\!\alpha\frac{1}{N}\sum_{i=1}^N (\CE(V^i_\tau)-\underbar\CE)\big)}
		\leq \frac{I_M(\tau)}{\exp\!\big(\!-\!\alpha C_2\frac{1}{N}\sum_{i=1}^N (1\!+\!\Nnormal{V^i_\tau}_2^2)\big)}\\
		&%\leq \frac{1}{\exp\big(-\alpha C_2 (1+\sqrt{M})\big)}
		\leq \exp\big(\alpha C_2 (1\!+\!\sqrt{M})\big) =: c_M,
	\end{split}
	\end{equation}
	where, in the second inequality, we used the assumption~\eqref{eq:quadratic_boundedness_condition_car} on $\CE$.
	An analogous statement can be obtained for the processes~$\overbar{V}^i_\tau$.
	
	\noindent
%	For the difference between $\conspoint{\empmeasure{\tau}}$ and $\conspoint{\monopmeasure{\tau}}$ we have the decomposition
%	\begin{equation} \label{app:eq:proof:lem:Wstability_conspoint:difference}
%	\begin{split}
%		\left(\conspoint{\empmeasure{\tau}}-\conspoint{\monopmeasure{\tau}}\right)I_M(\tau)
%		&= \left(
%			\frac{\sum_{i=1}^N V^i_\tau \omegaa(V^i_\tau)}{\sum_{j=1}^N \omegaa(V^j_\tau)}
%			-
%			\frac{\sum_{i=1}^N \overbar{V}^i_\tau \omegaa(\overbar{V}^i_\tau)}{\sum_{j=1}^N \omegaa(\overbar{V}^j_\tau)}
%			\right)
%			I_M(\tau) \\
%%		&= \left(
%%			\left(
%%			\frac{\sum_{i=1}^N V^i_\tau \omegaa(V^i_\tau)}{\sum_{j=1}^N \omegaa(V^j_\tau)}
%%			-
%%			\frac{\sum_{i=1}^N \overbar{V}^i_\tau \omegaa(V^i_\tau)}{\sum_{j=1}^N \omegaa(V^j_\tau)}
%%			\right)
%%			+
%%			\left(
%%			\frac{\sum_{i=1}^N \overbar{V}^i_\tau \omegaa(V^i_\tau)}{\sum_{j=1}^N \omegaa(V^j_\tau)}
%%			-
%%			\frac{\sum_{i=1}^N \overbar{V}^i_\tau \omegaa(\overbar{V}^i_\tau)}{\sum_{j=1}^N \omegaa(V^j_\tau)}
%%			\right)
%%			+
%%			\left(
%%			\frac{\sum_{i=1}^N \overbar{V}^i_\tau \omegaa(\overbar{V}^i_\tau)}{\sum_{j=1}^N \omegaa(V^j_\tau)}
%%			-
%%			\frac{\sum_{i=1}^N \overbar{V}^i_\tau \omegaa(\overbar{V}^i_\tau)}{\sum_{j=1}^N \omegaa(\overbar{V}^j_\tau)}
%%			\right)
%%			\right)
%%			I_M(\tau) \\
%		&= \left(
%			T_1
%			+
%			T_2
%			+
%			T_3
%			\right)
%			I_M(\tau),
%	\end{split}
%	\end{equation}
	For the norm of the difference between $\conspoint{\empmeasure{\tau}}$ and $\conspoint{\monopmeasure{\tau}}$ we have the decomposition
	\begin{equation} \label{app:eq:proof:lem:Wstability_conspoint:difference_norm_aux}
	\begin{split}
		\N{\conspoint{\empmeasure{\tau}}-\conspoint{\monopmeasure{\tau}}}_2 I_M(\tau)
		&= \N{
			\frac{\sum_{i=1}^N V^i_\tau \omegaa(V^i_\tau)}{\sum_{j=1}^N \omegaa(V^j_\tau)}
			-
			\frac{\sum_{i=1}^N \overbar{V}^i_\tau \omegaa(\overbar{V}^i_\tau)}{\sum_{j=1}^N \omegaa(\overbar{V}^j_\tau)}
			}_2
			I_M(\tau) \\
		&\leq \left(
			\N{T_1}_2
			+
			\N{T_2}_2
			+
			\N{T_3}_2
			\right)
			I_M(\tau),
	\end{split}
	\end{equation}
	where the terms $T_1$, $T_2$ and $T_3$ are obtained by inserting mixed terms with respect to $V^i_\tau$ and $\overbar{V}^i_\tau$.
	They are defined implicitly below and their norm is bounded as follows.
	For the first term~$T_1$ we have
%	\begin{equation*}
%	\begin{split}
%		T_1I_M(\tau)
%		&= \left(
%		\frac{\sum_{i=1}^N V^i_\tau \omegaa(V^i_\tau)}{\sum_{j=1}^N \omegaa(V^j_\tau)}
%		-
%		\frac{\sum_{i=1}^N \overbar{V}^i_\tau \omegaa(V^i_\tau)}{\sum_{j=1}^N \omegaa(V^j_\tau)}
%		\right)
%		I_M(\tau)\\
%		&=
%		\frac{1}{N}\sum_{i=1}^N\left(V^i_\tau-\overbar{V}^i_\tau\right)\frac{\omegaa(V^i_\tau)}{\frac{1}{N}\sum_{j=1}^N \omegaa(V^j_\tau)}
%		I_M(\tau)\\
%	\end{split}
%	\end{equation*}
	\begin{equation} \label{app:eq:proof:lem:Wstability_conspoint:T1}
	\begin{split}
		\N{T_1}_2I_M(\tau)
		&= \N{
		\frac{1}{N}\sum_{i=1}^N\left(V^i_\tau-\overbar{V}^i_\tau\right)\frac{\omegaa(V^i_\tau)}{\frac{1}{N}\sum_{j=1}^N \omegaa(V^j_\tau)}
		}_2I_M(\tau)\\
		&\leq 
		\frac{1}{N}\sum_{i=1}^N\N{V^i_\tau-\overbar{V}^i_\tau}_2\abs{\frac{\omegaa(V^i_\tau)}{\frac{1}{N}\sum_{j=1}^N \omegaa(V^j_\tau)}}
		I_M(\tau)\\
		&\leq 
		\abs{\frac{e^{-\alpha\underbarscript\CE}\,I_M(\tau)}{\frac{1}{N}\sum_{j=1}^N \omegaa(V^j_\tau)}} \, \frac{1}{N}\sum_{i=1}^N\N{V^i_\tau-\overbar{V}^i_\tau}_2I_M(\tau)\\
		&\leq c_M \sqrt{\frac{1}{N}\sum_{i=1}^N\N{V^i_\tau-\overbar{V}^i_\tau}_2^2I_M(\tau)},
	\end{split}
	\end{equation}
	where we made use of \eqref{app:eq:proof:lem:Wstability_conspoint:aux_statement} and Cauchy-Schwarz inequality in the last step.
	For the second term~$T_2$, by using the assumption~\eqref{eq:lipschitz_condition_car} on $\CE$ in the third line and by following similar steps, we obtain
%	\begin{equation*}
%	\begin{split}
%		T_2I_M(\tau)
%		&= \left(
%		\frac{\sum_{i=1}^N \overbar{V}^i_\tau \omegaa(V^i_\tau)}{\sum_{j=1}^N \omegaa(V^j_\tau)}
%		-
%		\frac{\sum_{i=1}^N \overbar{V}^i_\tau \omegaa(\overbar{V}^i_\tau)}{\sum_{j=1}^N \omegaa(V^j_\tau)}
%		\right)
%		I_M(\tau)\\
%		&=
%		\frac{1}{N}\sum_{i=1}^N\left(\omegaa(V^i_\tau)-\omegaa(\overbar{V}^i_\tau)\right)\frac{\overbar{V}^i_\tau}{\frac{1}{N}\sum_{j=1}^N \omegaa(V^j_\tau)}
%		I_M(\tau)\\
%	\end{split}
%	\end{equation*}
	\begin{equation} \label{app:eq:proof:lem:Wstability_conspoint:T2}
	\begin{split}
		\N{T_2}_2I_M(\tau)
		&= \N{
		\frac{1}{N}\sum_{i=1}^N\left(\omegaa(V^i_\tau)-\omegaa(\overbar{V}^i_\tau)\right)\frac{\overbar{V}^i_\tau}{\frac{1}{N}\sum_{j=1}^N \omegaa(V^j_\tau)}
		I_M(\tau)
		}_2I_M(\tau)\\
		&\leq
		\frac{1}{N}\sum_{i=1}^N \abs{\omegaa(V^i_\tau)-\omegaa(\overbar{V}^i_\tau)}\N{\frac{\overbar{V}^i_\tau}{\frac{1}{N}\sum_{j=1}^N \omegaa(V^j_\tau)}}_2
		I_M(\tau)\\
		&\leq
		\alpha C_1 e^{-\alpha\underbarscript\CE}  \frac{1}{N}\sum_{i=1}^N \left(\N{V^i_\tau}_2+\N{\overbar V^i_\tau}_2\right)\N{V^i_\tau-\overbar{V}^i_\tau}_2 \frac{\N{\overbar{V}^i_\tau}_2}{\frac{1}{N}\sum_{j=1}^N \omegaa(V^j_\tau)} I_M(\tau)\\
		%&\leq
		%\alpha C_1 \abs{\frac{e^{-\alpha\underbarscript\CE}I_M(\tau)}{\frac{1}{N}\sum_{j=1}^N \omegaa(V^j_\tau)}} \frac{1}{N}\sum_{i=1}^N \left(\N{V^i_\tau}_2\N{\overbar{V}^i_\tau}_2+\N{\overbar V^i_\tau}_2^2\right)\N{V^i_\tau-\overbar{V}^i_\tau}_2
		%I_M(\tau)\\
		%&\leq
		%\frac{3}{2}\alpha C_1 \abs{\frac{e^{-\alpha\underbarscript\CE}I_M(\tau)}{\frac{1}{N}\sum_{j=1}^N \omegaa(V^j_\tau)}} \frac{1}{N}\sum_{i=1}^N \left(\N{V^i_\tau}_2^2+\N{\overbar V^i_\tau}_2^2\right)\N{V^i_\tau-\overbar{V}^i_\tau}_2
		%I_M(\tau)\\
		&\leq
		\frac{3}{2}\alpha C_1 \abs{\frac{e^{-\alpha\underbarscript\CE}I_M(\tau)}{\frac{1}{N}\sum_{j=1}^N \omegaa(V^j_\tau)}} \sqrt{\frac{1}{N}\sum_{i=1}^N \left(\N{V^i_\tau}_2^4+\N{\overbar V^i_\tau}_2^4\right)I_M(\tau)}\\
		&\qquad\qquad\quad\cdot\sqrt{\frac{1}{N}\sum_{i=1}^N \N{V^i_\tau-\overbar{V}^i_\tau}_2^2I_M(\tau)}\\
		&\leq 3\alpha C_1c_MM^{\frac{1}{2}} \sqrt{\frac{1}{N}\sum_{i=1}^N \N{V^i_\tau-\overbar{V}^i_\tau}_2^2I_M(\tau)}.
	\end{split}
	\end{equation}
	Analogously, for the third term~$T_3$, we get
%	\begin{equation*}
%	\begin{split}
%		T_3I_M(\tau)
%		&= \left(
%		\frac{\sum_{i=1}^N \overbar{V}^i_\tau \omegaa(\overbar{V}^i_\tau)}{\sum_{j=1}^N \omegaa(V^j_\tau)}
%		-
%		\frac{\sum_{i=1}^N \overbar{V}^i_\tau \omegaa(\overbar{V}^i_\tau)}{\sum_{j=1}^N \omegaa(\overbar{V}^j_\tau)}
%		\right)
%		I_M(\tau)\\
%%		&=
%%		\left(\frac{1}{\sum_{j=1}^N \omegaa(V^j_\tau)}-\frac{1}{\sum_{j=1}^N \omegaa(\overbar{V}^j_\tau)}\right)\sum_{i=1}^N \overbar{V}^i_\tau \omegaa(\overbar{V}^i_\tau)
%%		I_M(\tau)\\
%		&=
%		\left(\sum_{j=1}^N \left(\omegaa(\overbar{V}^j_\tau)-\omegaa(V^j_\tau)\right)\right)\frac{\sum_{i=1}^N \overbar{V}^i_\tau \omegaa(\overbar{V}^i_\tau)}{\left(\sum_{j=1}^N \omegaa(V^j_\tau)\right)\left(\sum_{j=1}^N \omegaa(\overbar{V}^j_\tau)\right)}
%		I_M(\tau)\\
%	\end{split}
%	\end{equation*}
	\begin{equation} \label{app:eq:proof:lem:Wstability_conspoint:T3}
	\begin{split}
		\N{T_3}_2I_M(\tau)
		&= \N{
		\frac{\sum_{i=1}^N \overbar{V}^i_\tau \omegaa(\overbar{V}^i_\tau)}{\sum_{j=1}^N \omegaa(V^j_\tau)}
		-
		\frac{\sum_{i=1}^N \overbar{V}^i_\tau \omegaa(\overbar{V}^i_\tau)}{\sum_{j=1}^N \omegaa(\overbar{V}^j_\tau)}
		}_2I_M(\tau)\\
%		&=
%		\left(\frac{1}{\sum_{j=1}^N \omegaa(V^j_\tau)}-\frac{1}{\sum_{j=1}^N \omegaa(\overbar{V}^j_\tau)}\right)\sum_{i=1}^N \overbar{V}^i_\tau \omegaa(\overbar{V}^i_\tau)
%		I_M(\tau)\\
%		&=
%		\abs{\frac{1}{N}\sum_{j=1}^N \left(\omegaa(\overbar{V}^j_\tau)-\omegaa(V^j_\tau)\right)I_M(\tau)}
%		\N{\frac{\left(\frac{1}{N}\sum_{i=1}^N \overbar{V}^i_\tau \omegaa(\overbar{V}^i_\tau)\right)I_M(\tau)}{\left(\frac{1}{N}\sum_{j=1}^N \omegaa(V^j_\tau)\right)\left(\frac{1}{N}\sum_{j=1}^N \omegaa(\overbar{V}^j_\tau)\right)}
%		}_2\\
		&\leq
		\frac{1}{N}\sum_{j=1}^N \abs{\omegaa(\overbar{V}^j_\tau)-\omegaa(V^j_\tau)}
		\N{\frac{\frac{1}{N}\sum_{i=1}^N \overbar{V}^i_\tau \omegaa(\overbar{V}^i_\tau)}{\left(\frac{1}{N}\sum_{j=1}^N \omegaa(V^j_\tau)\right)\left(\frac{1}{N}\sum_{j=1}^N \omegaa(\overbar{V}^j_\tau)\right)}
		}_2I_M(\tau)\\
		&\leq
		\alpha C_1e^{-2\alpha\underbarscript\CE}\frac{1}{N}\sum_{j=1}^N \left(\N{V^j_\tau}_2+\N{\overbar V^j_\tau}_2\right)\N{V^j_\tau-\overbar{V}^j_\tau}_2I_M(\tau)\\
		&\qquad\qquad\quad\cdot\frac{\frac{1}{N}\sum_{i=1}^N \N{\overbar{V}^i_\tau}_2 I_M(\tau)}{\left(\frac{1}{N}\sum_{j=1}^N \omegaa(V^j_\tau)\right)\left(\frac{1}{N}\sum_{j=1}^N \omegaa(\overbar{V}^j_\tau)\right)}\\
		%&\leq
		%\alpha C_1c_M^2M^{\frac{1}{4}}\frac{1}{N}\sum_{i=1}^N \left(\N{V^i_\tau}_2+\N{\overbar V^i_\tau}_2\right)\N{V^i_\tau-\overbar{V}^i_\tau}_2I_M(\tau)\\
		&\leq
		\sqrt{2}\alpha C_1c_M^2M^{\frac{1}{4}}\sqrt{\frac{1}{N}\sum_{j=1}^N \left(\Nbig{V^j_\tau}_2^2+\Nbig{\overbar V^j_\tau}_2^2\right)I_M(\tau)} \sqrt{\frac{1}{N}\sum_{j=1}^N\Nbig{V^j_\tau-\overbar{V}^j_\tau}_2^2I_M(\tau)}\\
		&\leq
		2\alpha C_1c_M^2M^{\frac{1}{2}} \sqrt{\frac{1}{N}\sum_{j=1}^N\Nbig{V^j_\tau-\overbar{V}^j_\tau}_2^2I_M(\tau)}.
	\end{split}
	\end{equation}
	By inserting the three individual bounds~\eqref{app:eq:proof:lem:Wstability_conspoint:T1}, \eqref{app:eq:proof:lem:Wstability_conspoint:T2} and \eqref{app:eq:proof:lem:Wstability_conspoint:T3} into~\eqref{app:eq:proof:lem:Wstability_conspoint:difference_norm_aux} and taking the squares of both sides, we obtain the upper bound from the statement.
\end{proof}

\begin{lemma} \label{app:lem:sampling_conspoint}
	Let $\rho_0\in\CP_{2}(\bbR^d)$ and let $I_M$ be as defined in~\eqref{app:eq:proof:prop:MF:cutoff}.
	Then, under the assumptions of Theorem~\ref{thm:well-posedness_FP}, it holds
	\begin{equation}
		\sup_{\tau\in[0,T]} \bbE\N{\conspoint{\monopmeasure{\tau}}-\conspoint{\rho_\tau}}_2^2I_M(\tau)
		\leq CN^{-1}
	\end{equation}
	for a constant $C=C(\alpha,C_2,M,\CM_2,b_1,b_2)$, where $\CM_2$ denotes the second-order moment bound of~$\rho$ and where $b_1$ and $b_2$ are the problem-dependent constants specified in~\eqref{app:eq:proof:constantsb1b2}.
\end{lemma}

\begin{proof}
	The proof follows the steps taken in~\cite[Lemma~3.1]{fornasier2020consensus_hypersurface_wellposedness}.
	
	\noindent
	By inserting a mixed term, we can bound the norm of the difference between $\conspoint{\monopmeasure{\tau}}$ and $\conspoint{\rho_\tau}$ by
	\begin{equation} \label{app:eq:proof:lem:largedeviation_conspoint:difference_norm_aux}
	\begin{split}
		\N{\conspoint{\monopmeasure{\tau}}-\conspoint{\rho_\tau}}_2 I_M(\tau)
		&= \N{
			\sum_{i=1}^N \overbar{V}^i_\tau\frac{\omegaa(\overbar{V}^i_\tau)}{\sum_{j=1}^N \omegaa(\overbar{V}^j_\tau)}
			-
			\int v\frac{\omegaa(v)}{\N{\omegaa}_{L_1(\rho_\tau)}} \,d\rho_\tau(v)}_2
			I_M(\tau) \\
		&\leq \left(\N{T_1}_2 + \N{T_2}_2\right)I_M(\tau),
	\end{split}
	\end{equation}
	where the terms $T_1$ and $T_2$ are defined implicitly and bounded in what follows.
	By utilizing the bound~\eqref{app:eq:proof:lem:Wstability_conspoint:aux_statement}, for the first term~$T_1$, we get
	\begin{equation} \label{app:eq:proof:lem:largedeviation_conspoint:T1}
	\begin{split}
		\N{T_1}_2I_M(\tau)
		&=\N{\sum_{i=1}^N \overbar{V}^i_\tau\frac{\omegaa(\overbar{V}^i_\tau)}{\sum_{j=1}^N \omegaa(\overbar{V}^j_\tau)}
			-
			\int v\frac{\omegaa(v)}{\frac{1}{N}\sum_{j=1}^N \omegaa(\overbar{V}^j_\tau)} \,d\rho_\tau(v)}_2
			I_M(\tau) \\
		&=\abs{\frac{I_M(\tau)}{\frac{1}{N}\sum_{j=1}^N \omegaa(\overbar{V}^j_\tau)}}\N{\frac{1}{N}\sum_{i=1}^N \overbar{V}^i_\tau\omegaa(\overbar{V}^i_\tau)
			-
			\int v\omegaa(v) \,d\rho_\tau(v)}_2 \\
		&\leq c_Me^{\alpha\underbarscript\CE} \N{\frac{1}{N}\sum_{i=1}^N \overbar{V}^i_\tau\omegaa(\overbar{V}^i_\tau)
			-
			\int v\omegaa(v) \,d\rho_\tau(v)}_2.
	\end{split}
	\end{equation}
	Similarly, for the second term we have
	\begin{equation} \label{app:eq:proof:lem:largedeviation_conspoint:T2}
	\begin{split}
		\N{T_2}_2I_M(\tau)
		&= \N{\int v\frac{\omegaa(v)}{\frac{1}{N}\sum_{j=1}^N \omegaa(\overbar{V}^j_\tau)} \,d\rho_\tau(v)
			-
			\int v\frac{\omegaa(v)}{\N{\omegaa}_{L_1(\rho_\tau)}} \,d\rho_\tau(v)}_2
			I_M(\tau) \\
		&= \abs{\frac{I_M(\tau)}{\frac{1}{N}\sum_{j=1}^N \omegaa(\overbar{V}^j_\tau)}}\N{\conspoint{\rho_\tau}}_2
		\abs{\frac{1}{N}\sum_{j=1}^N \omegaa(\overbar{V}^j_\tau)-\N{\omegaa}_{L_1(\rho_\tau)}} \\
		&\leq c_Me^{\alpha\underbarscript\CE} \sqrt{b_1 + b_2\CM_2}
		\abs{\frac{1}{N}\sum_{j=1}^N \omegaa(\overbar{V}^j_\tau)-\int\omegaa(v) \, d\rho_\tau(v)}	, \\
	\end{split}
	\end{equation}
	where the last step uses that by Jensen's inequality and \cite[Lemma~3.3]{carrillo2018analytical} it holds
	\begin{equation*}
	\begin{split}
		\N{\conspoint{\rho_\tau}}_2^2
		&\leq \int \N{v}_2^2 \frac{\omegaa(v)}{\N{\omegaa}_{L_1(\rho_\tau)}}\,d\rho_\tau(v)\leq b_1 + b_2 \int \N{v}_2^2 d\rho_\tau(v)
		\leq b_1 + b_2\CM_2
	\end{split}
	\end{equation*}
	with constants $b_1$ and $b_2$ as specified in~\eqref{app:eq:proof:constantsb1b2} and $\CM_2$ denoting a bound on the second-order moment of $\rho$, which exists according to the regularity of $\rho$ established in Theorem~\ref{thm:well-posedness_FP} as a consequence of the initial regularity~$\rho_0\in\CP_2(\bbR^d)$.
	In order to further bound \eqref{app:eq:proof:lem:largedeviation_conspoint:T1} and \eqref{app:eq:proof:lem:largedeviation_conspoint:T2}, respectively, let us introduce the random variables
	\begin{equation*}
		Z^i_\tau = \overbar{V}^i_\tau\omegaa(\overbar{V}^i_\tau)-\int v\omegaa(v) \,d\rho_\tau(v)
		\quad\text{ and }\quad
		z^i_\tau = \omegaa(\overbar{V}^i_\tau)-\int \omegaa(v) \,d\rho_\tau(v),
	\end{equation*} 
	which have zero expectation, i.e., $\bbE Z^i_\tau=0$ and $\bbE z^i_\tau=0$.
	Moreover, we observe that
	\begin{equation*}
	\begin{split}
		\frac{1}{N}\sum_{i=1}^N \overbar{V}^i_\tau\omegaa(\overbar{V}^i_\tau)-\int v\omegaa(v) \,d\rho_\tau(v)
		= \frac{1}{N}\sum_{i=1}^N Z^i_\tau
	\end{split}
	\end{equation*}
	and
	\begin{equation*}
	\begin{split}
		\frac{1}{N}\sum_{i=1}^N \omegaa(\overbar{V}^j_\tau)-\int\omegaa(v) \, d\rho_\tau(v)
		= \frac{1}{N}\sum_{i=1}^N z^i_\tau,
	\end{split}
	\end{equation*}
	respectively.
	Moreover, due to the independence of the $\overbar{V}^i_\tau$'s the $Z^i_\tau$'s are independent and thus satisfy $\bbE Z^i_\tau Z^j_\tau = 0$ for $i\neq j$.
	Using this we can rewrite
	\begin{equation} \label{app:eq:proof:lem:largedeviation_conspoint:T1_2}
	\begin{split}
		\bbE\N{\frac{1}{N}\sum_{i=1}^N \overbar{V}^i_\tau\omegaa(\overbar{V}^i_\tau)-\int v\omegaa(v) \,d\rho_\tau(v)}_2^2 
		&= \bbE\N{\frac{1}{N}\sum_{i=1}^N Z^i_\tau}_2^2
		= \frac{1}{N^2}\bbE\sum_{i=1}^N\sum_{j=1}^N \langle Z^i_\tau, Z^j_\tau\rangle\\
		&= \frac{1}{N^2}\bbE\sum_{i=1}^N \N{Z^i_\tau}_2^2
		= \frac{1}{N}\bbE\N{Z^1_\tau}_2^2
		\leq 4e^{-\alpha\underbarscript\CE}\CM_2\frac{1}{N},
	\end{split}
	\end{equation}
	where the inequality in the last step is due to the estimate
	\begin{equation*} %\label{app:eq:proof:lem:largedeviation_conspoint:T1_2_aux}
	\begin{split}
		\bbE\N{Z^1_\tau}_2^2
		%&= \bbE\N{\overbar{V}^1_\tau\omegaa(\overbar{V}^1_\tau)-\int v\omegaa(v) \,d\rho_\tau(v)}_2^2\\
		&\leq 2\bbE\N{\overbar{V}^1_\tau\omegaa(\overbar{V}^1_\tau)}_2^2+2\N{\int v\omegaa(v) \,d\rho_\tau(v)}_2^2\\
		&\leq 2e^{-\alpha\underbarscript\CE}\left(\bbE\N{\overbar{V}^1_\tau}_2^2+\int \N{v}_2^2 \,d\rho_\tau(v)\right)
		\leq 4e^{-\alpha\underbarscript\CE}\CM_2.
	\end{split} 
	\end{equation*}
	Following analogous arguments and noting that
	\begin{equation*} \label{app:eq:proof:lem:largedeviation_conspoint:T2_2_aux}
	\begin{split}
		\bbE\abs{z^1_\tau}^2
		%&= \bbE\abs{\omegaa(\overbar{V}^j_\tau)-\int\omegaa(v) \, d\rho_\tau(v)}^2\\
		&\leq 2\bbE\abs{\omegaa(\overbar{V}^j_\tau)}^2 + 2\abs{\int\omegaa(v) \, d\rho_\tau(v)}^2
		\leq 4e^{-\alpha\underbarscript\CE}
	\end{split} 
	\end{equation*}
	yields the inequality
	\begin{equation} \label{app:eq:proof:lem:largedeviation_conspoint:T2_2}
	\begin{split}
		\bbE\abs{\frac{1}{N}\sum_{i=1}^N \omegaa(\overbar{V}^j_\tau)-\int\omegaa(v) \, d\rho_\tau(v)}^2
		= \frac{1}{N}\bbE\abs{z^1_\tau}^2
		\leq  4e^{-\alpha\underbarscript\CE} \frac{1}{N}.
	\end{split}
	\end{equation}
    Taking the square and expectation on both sides of \eqref{app:eq:proof:lem:largedeviation_conspoint:T1} and \eqref{app:eq:proof:lem:largedeviation_conspoint:T2}, inserting the two individual bounds~\eqref{app:eq:proof:lem:largedeviation_conspoint:T1_2} and \eqref{app:eq:proof:lem:largedeviation_conspoint:T2_2}, gives the statement after recalling \eqref{app:eq:proof:lem:largedeviation_conspoint:difference_norm_aux}.
\end{proof}

\end{document}